\documentclass[paper=a4, 12pt]{article} % A4 paper and 12pt font size

\usepackage[latin1]{inputenc} 
\usepackage[english]{babel}
\usepackage{geometry}
\usepackage{amsmath,amsfonts,amsthm,amssymb,mathtools} % Math packages
\usepackage{enumerate}
\usepackage{stmaryrd}
\usepackage{color}
\usepackage{fancyhdr}
\usepackage{multicol}
\usepackage{tikz-cd}
 	\usetikzlibrary{trees}
	\usetikzlibrary{arrows}
 	\usetikzlibrary{shapes,snakes}
	\usetikzlibrary{decorations.pathreplacing}
\usepackage[noadjust]{cite}
\usepackage[hyperfootnotes=false]{hyperref}
\usepackage{url}
\usepackage{appendix}
\usepackage{float}
\theoremstyle{theorem}
\newtheorem{theorem}{Theorem}[subsection]
\newtheorem{proposition}[theorem]{Proposition}

\newtheorem{lemma}[theorem]{Lemma}

\theoremstyle{definition}
\newtheorem{definition}[theorem]{Definition}
\newtheorem{remark}[theorem]{Remark}
\newtheorem{example}[theorem]{Example}
\newtheorem*{example*}{Example}
\newtheorem*{definition*}{Definition}

\numberwithin{equation}{subsection}

\makeatletter
\def\maketag@@@#1{\hbox{\m@th\normalfont\normalsize#1}}
\makeatother

\newcommand{\dcorner}{\rotatebox[origin=c]{-45}{$\lrcorner$}}
\newcommand{\ucorner}{\rotatebox[origin=c]{135}{$\lrcorner$}}

\newcommand{\ldcorner}{\rotatebox[origin=c]{-90}{$\lrcorner$}}

\newcommand{\dpbk}{\arrow[dd,phantom,"\dcorner", very near start]}

\newcommand{\upbks}{\arrow[u,phantom,"\ucorner", very near start]}
\newcommand{\dpbks}{\arrow[d,phantom,"\dcorner ", very near start]}

\newcommand{\rdpbk}{\arrow[rd,phantom,"\lrcorner ", very near start]}

\newcommand{\ldpbk}{\arrow[ld,phantom,"\ldcorner ", very near start]}

 \tikzset{
  /tikz/commutative diagrams/on top/.style={inner sep=0pt, description}
}

	\renewcommand{\l}{\lambda}
	\renewcommand{\tilde}{\widetilde}
	\renewcommand{\bar}{\overline}

	\newcommand{\N}{\mathbb N}%natural numbers
	\newcommand{\Q}{\mathbb Q}%rational numbers
	\newcommand{\pow}[1]{\llbracket #1\rrbracket}
	\newcommand{\ncpow}[1]{\langle\!\langle #1\rangle\!\rangle}
	\newcommand{\X}{\mathbf X}
	\newcommand{\y}{\mathbf y}
	\newcommand{\x}{\mathbf x}

	\newcommand{\termon}{Y_1}%terminal monoid
	\newcommand{\monoid}{Y}%generic monoid
	\newcommand{\ambcat}{\mathcal{E}}%ambient category
	\newcommand{\C}{\mathcal C}%generic category (not internal to E)
	\newcommand{\symgroup}{\mathfrak{S}}%symmetric group
	\newcommand{\set}{\mathbf{Set}}%category of sets
	\newcommand{\infgrpd}{\mathbf{Grpd}}%category of groupoids 
	\newcommand{\sur}{\mathbf{S}}%category of finite sets and surjections
	\newcommand{\Pmulticat}{\gm \textbf{-Operad}}%Category of P-operads
	\DeclareMathAlphabet{\mathbbe}{U}{bbold}{m}{n}%
	\newcommand{\simplexcategory}{\mathbbe{\Delta}} %

	\newcommand{\pop}{Q}%generic P-operad
	\newcommand{\comm}{\mathsf{Sym}}
	\newcommand{\ass}{\mathsf{Ass}}

	\newcommand{\tcons}{\mathcal{T}}%T-construction
	\newcommand{\tsbar}{\mathcal{B}}%Two sided bar construction
	\newcommand{\Tpc}[3]{\mathcal{T}^{#1}_{#2}{#3}}%T_1^2(3)

	\newcommand{\id}{\mathrm{id}}
	\newcommand{\comp}{m}%composition
	\newcommand{\name}[1]{\ulcorner #1\urcorner}

	\newcommand{\gm}{\mathsf{P}}%generic monad
	\newcommand{\fm}{\mathsf{M}}%free monoid monad
	\newcommand{\mC}{\mathsf{Q}}%P-operad monad
	\newcommand{\idm}{\mathsf{Id}}%identity monad
	\newcommand{\fsg}{\mathsf{M}^{^\text{r}}}%free semigroup monad
	\newcommand{\mm}{\mathsf{Y}}%mmonoid monad
	\newcommand{\fsmc}{\mathsf{S}}%free symmetric monoidal category monad
	\newcommand{\pfsmc}{\mathsf{S}^{^\text{r}}}%\fscm without empty sequence
	\newcommand{\Tx}{\mathsf{L}}%T1xC

	\newcommand{\FdB}{\mathcal{F}}%FdB
	\newcommand{\OFdB}{\mathcal{F}_{\operatorname{ord}}}%ordinary FdB
	\newcommand{\XFdB}{\mathcal{F}^{\operatorname{nc}}}%ordinary FdB
	\newcommand{\IIFdB}{\mathcal{F}^2}%2-variable  FdB
	\newcommand{\IIXFdB}{\mathcal{F}^{\langle2\rangle}}%2- noncommutative variable  FdB
	\newcommand{\XIIXFdB}{\mathcal{F}^{\langle2\rangle,\operatorname{nc}}}%2-noncomm variable noncomm FdB

	\newcommand{\plethB}{\mathcal{P}}%Plethystic bialgebra
	\newcommand{\EplethB}{\mathcal{P}_{\!\operatorname{exp}}}%Exponential Plethystic bialgebra
	\newcommand{\MplethB}{\mathcal{P}^{\monoid}}%Monoid M Plethystic bialgebra
	\newcommand{\plethBX}{\mathcal{P}^{\diamondsuit}}%Non-commutative exponential Plethystic bialgebra
	\newcommand{\LplethBX}{\mathcal{P}_{\!\operatorname{lin}}^{\diamondsuit}}%Linear Non-commutative exponential Plethystic bialgebra
	\newcommand{\XplethBX}{\mathcal{P}^{{\diamondsuit},\operatorname{nc}}}%Non-commutative Non-commutative  Plethystic bialgebra
	\newcommand{\LXplethBX}{\mathcal{P}^{{\diamondsuit},\operatorname{nc}}_{\!\operatorname{lin}}}%Non-comm Non-comm  linear Plethystic bialgebra

	\newcommand{\IIplethB}{\mathcal{P}^2}%2-variable  Plethystic bialgebra
	\newcommand{\EIIplethB}{\mathcal{P}^2_{\!\operatorname{exp}}}%2-variable exponential Plethystic bialgebra
	\newcommand{\lcx}[1]{\!{\underset{^{#1}}{\times}}\!}%long fiber product
	\newcommand{\cx}{{\times}}%short fiber product
	\newcommand{\grpdSprod}{\odot}
	
	\DeclareMathOperator{\autiv}{aut}
	\DeclareMathOperator{\aut}{Aut}
	\DeclareMathOperator{\rep}{R}
	\DeclareMathOperator{\iso}{Iso}
	\DeclareMathOperator{\fun}{Fun}

	\newcommand{\arxiv}[1]{\href{http://arxiv.org/pdf/#1}{arXiv:#1}} 

\begin{document}
 
%-----------------------------------------------------------------------------------------------------------------------------------%
%-------------------------------------------------------Title----------------------------------------------------------------------%
%-----------------------------------------------------------------------------------------------------------------------------------%
 
 \title{Plethysms and operads}
 \date{}
 \author{Alex Cebrian}
 \maketitle
 
 {\let\thefootnote\relax\footnote{This work has received support from grant number MTM2016-80439-P of Spain.}}
%-----------------------------------------------------------------------------------------------------------------------------------%
%-----------------------------------------------------Abstract--------------------------------------------------------------------%
%-----------------------------------------------------------------------------------------------------------------------------------%
 
\begin{abstract} We introduce the $\tcons$-construction, an endofunctor on the category of generalized operads, as a general mechanism by which
various notions of plethystic
substitution arise from more
ordinary notions of
substitution. In the special case of one-object unary
operads, i.e.~monoids, we recover the
$T$-construction of Giraudo.  We realize several kinds of plethysm as convolution products arising from the homotopy cardinality of the incidence bialgebra of the bar construction of various operads obtained from the $\tcons$-construction. The bar constructions are simplicial
groupoids, and in the special case of the
terminal reduced operad $\comm$, we recover
the simplicial groupoid of \cite{Cebrian}, a
combinatorial model for ordinary plethysm
in the sense of P\'olya, given in the spirit
of Waldhausen $S$ and Quillen $Q$
constructions. In some of the cases of the
$\tcons$-construction, an analogous
interpretation is possible.
\end{abstract}
  
%-----------------------------------------------------------------------------------------------------------------------------------%
%-------------------------------------------------Table of contents-----------------------------------------------------------%
%-----------------------------------------------------------------------------------------------------------------------------------%

  \tableofcontents
  
%-------------------------------------------------------------------------------------------------------------------------------------------
%-----------------------------------------------INTRODUCTION---------------------------------------------------------------------
%-----------------------------------------------INTRODUCTION---------------------------------------------------------------------
%-----------------------------------------------INTRODUCTION---------------------------------------------------------------------
%-----------------------------------------------INTRODUCTION---------------------------------------------------------------------
%-----------------------------------------------INTRODUCTION---------------------------------------------------------------------
%-----------------------------------------------INTRODUCTION---------------------------------------------------------------------
%-----------------------------------------------INTRODUCTION---------------------------------------------------------------------
%-------------------------------------------------------------------------------------------------------------------------------------------

\section*{Introduction}\label{seccion:introduction}
  \addcontentsline{toc}{section}{Introduction} 
  
  Plethysm is a substitution law in the ring of power series in infinitely many variables. It was introduced
by P\'olya  \cite{polya1937} in unlabelled enumeration theory in combinatorics, motivated as a series analogue of the wreath product of permutation groups. 
Another notion of  plethysm was defined by Littlewood  \cite{Littlewood305}  in the context of symmetric functions and representation theory of the general linear groups \cite{Macdonald}. It appears also in algebraic topology, in connection with $\l$-rings \cite{Borger-Wieland} and power operations in cohomology \cite{bauer}. The two notions of plethysm are closely related, as described in \cite{StanleyII} and \cite{Bergeron:2}. 

The plethysms we deal with emerge from P\'olya's notion, which we proceed to recall. 
Let $\Q\pow{\x}$ be the ring of power series in the infinite set of variables $\x=(x_1,x_2,\dots)$ without constant term. Given $F,G\in \Q\pow{\x}$, their \emph{plethystic substitution} is defined as
$$(G\oast F)(x_1,x_2,\dots)=G(F_1,F_2,\dots), \;\;\text{where}\;\; F_k(x_1,x_2,\dots)=F(x_k,x_{2k},\dots).$$
The formal power series may be expressed as 
$$F(\x)=\sum_{\l}F_{\l}\frac{\x^{\l}}{\autiv (\l)},$$
where $\autiv (\l)$ are certain symmetry factors (see Section \ref{section:plethysmsandoperads} below).

 It is well appreciated in combinatorics that bijective proofs give a deeper understanding than algebraic manipulation, as well as a better ground for further development. The so-called objective method, pioneered by Lawvere \cite{LawvereMenni}, Joyal \cite{Joyal:1981}, and Baez and Dolan \cite{Baez-Dolan}, provides a systematic approach to bijective proofs. One of the starting points of objective combinatorics is the theory of species, developed by Joyal~\cite{Joyal:1981} as a combinatorial theory of formal series. Within this context, P\'olya enumeration theory of unlabelled structures, including cycle index series and their plethysm, was entirely renewed.
 
However, a full combinatorial model of plethysm was only given a few years later by Nava and Rota \cite{Nava-Rota}.
They developed the notion of partitional, a functor from the groupoid of partitions to the category of finite sets, and showed that a suitable notion of composition of partitionals yields plethystic substitution of their generating functions, in analogy with  composition of species and composition of their exponential generating functions. 
A variation of this combinatorial interpretation was given shortly after by Bergeron  \cite{Bergeron}, who instead of partitionals considered  permutationals, functors from the groupoid of permutations to the category of finite sets. This approach is nicely related to the theory of species and their cycle index series through an adjunction. Later on, Nava \cite{Nava} studied both partitionals and permutationals from the point of view of incidence coalgebras, and added a third class of functors called linear partitionals. 

The bialgebras arising from the various plethystic substitutions are called plethystic bialgebras in the present article. Here is a general definition, as given in \cite{Cebrian}: the \emph{plethystic bialgebra} is the free polynomial algebra on the linear functionals $A_{\l}(F)=F_{\l}$ with comultiplication dual to plethystic substitution,
$$\Delta(A_{\l})(F, G)=A_{\l}(G\oast F).$$
In this definition, the difference between the three bialgebras of Nava depends on the definition of the linear functionals $A_{\l}$, which in turn depend on the definition of the symmetry factors $\autiv(\l)$.

The present work introduces a construction on operads, called the $\tcons$-construction, which
establishes a relationship between ordinary substitutions and plethystic substitutions.
In particular, this construction produces combinatorial models for the partitional and the linear partitional (also called exponential) cases, but also for other kinds of plethysm: plethysm of power series with variables indexed by a (locally finite) monoid, introduced by M\'endez and Nava \cite{Mendez-Nava} in the course of generalizing Joyal's theory of colored species to an arbitrary set of colors; plethysm in two variables $\x,\y$; plethysm of series with coefficients in a noncommutative ring, in the style of \cite{BFK:0406117}, and plethysm of series with noncommuting variables. All these plethysms and their bialgebras are explained in Section \ref{section:plethysmsandoperads}. The $\tcons$-construction relies on operads and the theory of decomposition spaces and their  incidence bialgebras.

The theory of operads has long been a standard tool in topology and
algebra~\cite{Markl-Schnider-Stasheff,Loday-Vallette}, and in
category theory~\cite{Leinster}, and it is getting increasingly
important also in combinatorics~\cite{Giraudo:book,Mendez}. In the present work, for maximal flexibility, we work with operads in the form of
generalized multicategories \cite{Leinster}. This allows us to cover simultaneously notions such
as monoids, categories, nonsymmetric operads and symmetric operads.
  
On the other hand, decomposition spaces (certain simplicial spaces) provide a general machinery to objectify the notion of incidence algebra in algebraic combinatorics. They were introduced by  G\'alvez, Kock and Tonks \cite{GKT:HLA,GKT:DSIAMI-1,GKT:DSIAMI-2} in this framework, and they are the same as $2$-Segal spaces introduced by Dyckerhoff and Kapranov \cite{DK} in the context of homological algebra and representation theory.  To recover the algebraic incidence coalgebra from the categorified incidence coalgebra one takes homotopy cardinality, a cardinality functor defined from groupoids to the rationals.
            
    It was shown in \cite{Kock-Weber} that the two-sided bar construction \cite{May,Weber:iac} of an operad is a Segal groupoid, a particular type of decomposition space, and classical constructions of bialgebras arising from operads factor through this construction (see \cite{Chapoton-Livernet:0707.3725,vanderLaan:math-ph/0311013,vanderLaan-Moerdijk:hep-th/0210226} for related constructions). Next we give two relevant  examples of bialgebras that arise as incidence bialgebras of operads. 
 
\begin{example*}Let $\Q\pow{x}$ be the ring of formal power series in $x$ without constant term, and let $F,G\in \Q\pow{x}$. The \emph{Fa\`a di Bruno bialgebra} $\mathcal{F}$ is the free algebra $\Q[A_1,A_2,\dots]$, where $A_n\in \Q\pow{x}^{\ast}$ is the linear map  defined by 
$$A_n(F)=\frac{d^nF}{dx^n}.$$
Its comultiplication is defined to be dual to substitution of power series. That is
$$\Delta(A_n)(F, G)=A_n(G\circ F).$$
\end{example*}
It is a result of Joyal  \cite[\S 7.4]{Joyal:1981} that this bialgebra can be objectified by using the category of finite sets and surjections $\sur$. In the context of Segal spaces and incidence bialgebras the result reads as follows: \emph{the Fa\`a di Bruno bialgebra $\FdB$ is isomorphic to the homotopy cardinality of the incidence bialgebra of the fat nerve $N\sur$ of the category $\sur$}. The comultiplication here is given by summing over factorizations
of surjections.

\begin{example*} In earlier work~\cite{Cebrian} the author found a
  simplicial groupoid $T\sur$ (like $N\sur$ arising from
  the category of surjections cf.~\ref{seccion:oldTrelation}
  below), which plays the same role for
  plethystic substitution: \emph{the homotopy cardinality of the incidence bialgebra of $T\sur$ is isomorphic to the (partitional) plethystic bialgebra.} The
comultiplication extracted from this simplicial groupoid can be
interpreted as summing over certain transversals of partitions, as in the work
of Nava and Rota~\cite{Nava-Rota}.
\end{example*}

 Now, it is well-known that $N\sur$ is equivalent to the two-sided bar construction of $\comm$, the terminal reduced symmetric operad. This equivalence takes the surjection $n\twoheadrightarrow 1$ to the unique $n$-ary operation, and the comultiplication of an operation runs through all possible $2$-step factorizations. For example
$$
\Delta\Big(\hspace{-2pt}
\raisebox{-13pt}{ \begin{tikzpicture}[grow=up,level distance=10pt]
\tikzstyle{level 1}=[sibling distance=5pt]
\node at (0,0){}
child {
child 
child
child
};
\end{tikzpicture}}\hspace{1pt}\Big)
=
\raisebox{-13pt}{ \begin{tikzpicture}[grow=up,level distance=10pt]
\tikzstyle{level 1}=[sibling distance=5pt]
\node at (0,0){}
child {
child 
child
child
};
\end{tikzpicture}}
\otimes 
\raisebox{-13pt}{ \begin{tikzpicture}[grow=up,level distance=10pt]
\tikzstyle{level 1}=[sibling distance=5pt]
\node at (0,0){}
child {
child 
};
\end{tikzpicture}}
+
3 \raisebox{-13pt}{\begin{tikzpicture}[grow=up,level distance=10pt]
\tikzstyle{level 1}=[sibling distance=5pt]
\node at (0,0){}
child {
child 
child
};
\end{tikzpicture}}
\raisebox{-13pt}{ \begin{tikzpicture}[grow=up,level distance=10pt]
\tikzstyle{level 1}=[sibling distance=5pt]
\node at (0,0){}
child {
child 
};
\end{tikzpicture}}
\otimes
\raisebox{-13pt}{\begin{tikzpicture}[grow=up,level distance=10pt]
\tikzstyle{level 1}=[sibling distance=5pt]
\node at (0,0){}
child {
child 
child
};
\end{tikzpicture}}
+
\raisebox{-13pt}{ \begin{tikzpicture}[grow=up,level distance=10pt]
\tikzstyle{level 1}=[sibling distance=5pt]
\node at (0,0){}
child {
child 
};
\end{tikzpicture}}
\raisebox{-13pt}{ \begin{tikzpicture}[grow=up,level distance=10pt]
\tikzstyle{level 1}=[sibling distance=5pt]
\node at (0,0){}
child {
child 
};
\end{tikzpicture}}
\raisebox{-13pt}{ \begin{tikzpicture}[grow=up,level distance=10pt]
\tikzstyle{level 1}=[sibling distance=5pt]
\node at (0,0){}
child {
child 
};
\end{tikzpicture}}
\otimes 
\raisebox{-13pt}{ \begin{tikzpicture}[grow=up,level distance=10pt]
\tikzstyle{level 1}=[sibling distance=5pt]
\node at (0,0){}
child {
child 
child
child
};
\end{tikzpicture}}.
$$

The starting point of the present work is the observation that also $T\sur$ is equivalent to the two-sided bar construction of an operad. As we shall see, this operad can be obtained from $\comm$ by  the aforementioned $\tcons$-construction, which makes sense for any (nice enough) operad. As stated above, this construction 
 leads to many other flavors of plethysm, some of which had already
been studied in various contexts. For instance, from $\ass$, the reduced associative  operad, we obtain the exponential plethystic bialgebra, and from $n$-colored $\comm$ or $\ass$,   we obtain the $n$-variables plethystic bialgebra. The results relating the bialgebras to these operads are explained in Section \ref{section:plethysmsandoperads}. 

Let us give a brief introduction to the $\tcons$-construction. The word $\tcons$-construction comes from the simplicial $T$-construction \cite{Cebrian}, where $T$
stands for transversal (in the sense of Nava--Rota \cite{Nava-Rota}), and which is
analogous to Waldhausen S and Quillen Q constructions. By coincidence Giraudo \cite{Giraudo} had used the
same letter $T$ for a functor from monoids to nonsymmetric operads. The $\tcons$-construction of
the present work encompasses both these constructions, and the letter $T$
has been maintained, but now in a fancier font.

Let us first describe Giraudo's $T$-construction \cite{Giraudo}.  Let $(\monoid,\cdot,1)$ be a monoid. Then 
$$T \monoid:=\bigsqcup_{n\ge0} \fm \monoid(n),$$
where for all $n\ge 1$,
$$T \monoid(n):=\{(x_1,\dots,x_n)\, |\, x_i\in \monoid\;\; \text{for all } i=1,\dots,n\},$$
so that the $n$-ary operations are $n$-tuples of elements of $\monoid$. The substitution law in $T \monoid$,
$$\circ_i:T \monoid(n)\times T \monoid(m)\longrightarrow T \monoid(n+m-1),$$
is defined as follows: for all $x\in T \monoid(n)$, $y\in T \monoid(m)$, and $i=1,\dots, n$,
$$x\circ_i y:=(x_1,\dots,x_{i-1},x_i\cdot y_1,\dots,x_i\cdot y_m,x_{i+1},\dots,x_n).$$

Our $\tcons$-construction is developed in the context of $\gm$-operads (generalized multicategories in Leinster \cite{Leinster} terminology), for $\gm$ a cartesian monad on a cartesian category $\ambcat$.  This level of generality allows us to work with symmetric, nonsymmetric, colored and noncolored operads on the same footing. This includes also monoids and categories. A $\gm$-operad is represented by a span and two arrows
 %---D-------I------A-------G-----R---------A-------M-------
\begin{center}
\begin{tikzcd}[column sep=small]
       & \pop_1\arrow[dl,"s"']\arrow[dr,"t"] &       \\
 \gm \pop_0 &                                            & \pop_0
 \end{tikzcd}
 \begin{tikzcd}
\gm \pop_1\cx_{\gm \pop_0} \pop_1 \arrow[r,"\comp"]& \pop_1\\
 \phantom{aaaaa}\pop_0 \arrow[r,"e"] &\pop_1,
 \end{tikzcd}
 \end{center}
%--------------------------------------------------------------
where $\pop_0$ is thought of as the object of colors, $\pop_1$ is thought of as the object of operations, $s$ returns the $\gm$-configuration of input colors, $t$ returns the output color, $e$ is the unit and $\comp$ is composition. All these arrows have to satisfy associativity and unit axioms. 

For instance, if $\idm$ is the identity monad, then an $\idm$-operad is a category internal to $\ambcat$. The $\tcons$-construction is in fact a composition of two constructions, one from $\gm$-operads to (internal) categories and one from categories to $\gm$-operads. The latter contains the Giraudo $T$-construction for the case $\ambcat=\set$ if we consider monoids as categories with one object. 

However, we will mainly be interested in $\ambcat=\infgrpd$. In particular, nonsymmetric operads will be considered as $\fsg$-operads, where $\fsg$ is the free semimonoidal category monad in $\infgrpd$, and symmetric operads as $\pfsmc$-operads, where $\pfsmc$ is the free symmetric semimonoidal category monad in $\infgrpd$. There are two main reasons for working over $\infgrpd$: on the one hand, note that unlike nonsymmetric operads, symmetric operads cannot be portrayed as $\gm$-operads in $\set$, because the free commutative monoid monad is not cartesian; on the other hand, working in $\infgrpd$ adapts better with the theory of decomposition spaces and incidence coalgebras. This theory uses weak notions of simplicial groupoids, slice categories, and pullbacks, but by keeping track of fibrancy we can stay within strict notions and strict monads in the style of \cite{Weber:opm}.

In order for the $\tcons$-construction to work, it is necessary to assume that the monads
come equipped with a strength. This notion goes back to work of A. Kock \cite{AndersKock} in enriched
category theory, but it has turned out to be fundamental for the role monads play in
functional programming \cite{Moggi,Wadler}.
 In Section \ref{seccion:MonadsMulticategories} we recall the theory of generalized operads, the notion of strong monad, and the two-sided bar construction in this context. 
 
 In Section \ref{seccion:HoCard} we  briefly explain Segal groupoids, incidence coalgebras and homotopy cardinality. Section \ref{seccion:Tcons} is devoted to the $\tcons$-construction, and Section \ref{seccion:examples} to some examples. Next, in Section \ref{section:plethysmsandoperads} we  introduce the bialgebras and state and prove the main results: the equivalence between the homotopy cardinality of the incidence bialgebras  of the two-sided bar constructions of operads obtained from the $\tcons$-construction and the plethystic bialgebras, as well as the Fa\`a di Bruno bialgebra and some of its variations. Finally, in Section \ref{seccion:oldTrelation} we prove the equivalence between $T\sur$ and $\tsbar \tilde{\comm}$, and we characterize some of the two-sided bar constructions as simplicial groupoids similar to $T\sur$.

 \subsection*{List of notations}

 \begin{tabular}{r l}
  $\infgrpd$& category of groupoids and groupoid morphisms\\
 $\set$ & category of sets and set maps\\
 $(\gm,\mu,\eta)$&generic strong cartesian monad (\ref{def:cartesian} and \ref{def:strength})\\
 $\idm$& identity monad (\ref{ex:identitymonad})\\
 $\fm$& free monoid monad (page \pageref{fm})\\
 $\fsg$& free semigroup monad (\ref{ex:fsg})\\
 $\fsmc$& free  symmetric monoidal category monad (\ref{ex:fsmc})\\
 $\pfsmc$&free symmetric semimonoidal category monad (\ref{ex:pfsmc})\\
 $\Tx$& monad $A\mapsto \gm 1\cx A$ (page \pageref{eq:TxD})\\
 $\monoid$ & generic (locally finite) monoid (\ref{ex:mm})\\
 %$\termon$ & terminal monoid, Example \ref{ex:TassTsym1}\\
 $\mm$ & monad given by $A\mapsto \monoid \cx A$, for $\monoid$ a monoid  (\ref{ex:mm})\\
 $T\monoid$ & Giraudo $T$-construction of $\monoid$ (\ref{ex:giraudo})\\
 $D_{A,B}$& strength natural transformation (\ref{def:strength})\\
 $D_B$ & strength for $A=1$ (page \pageref{eq:TxD})\\
 $R_A$ & projection $\gm 1\cx A\mapsto A$ (page \pageref{eq:TxD})\\
 $\tsbar$ &two-sided bar construction (page \pageref{tsbar})\\
 $\tsbar^{\gm}$ & two-sided bar construction relative to a monad $\gm$ (page \pageref{tsbargm})\\
 $\tsbar_n$ & $n$-simplices of the two-sided bar construction $\tsbar$ (page \pageref{tsbarnsimplices}) \\
 $\tsbar^{^{\circ}}_n$& subgroupoid of connected objects of $\tsbar_n$ (page \pageref{tsbarconnectednsimplices})\\
 $\comm$ & the reduced symmetric  operad (\ref{ex:pfsmc})\\
 $\ass$ & the reduced associative  operad (\ref{ex:fsg})\\
 $\ambcat$ & generic ambient cartesian category, mainly $\set$ or $\infgrpd$) (\ref{ambcat})\\
 $C$ & category internal to $\ambcat$ (page \pageref{C})\\
 $\pop$ &$\gm$-operad internal to $\ambcat$ (page \pageref{dgrm:MCat})\\
 $\pop_0,\pop_1$& objects and operations of $\pop$ (page \pageref{dgrm:MCat})\\
% $\gm C_0\xleftarrow{s/d_1}C_1\xrightarrow{t/d_0}C_0$ & span defining $C$\\
 $(\mC,\mu^{\pop},\eta^{\pop})$ & monad on $\ambcat/\pop_0$ defined by the $\gm$-operad $\pop$ (page \pageref{mC})\\
 $\Tpc{}{\gm}{C}$ & $\tcons$-construction from $C$ to a $\gm$-operad (page \pageref{TpcCQ})\\
 $\Tpc{\gm}{}{\pop}$ & $\tcons$-construction from a $\gm$-operad to a category $C$ (page \pageref{TpcQC})\\
 $\Tpc{}{\gm}{\pop}$ & $\tcons$-construction from a $\gm'$-operad to a $\gm$-operad (page \pageref{TpcQQ})\\
  $\Lambda$& set of infinite vectors $\l=(\l_1,\l_2,\dots)$ of natural\\[-5pt]
                 & numbers with $\l_i=0$ for all $i$ large enough (page \pageref{plethysticnotation})\\
 $\simplexcategory$ & simplex category (page \pageref{simplexcategory})\\
 $T\sur$ & simplicial groupoid of \cite{Cebrian} (page \pageref{TS})\\
 \end{tabular}

\section{Monads and operads}\label{seccion:MonadsMulticategories}
 
 As mentioned in the introduction, the $\tcons$-construction fits neatly within the context of generalized operads and strong monads. The following discussion of generalized operads is taken from \cite{Leinster}. Let us start by expressing the notions of category and of plain operad in this setting. 

 A small category $C$\label{C} can be described by sets and functions
 %---D-------I------A-------G-------R---------A-------M--------
\begin{center} 
 \begin{tikzcd}[column sep=small]
       & C_1\arrow[dl,"s"']\arrow[dr,"t"] &       \\
 C_0 &                                            & C_0
 \end{tikzcd}
 \begin{tikzcd}
 C_1\cx_{C_0} C_1 \arrow[r,"\comp"]& C_1\\
 \phantom{aaaaa}C_0 \arrow[r,"e"] &C_1
 \end{tikzcd}
 \end{center}
 %--------------------------------------------------------------
 where the pullback is taken along $C_1\xrightarrow{s}C_0\xleftarrow{t}C_1$,
 satisfying associativity and identity axioms, which can be expressed with commutative diagrams in $\set$ (see Appendix \ref{Cat_Axioms}). The set $C_0$ is the set of objects and $C_1$ is the set of arrows of $C$. The map $s$ returns the source of an arrow and $t$ returns its target. The maps $\comp$ and $e$ represent composition and identities. 
 
 A nonsymmetric operad can be defined in a similar way. Let $\fm{\colon}\set\rightarrow \set$ \label{fm} be the free monoid monad: it sends a set $A$ to $\bigsqcup_{n\in \N}A^n$ (see Example~\ref{example:plainoperadoperations} below). Then an operad can be described as consisting of sets and functions
%---D-------I------A-------G-------R---------A-------M--------
 \begin{equation} \label{dgrm:MCat}
 \begin{tikzcd}[column sep=small]
       & \pop_1\arrow[dl,"s"']\arrow[dr,"t"] &       \\
 \fm \pop_0 &                                            & \pop_0
 \end{tikzcd}
 \begin{tikzcd}
\fm \pop_1\cx_{\fm \pop_0} \pop_1 \arrow[r,"\comp"]& \pop_1\\
 \phantom{aaaaa}\pop_0 \arrow[r,"e"] &\pop_1
 \end{tikzcd}
  \end{equation}
  %---------------------------------------------------------------
  satisfying associativity and identity axioms, which can be expressed with commutative diagrams in $\set$ (see Appendix \ref{PCat_Axioms}) and involve the monad structure on $\fm$. The set $\pop_0$ is the set of objects and $\pop_1$ is the set of operations of $\pop$. The map $s$ assigns to an operation the sequence of objects constituting its source,  and $t$ returns its target. The maps $\comp$ and $e$ represent composition and identities.
  
Operations of classical operads, such as nonsymmetric operads or symmetric operads are pictured as
 
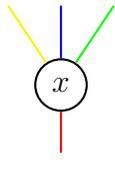
\begin{center}
 \begin{tikzpicture}[grow=up,level distance=30pt,thick]
 \tikzstyle{level 1}=[sibling distance=20pt]
\node at (0,0){}
child[red] {node[black,circle,draw]{$x$}}{
child[green]
child[blue]
child[yellow]
};
\end{tikzpicture},
 \end{center}
  where the colors are objects of $\pop_0$.

 \subsection{$\gm$-operads}
The above characterization of nonsymmetric $\set$ operads can be generalized to any ambient category  and any monad $\gm$ as long as they are cartesian. The classical case is $\set$; we shall be concerned also with $\infgrpd$.
\begin{definition} \label{def:cartesian} A category is \emph{cartesian} if it has all pullbacks. A functor is \emph{cartesian} if it preserves pullbacks. A natural transformation is \emph{cartesian} if all its naturality squares are pullbacks. A monad $(\gm,\mu,\eta)$ is \emph{cartesian} if $\gm$ is cartesian as a functor and $\mu$ and $\eta$ are cartesian natural transformations.
\end{definition}
Given a cartesian category $\ambcat$ \label{ambcat} and a cartesian monad $(\gm,\mu,\eta)$, we define  $\ambcat_{(\gm)}$ as the bicategory whose $0$-cells are the objects $E$ of $\ambcat$, whose $1$-cells $E\rightarrow E'$ are spans $\gm E\leftarrow M \rightarrow E'$, and $2$-cells are the usual morphisms $M\rightarrow N$ between spans:
 %---D-------I------A-------G-----R---------A-------M-------
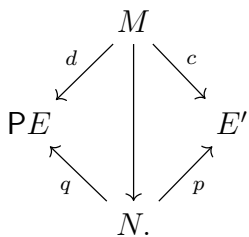
\begin{center} 
\begin{tikzcd}[column sep=small]
       &M\arrow[dl,"d"']\arrow[dr,"c"] \arrow[dd]&       \\
\gm E &                                            & E'\\
& N.\arrow[ul,"q"]\arrow[ur,"p"']&
 \end{tikzcd}
 \end{center}
%--------------------------------------------------------------
Given two $1$-cells
 %---D-------I------A-------G-----R---------A-------M-------
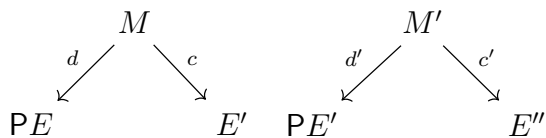
\begin{center}  
 \begin{tikzcd}[column sep=small]
       &M\arrow[dl,"d"']\arrow[dr,"c"] &       \\
\gm E &                                            & E'\\
 \end{tikzcd}
  \begin{tikzcd}[column sep=small]
       &M'\arrow[dl,"d'"']\arrow[dr,"c'"] &       \\
\gm E' &                                            & E''\\
 \end{tikzcd}
\end{center}
%--------------------------------------------------------------
the composite is given by taking a pullback and using the multiplication $\mu$ of $\gm$, and the $1$-cell identity is given by $\eta$ and $\id$. They are shown in the following diagram:
%---D-------I------A-------G-----R---------A-------M-------
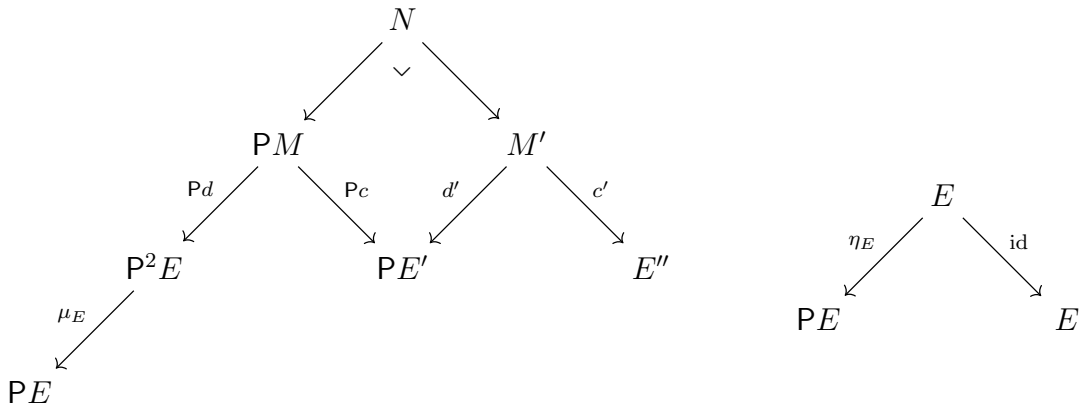
\begin{center}
\begin{tikzcd}[sep={4em,between origins}]
 &&& N \ar[dl] \ar[dr] \dpbk &&\\
 && \gm M\ar[dl,"\gm d"']\ar[dr,"\gm c"] & & M' \ar[dl,"d'"'] \ar[dr,"c'"]& \\
& \gm^2E \arrow[dl,"\mu_E"']&& \gm E' && E''\\
\gm E&&&&&
 \end{tikzcd}
 \begin{tikzcd}[sep={4em,between origins}]
 &&&\\
& &&\\
& & E \ar[dl,"\eta_E"'] \ar[dr,"\id"]&\\
&\gm E& & E\\
 \end{tikzcd}
\end{center}
%--------------------------------------------------------------
Composition and identity of $2$-cells are obvious. Since composition assumes a global choice of pullbacks, and since the
pasting of two chosen pullbacks is not generally a chosen pullback, composition is associative up to coherent isomorphism. The coherence $2$-cells are defined using the universal property of the pullback. 
\begin{definition}[Burroni \cite{burroni}] Let $\gm$ be a cartesian monad in a cartesian category $\ambcat$. A $\gm$-\emph{operad} is a monad in the bicategory $\ambcat_{(\gm)}$.
\end{definition}
This means that a $\gm$-operad $\pop$ consists precisely of objects $\pop_0$ and $\pop_1$ of $\ambcat$ together with maps $s$, $t$, composition $\comp$ and identities $e$ as in Diagram (\ref{dgrm:MCat}) satisfying associativity and identity axioms (Appendix \ref{PCat_Axioms}). A \emph{morphism} $\pop\rightarrow \pop'$ of $\gm$-operads is defined as a pair of arrows $\pop_0\xrightarrow{f_0}\pop'_0$, $\pop_1\xrightarrow{f_1}\pop'_1$, satisfying the following diagrams,
%---D-------I------A-------G-----R---------A-------M-------
\begin{equation} \label{eq:Poperadsmorphism}
\begin{tikzcd}[column sep=12pt, row sep=7pt]
       & \pop_1\arrow[dl]\arrow[dr] \arrow[dd,"f_1"]&       \\
 \gm \pop_0\arrow[dd,"\gm f_0"'] &                                            & \pop_0\arrow[dd,"f_0"]\\
        & \pop'_1\arrow[dl]\arrow[dr] &       \\
 \gm \pop'_0 &                                            & \pop'_0,
 \end{tikzcd}
 \hspace{20pt}
 \begin{tikzcd}[row sep=30pt]
 \pop_0\arrow[r,"e"]\arrow[d,"f_0"']& \pop_1\arrow[d,"f_1"]\\
 \pop'_0\arrow[r,"e"']&\pop'_1,
 \end{tikzcd}
  \hspace{20pt}
  \begin{tikzcd}[row sep=30pt]
 \gm \pop_1\cx_{\gm \pop_0}\pop_1\arrow[r,"\comp"]\arrow[d,"\gm f_1\cx_{\gm f_0}f_1"']& \pop_1\arrow[d,"f_1"]\\
 \gm \pop'_1\cx_{\gm \pop'_0}\pop'_1\arrow[r,"\comp"']&\pop'_1,
 \end{tikzcd}
\end{equation}
%--------------------------------------------------------------
 regarding compatibility with the spans, identities and composition maps. Notice that this is not an arrow in $\ambcat_{(\gm)}$. The category of $\gm$-operads is denoted $\Pmulticat$.
 
\subsection{Morphisms of spans}
In Section \ref{seccion:Tcons} we will deal with morphisms between long horizontal composites of spans. It is thus worth to set up a framework for such morphisms: consider the following diagrams, named blocks, made of maps in $\ambcat$,
%---D-------I------A-------G-----R---------A-------M-------
\begin{equation} \label{dgm:calculusid}
\begin{tikzcd}
\cdot \arrow[d]&\cdot\arrow[l]\arrow[r]\arrow[d]&\cdot\arrow[d]\\
\cdot&\cdot\arrow[l]\arrow[r]&\cdot
\end{tikzcd}
 \end{equation}
 \begin{equation}\label{dgm:calculusd1}
 \begin{tikzcd}
\cdot \arrow[d]&\cdot\arrow[l]\arrow[r]&\cdot&\cdot\arrow[l]\arrow[r]&\cdot\arrow[d]\\
\cdot&&\cdot\arrow[ul]\arrow[ur]\arrow[ll]\arrow[rr]\upbks&&\cdot 
\end{tikzcd}
\phantom{aaaa}
 \begin{tikzcd}
 \cdot\arrow[d]&&\cdot\arrow[dl]\arrow[dr]\arrow[ll]\arrow[rr]\dpbks&&\cdot \arrow[d]\\
\cdot &\cdot\arrow[l]\arrow[r]&\cdot&\cdot\arrow[l]\arrow[r]&\cdot
\end{tikzcd}
\end{equation}
 \begin{equation}\label{dgm:calculusprojections}
 \begin{tikzcd}
\cdot&\cdot\arrow[l]\arrow[r]&\cdot\arrow[d]&\cdot\arrow[l]\arrow[r]\arrow[d]&\cdot\arrow[d]\\
&&\cdot&\cdot\arrow[l]\arrow[r]&\cdot
\end{tikzcd}
\phantom{aaaa}
 \begin{tikzcd}
\cdot\arrow[d]&\cdot\arrow[l]\arrow[r]\arrow[d]&\cdot\arrow[d]&\cdot\arrow[l]\arrow[r]&\cdot\\
\cdot&\cdot\arrow[l]\arrow[r]&\cdot&&
\end{tikzcd}
\end{equation}
%--------------------------------------------------------------
Notice that (\ref{dgm:calculusd1}) induce isomorphisms of spans if the vertical maps are isomorphisms, since in this case they represent horizontal composition of spans. Diagram (\ref{dgm:calculusid}) is an isomorphism when all the vertical arrows are isomorphisms, and (\ref{dgm:calculusprojections}) are isomorphisms when all the vertical arrows and the span projected away are isomorphisms.
Besides, the blocks can be horizontally and vertically attached in the obvious way to get morphisms of longer spans, with the only restriction that the diagrams (\ref{dgm:calculusprojections}) can be attached to the right and to the left respectively.
 \begin{lemma} Any pasting of blocks defines a morphism between the limit of the top row and the limit of the bottom row. Moreover, such a morphism is an isomorphism if it can be constructed from blocks that are isomorphisms.
\end{lemma}
The morphisms between long spans are pictured with diagrams
%---D-------I------A-------G-----R---------A-------M-------
\begin{center}
\begin{tikzcd}
\pmb{\cdot\arrow[d]}&\cdot \arrow[d,white, "\vdots" black]&\cdot\arrow[l] &\dots &\cdot\arrow[r]&\cdot\arrow[d,white, "\vdots" black]\\
\pmb{\cdot}&\cdot &\cdot\arrow[l] &\dots &\cdot\arrow[r]&\cdot
\end{tikzcd}
\end{center}
%--------------------------------------------------------------
where the left bold part is the limit of the diagram: the upper dot is the limit of the upper row, and same for the bottom row.
Observe that the decomposition of a morphism into blocks is not unique, and there may be decompositions of isomorphisms whose blocks are not necessarily isomorphisms. Here is an example that will be used later on.
\begin{example}The following diagram represents an isomorphism of composites of spans:
%---D-------I------A-------G-----R---------A-------M-------
 \begin{equation}\label{dgm:calculusisopbk}
 \begin{tikzcd}
\cdot\arrow[d,equal]&\cdot\arrow[l,"a"']\arrow[r,"b"]\arrow[d,"f"']\rdpbk&\cdot\arrow[d,"g"]&\cdot\arrow[l,"c"']\arrow[r]\arrow[d,equal]&\cdot\arrow[d,equal]\\
\cdot&\cdot\arrow[l,"a'"]\arrow[r,"b'"']&\cdot&\cdot\arrow[l,"c'"]\arrow[r]&\cdot
\end{tikzcd}
\end{equation}
%--------------------------------------------------------------
Indeed, it can be expressed by pasting isomorphism blocks:
%---D-------I------A-------G-----R---------A-------M-------
 \begin{equation}\label{dgm:calculusisopbkproof}
 \begin{tikzcd}
\cdot\arrow[d,equal]&&\cdot\arrow[ll,"a"']\arrow[rr,"b"]\arrow[dl,"f"']\arrow[dr,"b"]\dpbks&&\cdot\arrow[d,equal]&\cdot\arrow[l,"c"']\arrow[r]\arrow[d,equal]&\cdot\arrow[d,equal]\\
\cdot\arrow[d,equal]&\cdot\arrow[l,"a'"']\arrow[r,"b'"]\arrow[d,equal]&\cdot\arrow[d,equal]&\cdot\arrow[l,"g"']\arrow[r,equal]&\cdot&\cdot\arrow[l,"c"']\arrow[r]&\cdot\arrow[d,equal]\\
\cdot&\cdot\arrow[l,"a'"]\arrow[r,"b'"']&\cdot&&\cdot\arrow[ll,"c'"]\arrow[rr]\arrow[ul,"c"]\arrow[ur,equal]\upbks&&\cdot
\end{tikzcd}
\end{equation}
%--------------------------------------------------------------
\end{example}

\subsection{Strong monads}
 We now recall the notion of strong monad \cite{AndersKock}, which is central in the $\tcons$-construction. From now on the ambient category $\ambcat$ is required to have a terminal object, hence all finite limits.
  
\begin{definition}\label{def:strength}Let  $(\gm,\mu,\eta)$ be a monad on $\ambcat$. A \emph{strength} for $\gm$ is a natural transformationwith components $D_{A,B}\colon A\cx \gm B\rightarrow \gm(A\cx B)$, satisfying the following two axioms concerning tensoring with $1$ and consecutive applications of $D$,
 %---D-------I------A-------G-----R---------A-------M-------
 \begin{subequations}
 \begin{equation}\label{dgm:StrengthWith1_Axiom}
 \begin{tikzcd}[column sep=4em, row sep=2em]
 1{\times}\gm A\arrow[r,"D_{1,A}"]\arrow[dr,"p_2"']& \gm (1{\times}A)\arrow[d,"\gm p_2"]\\
                                                                  &    \gm A
 \end{tikzcd}
\end{equation}
\begin{equation}\label{dgm:StrengthConsecutiveApplications_Axiom}
\begin{tikzcd}[column sep=5em, row sep=2em]
 (A{\times}B){\times}\gm C\arrow[r]\arrow[d,"D_{A{\times}B,C}"'] &   A{\times}(B{\times }\gm C)\arrow[r,"A{\times}D_{B,C}"]  &A{\times}\gm (B{\times}C)     \arrow[d,"D_{A,B{\times}C}"]    \\
      \gm ((A{\times}B){\times} C)\arrow[rr]    &                                              & \gm (A{\times}(B{\times} C))
\end{tikzcd}
\end{equation}
\end{subequations}
and two axioms concerning compatibility with monad unit and multiplication,
 \begin{subequations}
 \begin{equation}\label{dgm:StrongMonad_AxiomEta}
\begin{tikzcd}[column sep=6em, row sep=3em]
A{\times }B \arrow[rd, "\eta_{A{\times}B}"'] \arrow[r, "A{\times }\eta_B"] & A{\times }\gm B \arrow[d,"{D_{A,B}}"]    \\
                                                                    &            \gm(A{\times }B)  
\end{tikzcd}
\end{equation}
 \begin{equation}\label{dgm:StrongMonad_AxiomMu}
\begin{tikzcd}[column sep=6em, row sep=3em]
 A{\times }\gm^2B \arrow[d, "A{\times }\mu_B"'] \arrow[r,"D_{A,\gm B}"] &  \gm(A{\times }\gm B) \arrow[r, "{\gm D_{A,B}}"] & \gm^2(A{\times }B) \arrow[d,"\mu_{A{\times }B}"] \\
A{\times }\gm B \arrow[rr,"{D_{A,B}}"']&& \gm(A{\times }B)
\end{tikzcd}
\end{equation}
\end{subequations}
%--------------------------------------------------------------
\end{definition}

 Before seeing some examples of $\gm$-operads and strong monads, we prove the following lemma, which will be useful in Section \ref{seccion:Tcons}.
 \begin{lemma} \label{lemma:upbk} Let $u$ be the unique morphism $u\colon \gm 1\rightarrow 1$. Then the square
  %---D-------I------A-------G-----R---------A-------M-------
  \begin{equation}\label{dgm:StrongMonadTupullback_Lemma}
 \begin{tikzcd}[column sep=4em, row sep=3em]
 A{\times}\gm ^21\arrow[r,"A{\times}\gm u"]\arrow[d,"D_{A,\gm 1}"']\rdpbk  &   A{\times}\gm 1 \arrow[d,"D_{A,1}"] \\
 \gm (A{\times}\gm 1)\arrow[r,"\gm (A{\times}u)"']                                     &   \gm (A{\times}1)
 \end{tikzcd}
 \end{equation}
 %--------------------------------------------------------------
 is a pullback.
 \begin{proof} Observe that if we project the bottom rows of this square to the first component,
  %---D-------I------A-------G-----R---------A-------M-------
  \begin{center}
 \begin{tikzcd}[column sep=4em, row sep=3em]
 A{\times}\gm ^21\arrow[r,"A{\times}\gm u"]\arrow[d,"D_{A,\gm 1}"']\rdpbk  &   A{\times}\gm 1 \arrow[d,"D_{A,1}"] \\
 \gm (A{\times}\gm 1)\arrow[r,"\gm (A{\times}u)"]        \arrow[d,"\gm p_2"']      &   \gm (A{\times}1) \arrow[d,"\gm p_2"]   \\
 \gm ^21   \arrow[r,"\gm u"']                                                              &   \gm 1
 \end{tikzcd}
 \end{center}
 %--------------------------------------------------------------
 then the lower square is a pullback because $\gm$ is cartesian, and the outer square is a  pullback because it is a projection, by (\ref{dgm:StrengthWith1_Axiom}). Therefore the upper square is a pullback too.
 \end{proof}
 \end{lemma}
Let us see some examples of strong monads.

\begin{example}  \label{ex:identitymonad} Obviously the identity monad is strong. If we take the identity monad $\idm$ on any cartesian cartesian category $\ambcat$ then a $\idm$-operad is the same as a category internal to $\ambcat$, and a non colored $\idm$-operad is a monoid in $\ambcat$. In particular if $\ambcat=\set$ they are small categories and monoids, respectively. 
\end{example}

\begin{example} \label{example:plainoperadoperations}Let $(\fm,\mu,\eta)$ be the free monoid monad on the category $\ambcat{=}\set$. As mentioned above, a $\fm$-operad is the same thing as a nonsymmetric operad. Here is the full explicit description of  $\fm$. Let $A$ be a set and $a_0, \dots, a_n{\in }A$, then
\begin{align}\label{eqdef:fm}
\nonumber \fm A&=\bigsqcup_{n\in N}A^n,&\phantom{aaaaaaaaaaaaaaaaaaa}\\\nonumber
\eta_A (a_0)&=(a_0),&\\
\mu_A\big((a_1,\dots,a_i),\dots,(a_j,\dots,a_n)\big)&=(a_1,\dots,a_n).&
\end{align}
The free monoid monad is strong with the following strength:
\begin{center}
\begin{tikzcd}[row sep=3pt, column sep=40pt]
D_{A,B}\colon A\cx \fm B \arrow[r] & \fm(A\cx B)\phantom{aaaaaaaaa}\\
\phantom{aaaaa}\big(a,(b_1,\dots,b_n)\big)\arrow[r] & \big((a,b_1),\dots,(a,b_n)\big).
\end{tikzcd}
\end{center}
It is straightforward to check that the diagrams (\ref{dgm:StrongMonad_AxiomMu}) and (\ref{dgm:StrongMonad_AxiomEta}) are satisfied and clear that $D_{A,B}$ is injective. This last feature is relevant because to define the $\tcons$-construction, in Section \ref{seccion:Tcons}, it will be necessary that $D_{1,C_0}$ is a monomorphism.
\end{example}

\begin{example}\label{ex:fsg} The free semigroup monad $\fsg$  on $\set$ is defined in the same way as the free monoid monad, except that in this case $\fsg A=\bigsqcup_{n\ge 1} A^n$. This means that a $\fsg$-operad is a nonsymmetric operad without nullary operations. The terminal $\fsg$-operad is denoted $\ass$, which is of course the reduced associative  operad. Notice that $\fsg$ is also a strong cartesian monad on $\infgrpd$. In this sense the operad $\ass$ can also be considered as an $\fsg$-operad in $\infgrpd$, with discrete groupoid of objects and discrete groupoid of operations. The context will suffice to distinguish between $\set$ and $\infgrpd$, but in the main applications (Section \ref{section:plethysmsandoperads}) we work over $\infgrpd$.
\end{example}

\begin{example} \label{ex:mm}Let $\monoid$ be a monoid. Denote by $\mm$ the monad on $\set$ given by $\mm A=\monoid\cx A$ with unit and multiplication given by those of $\monoid$. Then $\mm$ is strong with strength given by the associator of the cartesian product. Therefore in this case the strength is an isomorphism. The same holds if $\monoid$ is a monoid in $\infgrpd$ and $\mm$ is then a monad on $\infgrpd$.
\end{example}

\begin{example} \label{ex:fsmc}Let $(\fsmc,\mu,\eta)$ be the free symmetric monoidal category monad on $\infgrpd$. An $\fsmc$-operad is an operad internal to groupoids, so that it has a groupoid of colors and a groupoid of operations. Let $A$ be a groupoid and $\symgroup_n$ the symmetric group on $n$ elements. The  monad $\fsmc$ acts on $A$ by
$$ \fsmc A=\bigsqcup_{n\in N}A^n//\symgroup_n,$$
where $//$ means homotopy quotient \cite{Baez-Dolan,GKT:FdB}. Hence it is analogous to $\fm$, but we add an arrow 
$$(a_1,\dots,a_n)\xrightarrow{\;\;\;\sigma\;\;\;}(a_{\sigma 1},\dots, a_{\sigma n})$$
for every element $\sigma \in \symgroup_n$. The multiplication and unit natural transformations are defined as in (\ref{eqdef:fm}) for both objects and operations. Notice that any symmetric operad $\pop$ is in particular an $\fsmc$-operad, where the groupoid of objects $\pop_0$ is discrete and the groupoid $\pop_1$ has only the arrows coming from the permutations of its source sequence. In other words, a symmetric operad is an $\fsmc$-operad
$$\fsmc \pop_0\xleftarrow{\;\;\; s \;\;\;} \pop_1 \xrightarrow{\;\;\; t \;\;\;} \pop_0$$
such that $\pop_0$ is discrete and $s$ is a discrete fibration. 
The strength for $\fsmc$ is defined the same way as for $\fm$,
\begin{center}
\begin{tikzcd}[row sep=3pt, column sep=40pt]
D_{A,B}\colon A\cx \fsmc B \arrow[r] & \fsmc(A\cx B)\phantom{aaaaaaaaa}\\
\phantom{aaaaa}\big(a,(b_1,\dots,b_n)\big)\arrow[r] & \big((a,b_1),\dots,(a,b_n)\big),
\end{tikzcd}
\end{center}
and it is again a monomorphism, since it is injective both on objects and morphisms. 

Observe that symmetric operads cannot be expressed as $\gm$-operads in $\set$, since the actions of the symmetric groups have to be encoded necessarily as morphisms in $\pop_1$. Also, the only monad $\gm$ one could attempt to use to define them is the free commutative monoid monad, but it is not cartesian.

 %\begin{align}
%\nonumber \fm A&=\bigsqcup_{n\in N}A^n//\symgroup_n&\phantom{aaaaaaaaaaaaaaaaaaa},\\\nonumber
%\eta_A (a_0)&=(a_0)&\eta_A(\theta_0)=(\theta_0),\\ \nonumber
%\mu_A((a_1,\dots,a_i),\dots,(a_j,\dots,a_n))&=(a_1,\dots,a_n)&\\\nonumber
%\mu_A((\theta_1,\dots,\theta_i),\dots,(\theta_j,\dots,\theta_n))&=(\theta_1,\dots,\theta_n).&
%\end{align}

\end{example}

\begin{example} \label{ex:pfsmc}As for $\fm$ and $\fsg$, we can remove the empty sequence from $\fsmc$ to get a monad $\pfsmc$ on $\infgrpd$ whose operads do not have nullary operations. We denote by $\comm$ the terminal $\pfsmc$-operad, which is the reduced commutative  operad.
\end{example}

\subsection{The two-sided bar construction for $\gm$-operads}

The two-sided bar construction for operads is standard \cite{May}. In this section we introduce the construction in the more general setting of $\gm$-operads by using induced monads. Any $\gm$-operad $\pop$ defines a monad $(\mC,\mu^{\mC},\eta^{\mC})$\label{mC} on the slice category of $\ambcat$ over $\pop_0$
$$\mC:\ambcat/\pop_0\longrightarrow \ambcat{/\pop_0},$$
given by  pullback and composition, as shown in the following diagram for an element $X\xrightarrow{f} \pop_0$ of $\ambcat{/\pop_0}$
%---D-------I------A-------G-----R---------A-------M-------
\begin{equation}\label{eq:mC}
\begin{tikzcd}[sep={3em,between origins}]
&\mC X \arrow[ld]\arrow[rd,red]\dpbk&&\\
\gm X\arrow[rd,"\gm f"']&&\pop_1\arrow[ld,"s"]\arrow[rd,red,"t"]&\\
&\gm \pop_0& & \pop_0.
\end{tikzcd}
\end{equation}
%--------------------------------------------------------------
The image of $f$ is thus the red composite. The multiplication $\mu^{\mC}$ and the unit $\eta^{\mC}$ are defined by the following morphisms
%---D-------I------A-------G-----R---------A-------M-------
\begin{equation}
 \begin{tikzcd}
\pmb{\mC^2X \arrow[dd,"\mu^{\mC}_X"']}&\gm^2X\arrow[r,"\gm^2f"]\arrow[d,equal]&\gm^2\pop_0 \arrow[d,equal]&\gm \pop_1\arrow[l,"\gm s"']\arrow[r,"\gm t"]&\gm \pop_0&\pop_1\arrow[l,"s"']\arrow[r,"t"]&\pop_0\arrow[d,equal]\\
&\gm^2X\arrow[r,"\gm^2f"]\arrow[d,"\mu_X"]&\gm^2\pop_0 \arrow[d,"\mu_{\pop_0}"]&&\gm \pop_1\lcx{\gm \pop_0}\pop_1\arrow[ul,"s"]\arrow[ur,"t"']\arrow[ll]\arrow[rr]\arrow[d,"m"]\upbks&&\pop_0\arrow[d,equal]\\
\pmb{\mC X}& \gm X\arrow[r,"\gm f"]&\gm \pop_0 &&\pop_1\arrow[ll,"s"]\arrow[rr,"t"']&&\pop_0,
\end{tikzcd}
\end{equation}
%--------------------------------------------------------------
%---D-------I------A-------G-----R---------A-------M-------
\begin{equation}
\begin{tikzcd}
\pmb{X\arrow[d,"\eta^{\mC}_X"']}&X\arrow[r,"f"]\arrow[d,"\eta_{X}"]&\pop_0 \arrow[d,"\eta_{\pop_0}"]&\pop_0\arrow[l,equal]\arrow[r,equal]\arrow[d,"e"]&\pop_0\arrow[d,equal]&&&&&\\
\pmb{\gm_{\pop}X}&X\arrow[r,"\gm f"']&\gm \pop_0&\pop_1\arrow[l,"s"]\arrow[r,"t"']&\pop_0.&&&&&
\end{tikzcd}
\end{equation}
%--------------------------------------------------------------
\begin{definition} An \emph{algebra} over the $\gm$-operad $\pop$ is an algebra over the monad $\mC$.
\end{definition}
Notice that the category $\ambcat{/\pop_0}$ has a terminal object, $\pop_0\xrightarrow{1} \pop_0$, so that there is an algebra over $\mC$ given by the unique arrow $q:\mC 1\rightarrow 1$. Moreover, since $\ambcat$ has a terminal object, the $\gm$-operad $\gm:\ambcat{/1}\rightarrow \ambcat{/1}$ itself can be represented by the span
$$\gm 1\longleftarrow \gm 1\longrightarrow 1,$$
and is the terminal $\gm$-operad. Now, the terminal arrow $u:\pop_0\rightarrow 1$  induces, by postcomposition,  a functor $u_{!}:\ambcat{/{\pop_0}}\rightarrow \ambcat{/1}$. The  diagram
%---D-------I------A-------G-----R---------A-------M-------
\begin{equation}\label{eq:phi}
\begin{tikzcd}
\gm \pop_0\arrow[d,"\gm u"']&\pop_1\arrow[l]\arrow[r]\arrow[d]&\pop_0\arrow[d,"u"]\\
\gm 1&\gm 1\arrow[l]\arrow[r]&1.
\end{tikzcd}
\end{equation}
%--------------------------------------------------------------
represents a natural transformation $u_!\mC\xRightarrow{\phi\;} \gm u_!$ which is compatible with the comultiplications and units of $\mC$ and $\gm$, meaning that 
\begin{align*}%\label{eq:CoverPaxioms}
 \nonumber   u_!\mC^2\xRightarrow{\phi^2\;} \gm^2 u_! \xRightarrow{\mu u_!}\gm u_!&= u_!\mC^2\xRightarrow{u_!\mu^{\mC}} u_!\mC\xRightarrow{\phi\;} \gm u_! \;\;\;\;\;\;\;\; \text{and} \\
u_!\xRightarrow{u_!\eta^{\mC}}u_!\mC\xRightarrow{\phi}\gm u_!&=u_!\xRightarrow{\eta u_!}\gm u_!
\end{align*}
or, equivalently,

%---D-------I------A-------G-----R---------A-------M-------
\begin{equation}\label{dgm:CoverPmu}
\tiny
\begin{tikzcd}
\gm \pop_0\arrow[d,"\gm u"']&\gm^2\pop_0\arrow[l,"\mu_{\pop_0}"']\arrow[d,"\gm^2u"]&\gm \pop_1\arrow[l]\arrow[r]\arrow[d]&\gm \pop_0\arrow[d,"\gm u"']&\pop_1\arrow[l]\arrow[r]\arrow[d]&\pop_0\arrow[d,"u"]\\
\gm 1\arrow[dd,equal]&\gm^2 1\arrow[l,"\mu_1"]&\gm^2 1\arrow[l]\arrow[r,"\gm u"]&\gm 1&\gm 1\arrow[l]\arrow[r]&1\arrow[dd,equal]\\
&&&\gm^2 1\arrow[ur,"\gm u"']\arrow[ul]\arrow[d,"\mu_1"]\upbks&&\\
\gm 1&&&\gm 1\arrow[lll]\arrow[rr]&&1
\end{tikzcd}
=
\begin{tikzcd}
\gm \pop_0\arrow[dd,equal]&\gm^2\pop_0\arrow[l,"\mu_{\pop_0}"']&\gm \pop_1\arrow[l]\arrow[r]&\gm \pop_0&\pop_1\arrow[l]\arrow[r]&\pop_0\arrow[dd,equal]\\
&&&\pop_2\arrow[ur]\arrow[ul]\arrow[d]\upbks&&\\
\gm \pop_0\arrow[d,"\gm u"']&&&\pop_1\arrow[lll]\arrow[rr]\arrow[d]&&\pop_0\arrow[d,"u"]\\
\gm 1&&&\gm 1\arrow[lll]\arrow[rr]&&1,
\end{tikzcd}
\end{equation}
%--------------------------------------------------------------

%---D-------I------A-------G-----R---------A-------M-------
\begin{equation}\label{dgm:CoverPeta}
\begin{tikzcd}
\pop_0\arrow[d,"\eta_{\pop_0}"']&\pop_0\arrow[l,equal]\arrow[r,equal]\arrow[d,"e"]&\pop_0\arrow[d,equal]\\
\gm \pop_0\arrow[d,"\gm u"']&\pop_1\arrow[l]\arrow[r]\arrow[d]&\pop_0\arrow[d,"u"]\\
\gm 1&\gm 1\arrow[l]\arrow[r]&1.
\end{tikzcd}
=
\begin{tikzcd}
\pop_0\arrow[d,"u"']&\pop_0\arrow[l,equal]\arrow[r,equal]\arrow[d,"u"]&\pop_0\arrow[d,"u"]\\
1\arrow[d,"\eta_1"']&1\arrow[l]\arrow[r]\arrow[d,"\eta_1"]&1\arrow[d,equal]\\
\gm 1&\gm 1\arrow[l]\arrow[r]&1.
\end{tikzcd}
\end{equation}
%--------------------------------------------------------------

\begin{lemma} \label{lemma:phicartesian}The natural transformation $\phi$ is cartesian.
\begin{proof} Let us describe the naturality squares of $\phi$. Let $H$ be a map in $\ambcat/\pop_0$, that is, a commutative triangle
\begin{equation*}
\begin{tikzcd}[column sep=1em]
X\arrow[rr,"h"]\arrow[dr,"f"']&&Y\arrow[dl,"g"]\\
&\pop_0.&
\end{tikzcd}
\end{equation*}
Consider the diagram
%---D-------I------A-------G-----R---------A-------M-------
\begin{equation*}
\begin{tikzcd}
\gm X\lcx{\gm \pop_0}\pop_1\arrow[r,"u_!\mC H"]\arrow[d,"\phi_X"']&\gm Y\lcx{\gm \pop_0}\pop_1\arrow[d,"\phi_Y"']\arrow[r]\rdpbk&\pop_1\arrow[d,"s"]\\
\gm X\arrow[r,"\gm u_! H"']&\gm Y\arrow[r,"\gm g"']&\gm\pop_0.
\end{tikzcd}
\end{equation*}
%--------------------------------------------------------------
From (\ref{eq:mC}) it is clear that the pullback square on the right is precisely the definition of $u!\mC g$. From (\ref{eq:mC}) and (\ref{eq:phi}) we have that the square on the left is the naturality square for $\phi$ at $H$, and moreover that $\phi_X$ and $\phi_Y$ are projections. But $\gm u_!H=\gm h$ and $\gm g\circ \gm h=\gm f$, so that the composite square is precisely the definition of $u!\mC f$, which is a pullback. As a consequence, the naturality square is a pullback too.
\end{proof}
\end{lemma}

Given a $\gm$-operad $\pop$, we define its \emph{two-sided bar construction}\label{tsbar}  \cite{May,Weber:iac,Kock-Weber}
\label{simplexcategory}$$\tsbar \pop:\simplexcategory^{\text{op}}\longrightarrow \ambcat$$ 
\noindent as the two-sided bar construction of $\mC$, $\phi$ and the terminal algebra $1$. This means that the space of $n$-simplicies \label{tsbarnsimplices}$\tsbar_n\pop$ is given by 
$$\gm u_!\mC^n1,$$
the inner face maps are given by the monad multiplication $\mu^{\mC}$, the bottom face map is given by $c:\mC 1\rightarrow 1$ and the top face maps are given by $\phi$ and $\mu$. Similarly, the degeneracy maps are given by $\eta^{\mC}$. Diagrams (\ref{dgm:CoverPmu}) and (\ref{dgm:CoverPeta}) and the monad axioms for $\gm$ and $\mC$ guarantee that the simplicial identities are satisfied. 

In practice, the bar construction of $\pop$ is simply
\begin{equation}
\begin{tikzcd}
\gm \pop_0 &\gm \pop_1\arrow[l,shift left]\arrow[l, shift right]& \gm \pop_2 \arrow[l,shift left=2]\arrow[l, shift right=2]\arrow[l]&\gm \pop_3\arrow[l,shift left=3]\arrow[l, shift right=3]\arrow[l,shift left=1]\arrow[l, shift right=1]&\arrow[l]\arrow[l,shift left=2]\arrow[l, shift right=2]\arrow[l,shift left=4]\arrow[l, shift right=4]\cdots,
\end{tikzcd}
\end{equation}
where 
\begin{enumerate}[(i)]
\item $\pop_2:=\gm \pop_1\lcx{\gm \pop_0}\pop_1$ and $\pop_3:=\gm^2 \pop_1\lcx{\gm^2 \pop_0}\gm \pop_1\lcx{\gm \pop_0}\pop_1$, etc.;
\item the bottom face maps $d_0$ are induced by $t$;
\item the top face maps $d_n$ are induced by $s$ and $\mu$;
\item the inner face maps are induced by $m$ and $\mu$, and
\item the degeneracy maps are induced by $e$ and $\eta$.
\end{enumerate}
Henceforth we may indiscriminately use this simplicial notation. Let us see some examples.

\begin{example} Let $C$ be a small category. Hence $C$  is a $\idm$-operad in $\set$. Then $\tsbar C$ is the nerve of $C$. Moreover, we can consider $C$ as a category internal to $\infgrpd$ whose groupoid of objects has as morphisms the isomorphisms of $C$, and whose groupoid of arrows has as morphisms the isomorphisms of the arrow category of $C$. In this case  $\tsbar C$ is the fat nerve of $C$, whose groupoid of $n$-simplices is the groupoid $\text{Map}(\Delta[n],C)$. In the theory of
incidence coalgebras, this is often more intresting than the ordinary nerve, cf.\cite{GKT:DSIAMI-1,GKT:DSIAMI-2,Cebrian}
\end{example}

\begin{example} If $\pop$ is a symmetric operad, as in Example \ref{ex:fsmc}, then $\tsbar \pop$ is the usual operadic two-sided bar construction. Its $n$-simplices have as objects forests of $n$-level $\pop$-trees, and as morphisms permutations at each level. For example, the following picture
%---D-------I------A-------G-----R---------A-------M-------
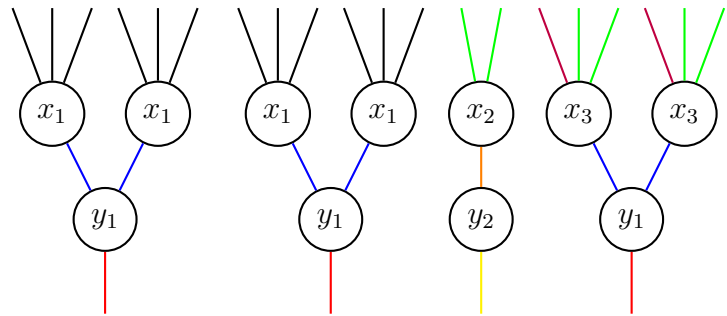
\begin{center}
\begin{tikzpicture}[grow=up,level distance=40pt,thick]
\tikzstyle{level 1}=[sibling distance=40pt]
\tikzstyle{level 2}=[sibling distance=15pt]
%\tikzstyle{every node}=[fill=red!60,circle,inner sep=1pt]
\node at (0,0){}
child[red] {node[black,circle,draw]{$y_1$}}{
	child[blue] {node[black,circle,draw]{$x_1$} child[black] child[black] child[black] }
	child[blue] {node[black,circle,draw]{$x_1$} child[black] child[black] child[black]}
};
\node at (3,0){}
child[red] {node[black,circle,draw]{$y_1$}}{
	child[blue] {node[black,circle,draw]{$x_1$} child[black] child[black] child[black] }
	child[blue] {node[black,circle,draw]{$x_1$} child[black] child[black] child[black]}
};
\node at (5,0){}
child[yellow] {node[black,circle,draw,minimum size=20pt]{$y_2$}}{
child[orange] {node[black,circle,draw]{$x_2$} child[green] child[green]  }
};
\node at (7,0){}
child[red] {node[black,circle,draw]{$y_1$}}{
	child[blue] {node[black,circle,draw]{$x_3$} child[green] child[green] child[purple] }
	child[blue] {node[black,circle,draw]{$x_3$} child[green] child[green] child[purple]}
};

%\draw (0,0) edge[<->,bend right=45] node[above]{2!} (3,0);
%\draw (0.5,0) edge[<->,bend right=45] node[below]{2!} (4.5,0);

\end{tikzpicture}
\end{center}
%--------------------------------------------------------------
is an object of $\gm \pop_2$ with $(2!\cdot 2!^2\cdot 3!^4)\cdot (2!)\cdot (2!\cdot 2!^2)$ automorphisms.
\end{example}

The following result is a reformulation of \cite[Proposition 4.4.1]{Weber:iac} and \cite[Proposition 3.3]{Kock-Weber} in the context of $\gm$-operads.
\begin{proposition} \label{prop:categoryobject}The simplicial object $\tsbar \pop$ is a strict category object.
\begin{proof} We have to check that the squares
%---D-------I------A-------G-----R---------A-------M-------
\begin{equation}\label{eq:categorysquares}
\begin{tikzcd}
\tsbar_{n+2}\pop\arrow[r,"d_0"]\arrow[d,"d_{n+2}"']&\tsbar_{n+1}\pop\arrow[d,"d_{n+1}"]\\
\tsbar_{n+1}\pop\arrow[r,"d_0"']&\tsbar_n\pop.
\end{tikzcd}
\end{equation}
%--------------------------------------------------------------
are pullbacks for $n\ge0$. We show the case $n=0$, the rest are similar. The square is given by
%---D-------I------A-------G-----R---------A-------M-------
\begin{equation}
\begin{tikzcd}
\gm u_!\mC\mC1\arrow[r,"\gm u_!\mC c"]\arrow[d,"\gm (\phi_{\mC1})"']&\gm u_!\mC1\arrow[d," \gm(\phi_1)"]\\
\gm\gm u_!\mC1\arrow[r,"\gm\gm u_! c"]\arrow[d,"\mu^{\gm}_{u!\mC1}"']&\gm\gm u_!1\arrow[d,"\mu^{\gm}_{u!1}"]\\
\gm u_!\mC1\arrow[r,"\gm u_! c"']&\gm u_!1.
\end{tikzcd}
\end{equation}
%--------------------------------------------------------------
The bottom square is cartesian because it is a naturality square for $\mu^{\gm}$, and $\gm$ is a cartesian monad. The top square is $\gm$ applied to a naturality square of $\phi$, which is cartesian, by Lemma \ref{lemma:phicartesian}. Since $\gm$ preserves pullbacks, the square is cartesian.
\end{proof}
\end{proposition}
\noindent This allows to obtain the following result, in the special case where $\ambcat=\infgrpd$.
\begin{proposition} \label{prop:segalsquares}Let $\gm:\infgrpd\rightarrow\infgrpd$ be a cartesian monad that preserves fibrations. Let $\pop$ be a $\gm$-operad such that $\pop_0$ is a discrete groupoid. Then the simplicial groupoid $\tsbar \pop$ is a Segal groupoid.  
\begin{proof} It is enough to see that the strict pullbacks \ref{eq:categorysquares} are also homotopy pullbacks. For $n=0$, notice that $\gm u_!\mC1\xrightarrow{\gm u_! c}\gm u_!1$ is precisely the map $\gm Q_1\xrightarrow{\gm\comp}\gm Q_0$. But since $Q_0$ is discrete $\comp$ is a fibration, which means that $\gm \comp$ is a fibration, because $\gm$ preserves fibrations. This implies that the square is also a homotopy pullback. Moreover, since pullbacks preserve fibrations the map $\gm u_!\mC\mC1\xrightarrow{\gm u_!\mC c}\gm u_!\mC1$ is again a fibration. The same argument then implies that the square for $n=1$ is also a homotopy pullback, and so on.
\end{proof}
\end{proposition}

Suppose now that $\mathsf{R}:\ambcat\rightarrow\ambcat$ is another cartesian monad and that there is a cartesian monad map $\gm\xRightarrow{\psi}\mathsf{R}$. Then we can take the bar construction over $\mathsf{R}$
$$\tsbar^{\mathsf{R}} \pop:\simplexcategory^{\text{op}}\longrightarrow \ambcat$$ \label{tsbargm}
whose $n$-simplices are given by
$$\mathsf{R} u_!\mC^n1 \;\;(\text{or }\;\mathsf{R} Q_n).$$
In this case all the face maps coincide with the previous ones except the top face map, which is given by $$\mathsf{R} u_!\mC^{n+1}1\xrightarrow{\mathsf{R}(\phi_{Q1})}\mathsf{R} \gm u_!\mC^n1\xrightarrow{\mathsf{R}(\psi_{u_!\mC^n1})}\mathsf{RR}  u_!\mC^n1 \xrightarrow{\mu^{\mathsf{R}}_{u_!\mC^n1}} \mathsf{R}u_!\mC^n1.$$
Since $\psi$ is cartesian the simplicial object $\tsbar^{\mathsf{R}}$ is also a strict category object. Moreover, if $\mathsf{R}$ preserves fibrations, it is a Segal groupoid, for the same reason as $\tsbar \pop$ in Proposition \ref{prop:segalsquares}.
The main examples of this bar construction that we use come from the natural transformations $\fsg\Rightarrow \fsmc$, $\pfsmc\Rightarrow \fsmc$ and $\fsg\Rightarrow \fm$, as in~\cite{Kock-Weber}.

%--------------------------------------------------------------------------------------------------------------------------------------
%---------------------------------------------SEGAL GROUPOIDS-----------------------------------------------------------
%---------------------------------------------SEGAL GROUPOIDS-----------------------------------------------------------
%---------------------------------------------SEGAL GROUPOIDS-----------------------------------------------------------
%---------------------------------------------SEGAL GROUPOIDS-----------------------------------------------------------
%---------------------------------------------SEGAL GROUPOIDS-----------------------------------------------------------
%---------------------------------------------SEGAL GROUPOIDS-----------------------------------------------------------
%---------------------------------------------SEGAL GROUPOIDS-----------------------------------------------------------
%-------------------------------------------------------------------------------------------------------------------------------------

\section{Segal groupoids and incidence coalgebras}\label{preliminaries}\label{seccion:HoCard}

Throughout  this section, pullbacks and fibers of groupoids refer to homotopy pullbacks and homotopy fibers. A brief introduction to the homotopy approach to groupoids in combinatorics can be found in \cite[\S 3]{GKT:FdB}.
\subsection{Segal groupoids}
 A simplicial groupoid  $X\colon \simplexcategory^{\text{op}}\longrightarrow \infgrpd$  is a \emph{Segal space} \cite[\S 2.9, Lemma 2.10]{GKT:DSIAMI-1} if the following square is a pullback for all $n>0$:
 %---D-------I------A-------G-----R---------A-------M-------
\begin{equation}\label{segal}
\begin{tikzcd}[row sep=2em, column sep=2em]
X_{n+1} \ar[r,"d_0"] \ar[d,"d_{n+1}"'] \rdpbk& X_n \ar[d,"d_n"]\\
	      X_n \ar[r,"d_0"'] & X_{n-1}.
\end{tikzcd}
\end{equation}
%--------------------------------------------------------------

Segal spaces arise prominently through the fat nerve construction: the fat nerve of a category $\mathcal{C}$ is the simplicial groupoid $X=N\mathcal{C}$ with $X_n=\fun ([n],\C)^{\simeq}$, the groupoid of functors $[n] \to \C$. In this case the pullbacks above are strict, so that all the simplices are strictly determined by $X_0$ and $X_1$, respectively the objects and arrows of $\C$, and the inner face maps are given by composition of arrows in $\C$. In the general case, $X_n$ is determined from $X_0$ and $X_1$ only up to equivalence, but one may still think of it as a ``category'' object whose composition is defined only up to equivalence.

\begin{remark} \label{rk:homotopystrictpullback}
Despite the Segal conditions (\ref{segal}) require the squares to be homotopy pullbacks, if the top or bottom face maps are
fibrations, the ordinary pullbacks are
also homotopy pullbacks. In the present
work, homotopy pullbacks  mostly arise n this
way.\end{remark}

\subsection{Incidence coalgebras}
Let $X$ be a simplicial groupoid.
The spans
$$X_1\xleftarrow{\;\;d_1\;  \;} X_2 \xrightarrow{(d_2,d_0)\;}X_1\times X_1,    \;\;\;\;\;\;\;\;\;\;\;\;\;\; X_1\xleftarrow{\;\;s_0\; \;} X_0 \xrightarrow{\;\;t\;\;}1, $$

\noindent define two functors 

\begin{center}
\begin{tabular}{rllccrll}
$\Delta \colon \infgrpd_{/X_1}$& $\longrightarrow$ &$\mathbf{Grpd}_{/X_1\times X_1}$ && & $\epsilon \colon \infgrpd_{/X_1}$& $\longrightarrow$ &$\infgrpd$\\
$S\xrightarrow{s} X_1$& $\longmapsto$ & $(d_2,d_0)_!\circ d_1^{\ast}(s)$, && &$S\xrightarrow{s} X_1$& $\longmapsto$ & $t_!\circ s_0^{\ast}(s)$ .
\end{tabular}
\end{center} 
Recall that upperstar is homotopy pullback and lowershriek is postcomposition. This is the general way in which spans interpret homotopy linear algebra \cite{GKT:HLA}.

Segal spaces are a particular case of decomposition spaces \cite[Proposition 3.7]{GKT:DSIAMI-1}, simplicial groupoids with the property that the functor $\Delta$ is coassociative with the functor $\epsilon$ as counit (up to homotopy). In this case $\Delta$ and $\epsilon$ endow $\infgrpd_{/X_1}$ with a coalgebra structure \cite[\S 5]{GKT:DSIAMI-1}   called the \emph{incidence coalgebra} of $X$. Note that in the special case where $X$ is the nerve of a poset, this
construction becomes the classical incidence coalgebra \cite{Rota,Schmitt} construction after
taking cardinality, as we shall do shortly.

The morphisms of decomposition spaces that induce coalgebra homomorphisms are the so-called \emph{CULF} functors \cite[\S 4]{GKT:DSIAMI-1}, standing for conservative and unique-lifting-of-factorisations. A Segal space $X$ is \emph{CULF monoidal} if it is a monoid object in the monoidal category $(\mathbf{Dcmp}^{\text{CULF}},\times,1)$ of decomposition spaces and CULF functors \cite[\S 9]{GKT:DSIAMI-1}. More concretely, it is CULF monoidal
 if there is a product $X_n\times X_n\rightarrow X_n$ for each $n$, compatible with the degeneracy and face maps, and such that for all $n$ the squares
 %---D-------I------A-------G-----R---------A-------M-------
\begin{equation}
\label{monoidalpullback}
\begin{tikzcd}
X_n\cx X_n \ar[r,"g\cx g"] \ar[d] \rdpbk& X_1\cx X_1 \ar[d]\\
	      X_n \ar[r,"g"'] & X_1
\end{tikzcd}
\end{equation}
%--------------------------------------------------------------
are pullbacks \cite[\S  4]{GKT:DSIAMI-1}. Here $g$ is induced by the unique endpoint-preserving map $[1]\rightarrow [n]$. For example the fat nerve of a monoidal extensive category is a CULF monoidal Segal space. Recall that a category $\C$ is monoidal extensive if it is monoidal $(\C,+,0)$ and the natural functors $\C_{/A}\times \C_{/B}\rightarrow \C_{/A+B}$  and $\C_{/0}\rightarrow 1$ are equivalences.
 
If $X$ is CULF monoidal then the resulting coalgebra is in fact a bialgebra \cite[\S 9]{GKT:DSIAMI-1}, with product given by
\begin{center}
\begin{tabular}{rllccrll}
$\grpdSprod\colon \infgrpd_{/X_1}\otimes \infgrpd_{/X_1}$&$\xrightarrow{\; \sim \;}$& $\infgrpd_{/X_1\times X_1}$  &$\xrightarrow{\;\; +_! \;\;}$ &$\infgrpd_{/X_1}$ \\
$(G\rightarrow X_1)\otimes (H\rightarrow X_1)$& $\longmapsto$ & $G\times H\rightarrow X_1\times X_1$&$ \longmapsto $& $ G\times H\rightarrow X_1$.
\end{tabular}
\end{center} 
Briefly, a product in $X_n$ compatible with the simplicial structure endows $X$ with a product, but in order to be compatible with the coproduct it has to satisfy the diagram \eqref{monoidalpullback} (i.e. it has to be a CULF functor).

\subsection{Homotopy cardinality}
A groupoid $X$ is \emph{finite} if $\pi_0(X)$ is a finite set and $\pi_1(x)=\aut(x)$ is a finite group for every point $x$. If only the latter is satisfied then it is called \emph{locally finite}. A morphism of groupoids is called finite when all its fibers are finite. The \emph{homotopy cardinality} \cite{Baez-Dolan}, \cite[\S 3]{GKT:HLA} of a finite goupoid $X$ is defined as 
$$|X|:=\sum_{x\in \pi_0X}\frac{1}{|\aut( x)|}\in \Q,$$
and the homotopy cardinality of a finite map of groupoids $A\xrightarrow{p}B$ is 
$$|p|:=\sum_{b\in \pi_0B}\frac{|A_b|}{|\aut(b)|}\delta_b,$$
in  $\Q_{\pi_0B}$, the vector space spanned by $\pi_0B$. In this sum, $A_b$ is the homotopy fiber and $\delta_b$ is a formal symbol representing the isomorphism class of $b$. A simple computation shows that $|1\xrightarrow{\name{b}}B|=\delta_b$.

 A Segal space $X$ is \emph{locally finite} \cite[\S 7]{GKT:DSIAMI-2} if $X_1$ is a locally finite groupoid and both $s_0\colon X_0\rightarrow X_1$ and $d_1\colon X_2\rightarrow X_1$ are finite maps. In this case one can take homotopy cardinality to get a comultiplication
\begin{center}
\begin{tabular}{rll}
$\Delta \colon \Q_{\pi_0X_1}$& $\longrightarrow$ &$ \Q_{\pi_0X_1}\otimes  \Q_{\pi_0X_1}$\\
$|S\xrightarrow{s} X_1|$& $\longmapsto$ & $|(d_2,d_0)_!\circ d_1^{\ast}(s)|$
\end{tabular}
\end{center} 
and similarly for $\epsilon$ (cf. \cite[\S 7]{GKT:DSIAMI-2}). Moreover, if $X$ is CULF monoidal then $\Q_{\pi_0X_1}$ acquires a bialgebra structure with the product $\cdot=|\grpdSprod|$. In particular, if we denote by $+$ the monoidal product in $X$, then
$\delta_{a}\cdot \delta_{b}=\delta_{a+b}$ for any $|1\xrightarrow{\name{a}}X_1|$ and $|1\xrightarrow{\name{b}}X_1|$. The following result gives a closed formula for the computation of the comultiplication when $X$ is a Segal space.

\begin{lemma}[\cite{GKT:DSIAMI-2:corrigendum,Cebrian}]
\label{SegalCom} 
Let $X$ be a Segal space. Then for $f$ in $X_1$ we have
$$\Delta(\delta_f)=\sum_{b\in \pi_0X_1}\sum_{a\in \pi_0X_1} \frac{|\iso(d_0a,d_1b)_{f}|}{|\aut(b)||\aut(a)|}\delta_a\otimes \delta_b,$$
where $\iso(d_0a,d_1b)$ is the set of morphisms from $d_0a$ to $d_1b$ and $\iso(d_0a,d_1b)_{f}$ is its homotopy fiber along $d_1$.
\end{lemma}

%-----------------------------------------------------------------------------------------------------------------------------------
%----------------------------------------------THE CONSTRUCTION------------------------------------------------------
%----------------------------------------------THE CONSTRUCTION------------------------------------------------------
%----------------------------------------------THE CONSTRUCTION------------------------------------------------------
%----------------------------------------------THE CONSTRUCTION------------------------------------------------------
%----------------------------------------------THE CONSTRUCTION------------------------------------------------------
%----------------------------------------------THE CONSTRUCTION------------------------------------------------------
%----------------------------------------------THE CONSTRUCTION------------------------------------------------------
%-----------------------------------------------------------------------------------------------------------------------------------

%%%%%%%%%%%%%%%%%%%%%%%%%%%%%%%%%%%%%%%%%%%%%%%%%%%
 %%%%%%%%%%%%%%%%%%%%%%%%%%%%%%%%%%%%%%%%%%%%%%%%%%%
 %-----------------------------------------------------------------------------------------------------------------------------------%
 %------------------------------Section: Categories to T-Multicategories-----------------------------------------------%
  %-----------------------------------------------------------------------------------------------------------------------------------%

\section{The $\tcons$-construction}\label{seccion:Tcons}

 Throughout this section $(\gm,\mu,\eta)$ is a cartesian strong monad on a cartesian category $\ambcat$, and category means a category internal to $\ambcat$. As mentioned in the introduction, the $\tcons$-construction consists of two construction, one from internal categories to $\gm$-operads and another from $\gm$-operads to categories. With the purpose of reducing the diagrams and fiber products, we use the following notation for the endofunctors and natural transformations featuring in this section,
\begin{equation*}\label{eq:TxD}
\begin{tikzcd}[row sep=5pt]
\Tx: \ambcat \arrow[r] &\ambcat &&F:\idm\arrow[r,Rightarrow]&\Tx\\
\phantom{P:}A\arrow[r, maps to]& A\cx\gm 1, \!\!\!&&F_A: A\cx 1\arrow[r,"\id\cx\eta_1"]&\Tx A,\\
&&&&\\
D:\Tx \arrow[r,Rightarrow]&\gm &&R:\Tx\arrow[r,Rightarrow] &\idm\\
D_A:A\cx\gm1\arrow[r]&\gm A,&&R_A:A\cx\gm1\arrow[r,"p_1"]&A.
\end{tikzcd}
\end{equation*}
Observe that $\Tx$ is cartesian as a functor. Also, notice that $R$ and $F$ are cartesian natural transformations. Finally, by monomorphism we refer to the $1$-categorical notion. In the case of most interest where $\ambcat$ is  $\set$ or
$\infgrpd$, this means injective on objects and injective on arrows.

The material of this section is highly technical. The casual or application-oriented reader might wish to regard it as a black box and take on faith the well-definedness of the constructions, and still be able to appreciate the examples worked out in Chapters~\ref{seccion:examples} and~\ref{section:plethysmsandoperads}.

\subsection{From categories to $\gm$-operads}
%-----------------------------------------------------------------------------------------------------------------------------------% %-------------------------------------------------------------definition-------------------------------------------------------------%
%-----------------------------------------------------------------------------------------------------------------------------------%
 Let $C$ be a category such that $D_{C_0}:\gm 1\cx C_0\rightarrow \gm C_0$ is a monomorphism. It is convenient in this section to adopt a simplicial nomenclature. Hence $C$ is represented by the span
%---D-------I------A-------G-----R---------A-------M-------
\begin{center} 
 \begin{tikzcd}[column sep=small]
       & C_1\arrow[dl,"d_1"']\arrow[dr,"d_0"] &       \\
 C_0 &                                            & C_0
 \end{tikzcd}
 \begin{tikzcd}
 C_1\cx_{C_0}C_1=:C_2 \arrow[r,"d_1"]& C_1\\
 \phantom{aaaaa}C_0 \arrow[r,"e"] &C_1,
 \end{tikzcd}
\end{center}
%--------------------------------------------------------------
\noindent  with the only inconvenience that some of the face maps share their names. Notice that we still denote by $e$ the degeneracy map $s_0$. We now construct a $\gm$-operad $\Tpc{}{\gm}{C}$\label{TpcCQ} from the category $C$. To keep notation short, the simplicial nomenclature for $\Tpc{}{\gm}{C}$ is $\tilde{C}_i$ for the simplices and $\tilde{d}_i$ for the face maps. The span defining the objects and operations of $\Tpc{}{\gm}{C}$ is given by the pullback
%---D-------I------A-------G-----R---------A-------M-------
 \begin{equation}\label{dgm:Cat_TMulticat_Def}
\begin{tikzcd}[sep={4em,between origins}]
 && \tilde{C}_1 \ar[dl,tail,"i_1"'] \ar[dr] \arrow[ddll,bend right,"\tilde{d}_1"']\arrow[ddrr, bend left, "\tilde{d}_0"]\dpbk &&\\
 & \gm C_1\ar[dl,"\gm d_1"]\ar[dr,"\gm d_0"] & & \Tx C_0 \ar[dl,tail,"D_{C_0}"'] \ar[dr,"R_{C_0}"']& \\
 \gm C_0 && \gm C_0 && C_0.
 \end{tikzcd}
 \end{equation}
 %--------------------------------------------------------------
 Observe that $\tilde{C}_0=C_0$, so that $\Tpc{}{\gm}{C}$ has the same objects as $C$. Besides, the morphism $i_1$ is a monomorphism, since monomorphisms are preserved by pullbacks and $D_{C_0}$ is a monomorphism.

 %-----------------------------------------------------------------------------------------------------------------------------------% %-------------------------------------------------------------composition-------------------------------------------------------------%
%-----------------------------------------------------------------------------------------------------------------------------------%
%\subsubsection*{Composition in $\tilde{C}$}
 To define composition we need to specify a map $\tilde{C}_2 \xrightarrow{\tilde{d}_1} \tilde{C}_1$, where $\tilde{C}_2:=\gm\tilde{C}_1\lcx{\gm C_0}\tilde{C}_1$, satisfying the axioms of Appendix \ref{Cat_Axioms}. 
 %Hence, we need to give an arrow
% \begin{equation*}
% \tiny
% \begin{tikzcd}[row sep=0pt,column sep=30pt]
% \gm^2 C_0\arrow[ddrr,"\mu_{C_0}"']&\gm^2C_1\arrow[l,"\gm^2 d_1"']\arrow[r,"\gm^2d_0"]&\gm^2C_0& \gm \Tx C_0\arrow[l,"\gm D_{C_0}"']\arrow[r,"\gm R_{C_0}"]&\gm C_0&\gm C_1 \arrow[l,"\gm d_1"']\arrow[r,"\gm d_0"]&\gm C_0& \Tx C_0\arrow[l,"D_{C_0}"']\arrow[r,"R_{C_0}"]&C_0\arrow[ddll,"id"]\\
%&&&& \arrowdown{$\tilde{d}_1$}&&&&\\
%&&\gm C_0&\gm C_1 \arrow[l,"\gm d_1"']\arrow[r,"\gm d_0"]&\gm C_0& \Tx C_0\arrow[l,"D_{C_0}"']\arrow[r,"R_{C_0}"]&C_0&&
% \end{tikzcd}
% \end{equation*}
 However, to describe it we have to express $\tilde{C}_2$ in a way we can naturally use composition in the original category $C_2\xrightarrow{d_1} C_1$. The following diagram represents an isomorphism 
 $$\tilde{C}_2\cong \gm^2C_1\lcx{\gm^2C_0}\gm \Tx C_1\lcx{\gm\Tx C_0}(C_0\cx\gm^21)=:\tilde{C}_2',$$
 %---D-------I------A-------G-----R---------A-------M-------
 \begin{equation}\label{dgm:tildeC2tildeC2prima}
 \tiny
 \begin{tikzcd}[row sep=30pt,column sep=30pt]
\gm^2 C_0\arrow[d,equal]&\gm^2 C_1 \arrow[l,"\gm^2 d_1"']\arrow[r,"\gm^2 d_0"]\arrow[d,equal]&\gm^2 C_0 \arrow[d,equal]&&\gm \Tx C_1\arrow[ll]\arrow[rr,"\gm\Tx d_0" near end]\arrow[dl,"\gm \Tx d_1"',"\phantom{aaaa}(A)"]\arrow[rd,"\gm R_{C_1}"]\dpbks&&\gm\Tx C_0\arrow[d,"\gm R_{C_0}"']&C_0\cx\gm^21 \arrow[l,"D_{\gm 1,C_0}"']\arrow[r,"p_1"]\arrow[d,"(B)\phantom{aaaaa}"',"\id\cx\gm u"]\ldpbk&C_0.\arrow[d,equal]\\
\gm^2 C_0&\gm^2C_1\arrow[l,"\gm^2 d_1"]\arrow[r,"\gm^2d_0"']&\gm^2C_0& \gm \Tx C_0\arrow[l,"\gm D_{C_0}"]\arrow[r,"\gm R_{C_0}"']&\gm C_0&\gm C_1 \arrow[l,"\gm d_1"]\arrow[r,"\gm d_0"']&\gm C_0& \Tx C_0\arrow[l,"D_{C_0}"]\arrow[r,"R_{C_0}"']&C_0
 \end{tikzcd}
 \end{equation}
 %--------------------------------------------------------------
 It is clear that all the squares in (\ref{dgm:tildeC2tildeC2prima}) commute. Moreover, the square $(A)$ is cartesian because $R$ and $\gm$ are cartesian, and the square $(B)$ is the same as (\ref{dgm:StrongMonadTupullback_Lemma}) of Lemma \ref{lemma:upbk}.
 \begin{definition}\label{def:tildecomp} The composition of $\Tpc{}{\gm}{C}$ is given by the following arrow $\tilde{C}_2'\xrightarrow{\tilde{d}_1'} \tilde{C}_1$,
%---D-------I------A-------G-----R---------A-------M------- 
  \begin{equation}\label{dgm:tildeCcomposition}
 \tiny
 \begin{tikzcd}[row sep=20pt,column sep=30pt]
\pmb{\tilde{C}_2'\arrow[dddd,"\tilde{d}_1'"]}&\gm^2 C_0\arrow[d,equal]&\gm^2 C_1 \arrow[d,equal]\arrow[r,"\gm d_0"]\arrow[l,"\gm^2 d_1"']&\gm^2 C_0\arrow[d,equal]&\gm \Tx C_1\arrow[l]\arrow[r,"\gm\Tx d_0"]\arrow[d,"\gm D_{C_1}\phantom{aa}(A)"]&\gm\Tx C_0\arrow[d,"\gm D_{C_0}\phantom{aa}(B)"]&C_0\cx\gm^21\arrow[l,"D_{\gm1,C_0}"']\arrow[d,equal]\arrow[r,"p_1"]&C_0\arrow[d,equal]\\
&\gm^2 C_0\arrow[dd,equal,"\phantom{aaaaaa}(C)"]&\gm^2 C_1 \arrow[r,"\gm^2 d_0"]\arrow[l,"\gm^2 d_1"']&\gm^2 C_0&\gm^2 C_1\arrow[l,"\gm^2d_1"']\arrow[r,"\gm^2 d_0"']&\gm^2 C_0\arrow[dd,equal,"(D)\phantom{aaaaaaa}"']&C_0\cx\gm^21 \arrow[l,"D^2_{C_0}"]\arrow[r,"p_1"]\arrow[ddd,"\id\cx\mu_1","(G)\phantom{aaaaa}"']&C_0\arrow[ddd,equal]\\
&&& \gm^2 C_2\arrow[ul,"\gm^2d_2"]\arrow[ur,"\gm^2d_0"']\arrow[d,"\gm^2d_1"]\upbks&&&&\\
&\gm^2 C_0 \arrow[d,"\mu_{C_0}\phantom{aaaaa}(E)"]& &\gm^2 C_1\arrow[ll,"\gm^2d_1"']\arrow[rr,"\gm^2d_0"]\arrow[d,"\mu_{C_1}\phantom{aaaaa}(F)"]&&\gm^2C_0 \arrow[d,"\mu_{C_0}"]&&\\
\pmb{\tilde{C}_1}&\gm C_0 &&\gm C_1\arrow[ll,"\gm d_1"']\arrow[rr,"d_0"]&&\gm C_0& \Tx C_0\arrow[l,"D_{C_0}"']\arrow[r,"R_{C_0}"]&C_0.
 \end{tikzcd}
 \end{equation}
 %--------------------------------------------------------------
 It is clear that the diagram commutes: $(A)$ is $\gm$ applied to a naturality square of $D$; $(B)$ is the definition of $D^2$; $(C)$ and $(D)$ are $\gm^2$ applied to axioms (\ref{Cat_d1d1}) and (\ref{Cat_d0d1}) for composition in $C$; $(E)$ and $(F)$ are naturality squares of $\mu$, and $(G)$ is again axiom (\ref{dgm:StrongMonad_AxiomMu}) for strong monads. The remaining squares are trivial.
\end{definition}
Notice that from this definition it is clear that $\tilde{d}_1$ satisfies axioms (\ref{PCat_d1d1}) and (\ref{PCat_d0d1}). Furthermore, there is a map
\begin{center}\begin{tikzcd}\tilde{C}_2'\arrow[r,"i_2'"] &\gm^2 C_2,\end{tikzcd}\end{center}
given by the diagram
%---D-------I------A-------G-----R---------A-------M-------
 \begin{equation*}
 \small
\begin{tikzcd}[row sep=20pt,column sep=35pt]
\gm^2C_0\arrow[d,equal]&\gm^2 C_1\arrow[l,"\gm^2d_1"'] \arrow[d,equal]\arrow[r,"\gm d_0"]&\gm^2 C_0\arrow[d,equal]&\gm \Tx C_1\arrow[l]\arrow[r,"\gm\Tx d_0"]\arrow[d,"\gm D_{C_1}"]&\gm\Tx C_0\arrow[d,"\gm D_{C_0}"]&C_0\cx\gm^21 \arrow[l,"D_{\gm1,C_0}"']\arrow[r]&C_0\\
\gm^2C_0&\gm^2 C_1 \arrow[l,"\gm^2d_1"']\arrow[r,"\gm^2 d_0"']&\gm^2 C_0&\gm^2 C_1\arrow[l,"\gm^2d_1"]\arrow[r,"\gm^2 d_0"']&\gm^2 C_0,&&
 \end{tikzcd}
 \end{equation*}
%--------------------------------------------------------------
which clearly makes the square 
%---D-------I------A-------G-----R---------A-------M-------
\begin{equation}\label{dgm:Cat_TMulticat_C2d1restriction}
\begin{tikzcd}
\tilde{C}_2' \arrow[dd, "\tilde{d}_1'"'] \arrow[r,"i'_2"] & \gm ^2C_2 \arrow[d, "\gm ^2d_1"]     \\
                                                       & \gm ^2 C_1 \arrow[d, "\mu_{C_1}"] \\
\tilde{C}_1 \arrow[r, tail,"i_1"']                            & \gm C_1,                       
\end{tikzcd}
\phantom{aaa}
\begin{tikzcd}[row sep=2pt]
\text{and therefore}\\
\text{also the square}
\end{tikzcd}
\phantom{aaa}
\begin{tikzcd}
\tilde{C}_2 \arrow[dd, "\tilde{d}_1"'] \arrow[r,"i_2"] & \gm ^2C_2 \arrow[d, "\gm ^2d_1"]     \\
                                                       & \gm ^2 C_1 \arrow[d, "\mu_{C_1}"] \\
\tilde{C}_1 \arrow[r, tail,"i_1"']                            & \gm C_1,                       
\end{tikzcd}
\end{equation}
%--------------------------------------------------------------
commute, for the corresponding arrow $i_2$. This says, roughly speaking, that composition in $\Tpc{}{\gm}{C}$ is ``the same'' as composition in $\gm^2C$, as it is also clear in most of the examples.

%-----------------------------------------------------------------------------------------------------------------------------------% %-------------------------------------------------------------associativity------------------------------------------------------%
%-----------------------------------------------------------------------------------------------------------------------------------%
%\subsubsection*{Associativity}
We have to check that composition is associative (\ref{dgm:PCat_Associativity}). We state first the following lemma. We omit  its proof since it is long and unenlightening; it can be found in \cite{Cebrian2}.
%Continuing with the simplicial notation we define  $\tilde{C}_3$ to be the triple composite $\gm^2\tilde{C}_1\cx_{\gm^2C_0}\gm \tilde{C}_1\cx_{\gm C_0}\tilde{C}_1$ of $\tilde{C}$ as a $1$-cell in $\ambcat_{(\gm)}$, with $\tilde{d}_i$ its corresponding face maps for $i=0,1,2,3$.
\begin{lemma}  \label{lemma:tilderestriction} There is a map $\tilde{C}_3\xrightarrow{\;\;i_3\;\;}\gm^3C_3$ such that the following diagrams commute
%---D-------I------A-------G-----R---------A-------M-------
\begin{equation}\label{dgm:tilderestriction}
\begin{tikzcd}
\tilde{C}_3 \arrow[dd, "\tilde{d}_1"'] \arrow[r,"i_3"] & \gm ^3C_3 \arrow[d, "\gm ^3d_1"]     \\
                                                       & \gm ^3 C_2 \arrow[d, "\gm \mu_{C_2}"] \\
\tilde{C}_2 \arrow[r,"i_2"']                            & \gm ^2C_2,                      
\end{tikzcd}
\phantom{aaaaaaaaaaa}
\begin{tikzcd}
\tilde{C}_3 \arrow[dd, "\tilde{d}_2"'] \arrow[r,"i_3"] & \gm ^3C_3 \arrow[d, "\gm ^3d_2"]     \\
                                                       & \gm ^3 C_2 \arrow[d, " \mu_{\gm C_2}"] \\
\tilde{C}_2 \arrow[r,"i_2"']                            & \gm ^2C_2.                      
\end{tikzcd}
\end{equation}
%--------------------------------------------------------------
\end{lemma}
\begin{proposition} Composition is associative.
\begin{proof} In view of Lemma \ref{lemma:tilderestriction} there is a diagram
%---D-------I------A-------G-----R---------A-------M-------
\begin{equation*}
\begin{tikzcd}
\tilde{C}_3\arrow[rrrr,"\tilde{d}_1"]\arrow[dddd,"\tilde{d}_2"']\arrow[rd,"i_3"]&&&&\tilde{C}_2\arrow[dddd,"\tilde{d}_1"]\arrow[ld,"i_2"]\\
&\gm^3C_3\arrow[r,"\gm^3d_1"]\arrow[d,"\gm^3d_2"',"\phantom{aaaa}(A)"]&\gm^3C_2\arrow[r,"\gm \mu_{C_2}"]\arrow[d,"\gm^3d_1\phantom{aa}(B)"]&\gm^2C_2\arrow[d,"\gm^2d_1"]&\\
&\gm^3C_2\arrow[r,"\gm^3d_1"]\arrow[d,"\mu_{\gm C_2}"',"\phantom{aaaa}(C)"]&\gm^3C_1\arrow[r,"\gm \mu_{C_1}"]\arrow[d,"\mu_{\gm C_1}\phantom{aa}(D)"]&\gm^2C_1\arrow[d,"\mu_{C_1}"]&\\
&\gm^2C_2\arrow[r,"\gm^2d_1"]&\gm^2C_1\arrow[r,"\mu_{C_1}"]&\gm C_1&\\
\tilde{C}_2\arrow[rrrr,"\tilde{d}_1"]\arrow[ru,"i_2"]&&&&\tilde{C}_1,\arrow[lu,tail,"i_1"]\\
\end{tikzcd}
\end{equation*}
  %--------------------------------------------------------------
where the four trapeziums are diagrams (\ref{dgm:Cat_TMulticat_C2d1restriction}) and (\ref{dgm:tilderestriction}) of Lemma  \ref{lemma:tilderestriction}. The inner squares are the following: $(A)$ is $\gm^3$ applied to associativity of $C$; $(B)$ is $\gm$ applied to naturality of $\mu$ at $d_1$; $(C)$ is naturality of $\mu$ at $\gm d_1$ and $(D)$ is the associativity law of $\mu$. Since $i_1$ is a monomorphism (\ref{dgm:Cat_TMulticat_Def}) and all the inner diagrams commute, so does the outer square, as we wanted to see.
\end{proof}
\end{proposition}

%-----------------------------------------------------------------------------------------------------------------------------------% %-------------------------------------------------------------unit-------------------------------------------------------------%
%-----------------------------------------------------------------------------------------------------------------------------------%
%\subsubsection*{Unit of $\tilde{C}$}
The unit morphism of $\Tpc{}{\gm}{C}$ is easier to obtain than composition. Recall that the unit is a morphism $\widetilde{e}\colon C_0\rightarrow \widetilde{C}_1$ such that the following diagram (\ref{PCat_topbottom}) commutes,
%---D-------I------A-------G-----R---------A-------M-------
 \begin{center}
 \begin{tikzcd}[column sep=small]
 & C_0\ar[d,"\widetilde{e}"]\ar[dl,"\eta_{C_0}"']\ar[dr,"\id"] & \\
 \gm C_0 &\widetilde{C}_1\ar[l,"d_1"]\ar[r,"d_0"']& C_0.
 \end{tikzcd}
 \end{center}
 %--------------------------------------------------------------
 \begin{definition}\label{def:tildeunit} The unit of $\Tpc{}{\gm}{C}$ is given by the following arrow:
 %---D-------I------A-------G-----R---------A-------M-------
 \begin{equation}\label{dgm:Cat_TMulticat_unit_def}
  \begin{tikzcd}[row sep=15pt]
 \pmb{C_0\arrow[dd,"\tilde{e}"]}& \gm C_0\arrow[dd,equal,"\phantom{aaaa}(A)"]& C_0\arrow[l,"\eta_{C_0}"]\arrow[r,equal]\arrow[d,"\eta_{C_0}"]& C_0\arrow[r,equal]\arrow[dd,"(B)\phantom{aaaa}"',"\eta_{C_0}\phantom{aaa}(C)"]& C_0\arrow[r,equal]\arrow[dd,"F_{C_0}"]& C_0\arrow[dd,equal,"(D)\phantom{aaa}"']&\\
& &\gm C_0\arrow[d,"\gm e"']&&&&\\
 \pmb{\tilde{C}_1}&\gm C_0 &\gm C_1\arrow[l,"\gm d_1"]\arrow[r,"\gm d_0"']&\gm C_0&\Tx C_0\arrow[l,"D_{C_0}"]\arrow[r,"R_{C_0}"']&C_0.&
 \end{tikzcd}
%  \begin{tikzcd}[row sep=15pt]
% C_0\arrow[r,equal]\arrow[dd,"\eta_{C_0}"',"\phantom{aaaa}(A)"]& C_0\arrow[r,equal]\arrow[d,"\eta_{C_0}"]& C_0\arrow[r,equal]\arrow[dd,"(B)\phantom{aaaa}"',"\eta_{C_0}\phantom{aaa}(C)"]& C_0\arrow[r,equal]\arrow[dd,"F_{C_0}"]& C_0\arrow[dd,"id","(D)\phantom{aaa}"']\\
% &\gm C_0\arrow[d,"\gm e"']&&&\\
% \gm C_0 &\gm C_1\arrow[l,"\gm d_1"]\arrow[r,"\gm d_0"']&\gm C_0&\Tx C_0\arrow[l,"D_{C_0}"]\arrow[r,"R_{C_0}"']&C_0.
% \end{tikzcd}
 \end{equation}
 %--------------------------------------------------------------
 It is clear that all the diagrams commute: $(A)$ and $(B)$ come from $\gm$ applied to (\ref{Cat_d1e}) and (\ref{Cat_d0e}), this is $d_1\circ e=\id=d_0\circ e$; $(C)$ is the  compatibility between $D$ and $\eta$ (\ref{dgm:StrongMonad_AxiomEta}), and $(D)$ is obvious from the definitions of $R$ and $F$.
 \end{definition}

% \subsubsection*{Composition with the unit}
 We have to verify now that composition with the unit morphism is the identity (\ref{dgm:PCat_Unit}).
%  This means that the triangles 
%  \begin{equation}\label{dgm:Cat_TMulticat_unit_axioms}
%\input{diagrams/Cat_TMulticat_unit_axioms}
% \end{equation}
% commute, as in (\ref{dgm:PCat_Unit}). 
 To prove it we follow the same strategy as for associativity. That is, we project the diagrams into diagrams in the original category $C$ containing the corresponding unit axioms. Again, the proof of the following lemma can be found in \cite{Cebrian2}. Recall first that 
 $$C_2:=C_1\cx_{C_0}C_1\;\;\;\text{and}\;\;\;\tilde{C}_2:=\gm \tilde{C}_1\lcx{\gm C_0}\tilde{C}_1.$$
 \begin{lemma} \label{lemma:tildeerestriction}We have commutative squares
 \vspace{-10pt}
 %---D-------I------A-------G-----R---------A-------M-------
  \begin{figure}[H]
\centering
 \begin{subequations}
\begin{minipage}[b]{0.45\linewidth}
 \begin{equation}\label{dgm:tildeerestrictiona}
 \begin{tikzcd}[column sep=2em]
\gm C_0\lcx{\gm C_0} \tilde{C}_1\arrow[dd,"\gm\tilde{e}\cx_{\id}\id"]\arrow[r,"i_1^l"]&\gm C_0 \lcx{\gm C_0}\gm C_1\arrow[d,"\gm e\cx_{\id}\id"']\\
 & \gm C_2\arrow[d,"\gm\eta_{C_2}"']\\
 \tilde{C}_2\arrow[r,"i_2"']&\gm^2C_2,
 \end{tikzcd}
 \end{equation}
\end{minipage}
\quad
\begin{minipage}[b]{0.45\linewidth}
\begin{equation}\label{dgm:tildeerestrictionb}
 \begin{tikzcd}
  \tilde{C}_1\lcx{\,C_0\,}C_0\arrow[dd,"\eta_{\tilde{C}_1}\cx_{\eta_{C_0}}\tilde{e}"]\arrow[r,"i_1^r"]&\gm C_1 \lcx{\gm C_0}\gm C_0\arrow[d,"\id\cx_{\id}\gm e"']\\
 & \gm C_2\arrow[d,"\eta_{\gm C_2}"']\\
 \tilde{C}_2\arrow[r,"i_2"']&\gm^2C_2,
 \end{tikzcd}
 \end{equation}
 \end{minipage}
 \end{subequations}
\end{figure}
%--------------------------------------------------------------
\noindent where $i_1^l$ and $i_1^r$ are the morphisms corresponding to $i_1$.
 \end{lemma}
 
 \begin{proposition} The unit morphism $\tilde{e}$ of $\Tpc{}{\gm}{C}$ satisfies the left and right composition axioms (\ref{dgm:PCat_Unit}).
 \begin{proof}
 For the left composition (\ref{dgm:PCat_Unita}), the required commutative triangle is the outline of the diagram
%---D-------I------A-------G-----R---------A-------M-------
\begin{equation}\label{dgm:Cat_TMulticat_unit_axiom_leftproof}
\begin{tikzcd}[column sep=60pt]
\gm C_0\lcx{\gm C_0}\tilde{C}_1 \arrow[rrrr, "\gm \tilde{e}\cx_{\id}\id"] \arrow[rd, "i_1^r"] \arrow[rrdddd, bend right] \arrow[rrrd,white, "(B)" black]&&&&\tilde{C}_2 \arrow[ld,"i_2"'] \arrow[lldddd, "\tilde{d}_1", bend left=35] \\                                                                                                                     & \gm C_0\lcx{\gm C_0}\gm C_1 \arrow[r, "\gm e\cx_{\id}\id"] \arrow[rdd, "(A)\phantom{aaaaaa}p_2"'] & \gm C_2 \arrow[d, "(C)\phantom{aaa}\id"'] \arrow[r,"\gm \eta_{ C_2}"] \arrow[d,white,bend left ,"(E)" black]& \gm ^2C_2 \arrow[dl, "\mu_{C_2}"] \arrow[d, "\gm ^2d_1"] &\\
&&\gm C_2\arrow[d,"\gm d_1"]&\gm ^2C_1 \arrow[dl, "\mu_{C_1}\phantom{aa}(F)"]\arrow[l,white,"(D)" black]&\\
&& \gm C_1& &\\
&&\tilde{C}_1 \arrow[u, tail,"i_1"]&&                                                                                                                      
\end{tikzcd}
\end{equation}
%-----------------------------------------------------------------
 We have that Diagram $(A)$ commutes by definition of $i_1^l$;  $(B)$ is precisely (\ref{dgm:tildeerestrictiona}) of Lemma \ref{lemma:tildeerestriction}; $(C)$ is $\gm$ applied to the left composition with the unit axiom in the category $C$ (\ref{dgm:Cat_Unita}); $(D)$ is naturality of $\mu$ at $d_1$; $(E)$ is $\gm$ of the unit axiom of $\gm$ applied to $C_2$, and $(F)$ is the same as (\ref{dgm:Cat_TMulticat_C2d1restriction}). Since $i_1$ is a monomorphism and all the inner diagrams commute so does the outer triangle, as we wanted to see.
 
 For the right composition (\ref{dgm:PCat_Unitb}), the required commutative triangle is the outline of the diagram
 %---D-------I------A-------G-----R---------A-------M-------
  \begin{equation}\label{dgm:Cat_TMulticat_unit_axiom_rightproof}
 \begin{tikzcd}[column sep=60pt]
 \tilde{C}_1\lcx{\,C_0\,}C_0 \arrow[rrrr, "\eta_{\tilde{C}_1}\cx_{\eta_{C_0}}\tilde{e}"] \arrow[rd,"i_1^l"] \arrow[rrdddd, bend right]\arrow[rrrd,white, "(B)" black] \arrow[rrdddd, white, "(A)"' black]&   &  &  &\tilde{C}_2\arrow[ld,"i_2"'] \arrow[lldddd, "\tilde{d}_1", bend left=40] \\
 & \gm C_1\lcx{\gm C_0}\gm C_0 \arrow[r, "\id\cx_{\id}\gm e"] \arrow[rdd] & \arrow[d,white,bend left ,"(E)" black] \gm C_2 \arrow[d,"(C)\phantom{aaa}\id"'] \arrow[r,"\eta_{\gm C_2}"]& \gm ^2C_2 \arrow[dl, "\mu_{C_2}"] \arrow[d, "\gm ^2d_1"] & \\
  &&\gm C_2\arrow[d,"\gm d_1"]&\gm ^2C_1 \arrow[dl, "\mu_{C_1}\phantom{aaa}(F)"] \arrow[l,white,"(D)" black]&\\
 & & \gm C_1  &  &   \\
  &  & \tilde{C}_1 \arrow[u, tail,"i_1"]    &     &                                                                                                                        
\end{tikzcd}
 \end{equation}
 %-----------------------------------------------------------------
  We have that Diagram $(A)$ commutes by definition of $i_1^r$,  $(B)$ is precisely (\ref{dgm:tildeerestrictiona}) of Lemma \ref{lemma:tildeerestriction}, $(C)$ is $\gm$ applied to the right composition with the unit axiom in the category $C$ (\ref{dgm:Cat_Unitb}); $(D)$ is again naturality of $\mu$ at $d_1$, $(E)$ is the unit axiom of $\gm$ applied to $\gm C_2$ and $(F)$ is the same as (\ref{dgm:Cat_TMulticat_C2d1restriction}), as before. Since $i_1$ is a monomorphism and all the inner diagrams commute so does the outer triangle, as we wanted to see.
  \end{proof}
 \end{proposition}
 
%-----------------------------------------------------------------------------------------------------------------------------------% %-------------------------------------------------------------Functoriality----------------------------------------------------------%
%-----------------------------------------------------------------------------------------------------------------------------------% 

%\subsubsection*{Functoriality}
The last thing to check is that the construction is functorial. First of all we have to specify how the construction acts on morphisms. Let $C$ and $C'$ be two categories and $C\xrightarrow{f} B$ a functor, that is a diagram
%---D-------I------A-------G-----R---------A-------M-------
 \begin{equation*}
 \begin{tikzcd}
 C_0 \arrow[d,"f_0"']&C_1\arrow[l,"d_1"']\arrow[r,"d_0"]\arrow[d,"f_1"]&C_0\arrow[d,"f_0"]\\
 B_0&B_1\arrow[l,"d_1"']\arrow[r,"d_0"]&B_0
 \end{tikzcd}
 \end{equation*}
 %--------------------------------------------------------------
 satisfying the commutative squares of \ref{eq:Poperadsmorphism}. Then $\Tpc{}{\gm}{f}$ is the morphism given by
 %---D-------I------A-------G-----R---------A-------M-------
  \begin{equation*}
  \begin{tikzcd}[row sep=20pt]
 \pmb{\Tpc{}{\gm}{C}\arrow[d,"\Tpc{}{\gm}{f}"]}&\gm C_0 \arrow[d,"\gm f_0"]&\gm C_1\arrow[l,"\gm d_1"']\arrow[r,"\gm d_0"]\arrow[d,"\gm f_1"]&\gm C_0\arrow[d,"\gm f_0"]&\Tx C_0\arrow[l,"D_{C_0}"']\arrow[r,"R_{C_0}"]\arrow[d,"\Tx f_0"]&C_0\arrow[d,"f_0"]\\
 \pmb{\Tpc{}{\gm}{B}}&\gm B_0 &\gm B_1\arrow[l,"\gm d_1"']\arrow[r,"\gm d_0"]&\gm B_0&\Tx B_0\arrow[l,"D_{B_0}"']\arrow[r,"R_{B_0}"]&B_0.
 \end{tikzcd}
 \end{equation*}
 %--------------------------------------------------------------
 It is a bit tedious but not difficult to see that $\tilde{f}$ satisfies again the commutative squares of \ref{eq:Poperadsmorphism} \cite{Cebrian2}. Moreover, given another morphism $B\xrightarrow{g} A$ it is clear that $\Tpc{}{\gm}{(g\circ f)}=\Tpc{}{\gm}{g}\circ \Tpc{}{\gm}{f}$, just because of the functoriality of $\gm$ and $\Tx$.
 
 Since the construction is functorial, if the strength $D_{A}$ is a monomorphism for every object  $A\in\ambcat$ then $\Tpc{}{\gm}{}$ is in fact a functor from categories internal to $\ambcat$ to $\gm$-operads. 
%%%%%%%%%%%%%%%%%%%%%%%%%%%%%%%%%%%%%%%%%%%%%%%%%%%
%%%%%%%%%%%%%%%%%%%%%%%%%%%%%%%%%%%%%%%%%%%%%%%%%%%
%-----------------------------------------------------------------------------------------------------------------------------------%
%----------------------------------------Section: T-Operads to Categories-----------------------------------------------%
%-----------------------------------------------------------------------------------------------------------------------------------%

 \subsection{From $\gm$-operads to categories}\label{PoperadsToCat}
%-----------------------------------------------------------------------------------------------------------------------------------% %-------------------------------------------------------------definition-----------------------------------------------------------%
%------------------------------------------------------------------------------------------------------------------------------------%
This construction has a similar structure as the construction above, so we  follow the same steps. Let  $\pop$ be a $\gm$-operad, 
%---D-------I------A-------G-----R---------A-------M-------
 \begin{center}
\begin{tikzcd}[column sep=small]
       & Q_1\arrow[dl,"d_1"']\arrow[dr,"d_0"] &       \\
 \gm Q_0 &                                            & Q_0
 \end{tikzcd}
 \begin{tikzcd}
\gm Q_1\lcx{\gm Q_0}Q_1:=Q_2 \arrow[r,"d_1"]& Q_1\\
 \phantom{aaaaa}Q_0 \arrow[r,"e"] &Q_1,
 \end{tikzcd}
  \end{center}
%--------------------------------------------------------------
and assume that $D_{\pop_0}:\gm 1\cx \pop\rightarrow \gm\pop$ is a monomorphism. We construct a category $\Tpc{\gm}{}{\pop}$\label{TpcQC} from the $\gm$-operad $\pop$. In this case, the simplicial nomenclature for $\Tpc{\gm}{}{\pop}$ is $\bar{Q}_i$ for the simplices and $\bar{d}_i$ for the face maps. The following pullback defines the objects and arrows of $\Tpc{\gm}{}{\pop}$:
%---D-------I------A-------G-----R---------A-------M-------
\begin{equation}\label{dgm:TMulticat_Cat_Def}
\begin{tikzcd}[sep={4em,between origins}]
 && \bar{C}_1 \ar[dl] \ar[dr,tail,"j_1"] \arrow[rrdd,bend left, "\bar{d}_0"]\arrow[lldd,bend right, "\bar{d}_1"']\dpbk &&\\
 & \Tx C_0\ar[dl,"R_{C_0}"]\ar[dr,tail,"D_{C_0}"] & & C_1 \ar[dl,"d_1"'] \ar[dr,"d_0"']& \\
 C_0 && \gm C_0 && C_0.
 \end{tikzcd}
 \end{equation}
 %--------------------------------------------------------------
 Observe that $\bar{\pop}_0=\pop_0$, so that again $\Tpc{\gm}{}{\pop}$ has the same objects as $\pop$. Besides, the morphism $j_1$ is a monomorphism, since monomorphisms are preserved by pullbacks and $D_{\pop_0}$ is a monomorphism.
   %-----------------------------------------------------------------------------------------------------------------------------------% %-------------------------------------------------------------composition-------------------------------------------------------------%
%-----------------------------------------------------------------------------------------------------------------------------------

% \subsubsection*{Composition in $\bar{\pop}$}
To define composition we need to define a map $\bar{\pop}_2\xrightarrow{\bar{d}_1} \bar{\pop}_1$, where $\bar{\pop}_2:=\bar{\pop}_1\cx_{\pop_0}\bar{\pop}_1$, satisfying the axioms of Appendix \ref{PCat_Axioms}. 
%Hence, we need to give an arrow
% \begin{equation*}
% \tiny
% \begin{tikzcd}[row sep=0pt,column sep=30pt]
% C_0\arrow[ddrr,"id"']&\Tx C_0\arrow[r,"D_{C_0}"]\arrow[l,"R_{C_0}"']&\gm C_0& C_1\arrow[l,"d_1"']\arrow[r,"d_0"]&\gm C_0&\Tx C_0\arrow[r,"D_{C_0}"]\arrow[l,"R_{C_0}"']&\gm C_0& C_1\arrow[l,"d_1"']\arrow[r,"d_0"]&C_0\arrow[ddll,"id"]\\
%&&&& \arrowdown{$\bar{d}_1$}&&&&\\
%&&C_0&\gm C_0\arrow[r,"D_{C_0}"']\arrow[l,"R_{C_0}"]&\gm C_0& C_1\arrow[l,"d_1"]\arrow[r,"d_0"']&\gm C_0&&
% \end{tikzcd}
% \end{equation*}
\noindent However, to specify this map we need to express it in a way we can naturally use composition in the original $\gm$-operad $\pop_2\xrightarrow{d_1} \pop_1$. 
The following diagram represents an isomorphism 
 $$\bar{\pop}_2\cong  \Tx^2\pop_0\lcx{\gm\Tx \pop_0}\Tx \pop_1\lcx{\gm \pop_0}\pop_1=:\bar{\pop}_2',$$
 %---D-------I------A-------G-----R---------A-------M-------
 \begin{equation}\label{dgm:barC2barC2prima}
 \tiny
 \begin{tikzcd}[row sep=30pt,column sep=30pt]
 \pop_0\arrow[d,equal]&\Tx^2 \pop_0 \arrow[d,"R_{\Tx \pop_0}"']\arrow[r,"\Tx D_{\pop_0}"]\arrow[l]\rdpbk&\Tx \gm \pop_0 \arrow[d,"R_{\gm \pop_0}","(A)\phantom{aaaaa}"']&&\Tx \pop_1\arrow[ll,"\Tx d_1"']\arrow[rr]\arrow[dl,"R_{\pop_1}"',"\phantom{aaaa}(B)"]\arrow[rd,"\Tx d_0"]\dpbks&&\gm \pop_0\arrow[d,equal]&\pop_1 \arrow[l,"d_1"']\arrow[r,"d_0"]\arrow[d,equal]&\pop_0.\arrow[d,equal]\\
 \pop_0&\Tx \pop_0\arrow[r,"D_{\pop_0}"']\arrow[l,"R_{\pop_0}"]&\gm \pop_0& \pop_1\arrow[l,"d_1"]\arrow[r,"d_0"']& \pop_0&\Tx \pop_0\arrow[r,"D_{\pop_0}"']\arrow[l,"R_{\pop_0}"]&\gm \pop_0& \pop_1\arrow[l,"d_1"]\arrow[r,"d_0"']&\pop_0
 \end{tikzcd}
 \end{equation}
 %--------------------------------------------------------------
 It is clear that all the squares in (\ref{dgm:barC2barC2prima}) commute. Moreover, the squares $(A)$ and $(B)$ are cartesian because so is $R$.
 
 \begin{definition} The composition of $\bar{\pop}$ is given by the following arrow $\bar{\pop}_2'\xrightarrow{\bar{d}_1'} \bar{\pop}_1$,
 %---D-------I------A-------G-----R---------A-------M-------
  \begin{equation}\label{dgm:barCcomposition}
 \tiny
\begin{tikzcd}[row sep=30pt,column sep=30pt]
\pmb{\bar{\pop}_2'\arrow[ddd]}&\pop_0\arrow[ddd,equal]& \Tx^2 \pop_0 \arrow[d,"\id\cx D_{\gm1}"']\arrow[r,"\Tx D_{\pop_0}"]\arrow[l]&\Tx \gm \pop_0 \arrow[d,"(A)\phantom{aaa}D_{\gm \pop_0}"']&\Tx \pop_1\arrow[l,"\Tx d_1"']\arrow[r]\arrow[d,"D_{\pop_1}","(B)\phantom{aaaaa}"']&\gm \pop_0\arrow[d,equal]&\pop_1 \arrow[l,"d_1"']\arrow[d,equal]\arrow[r,"d_0"]&\pop_0\arrow[d,equal]\\
&&\pop_0\cx\gm^21\arrow[r,"D^2_{\pop_0}"']\arrow[dd,"\id\cx\mu_1"',"\phantom{aaaaa}(C)"]&\gm^2\pop_0\arrow[dd,"\mu_{\pop_0}"]&\gm \pop_1\arrow[l,"\gm d_1"]\arrow[r,"\gm d_0"]&\gm \pop_0&\pop_1 \arrow[l,"d_1"']\arrow[r,"d_0"']&\pop_0\arrow[dd,equal]\\
&&&&(D)&\pop_2\arrow[ul,"d_2"]\arrow[ur,"d_0"']\upbks\arrow[d,"d_1"]&(E)&\\
\pmb{\bar{\pop}_1}&\pop_0&\gm \pop_0\arrow[r,"D_{\pop_0}"']\arrow[l,"R_{\pop_0}"]&\gm \pop_0& &\pop_1\arrow[ll,"d_1"]\arrow[rr,"d_0"']&&\gm \pop_0.
\end{tikzcd}
\end{equation}
%--------------------------------------------------------------
Let us see that all the diagrams commute: $(A)$ is a combination of natuarlity of $D$ applied to $D_{\pop_0}$ and axiom (\ref{dgm:StrengthConsecutiveApplications_Axiom}) concerning consecutive applications of the strength, $(B)$ is naturality of $D$ at $d_1$, $(C)$ is axiom (\ref{dgm:StrongMonad_AxiomMu}) for strong monads and $(D)$ and $(E)$ are respectively axioms (\ref{PCat_d1d1}) and (\ref{PCat_d0d1}) for composition in $\pop$. The remaining diagrams are clear.
\end{definition}
Notice that from this definition it is clear that $\bar{d}_1$ satisfies axioms (\ref{Cat_d1d1}) and (\ref{Cat_d0d1}). Furthermore, there is a morphism
\begin{center}\begin{tikzcd}\bar{\pop}_2'\arrow[r,"j_2'"] & \pop_2,\end{tikzcd}\end{center}
given by the diagram
%---D-------I------A-------G-----R---------A-------M-------
 \begin{equation*}
\begin{tikzcd}[row sep=30pt,column sep=30pt]
\pop_0& \Tx^2 \pop_0 \arrow[r,"\Tx D_{\pop_0}"]\arrow[l]&\Tx \gm \pop_0 \arrow[d,"D_{\gm \pop_0}"']&\Tx \pop_1\arrow[l,"\Tx d_1"']\arrow[r]\arrow[d,"D_{\pop_1}"]&\gm \pop_0\arrow[d,equal]&\pop_1 \arrow[l,"d_1"']\arrow[d,equal]\arrow[r,"d_0"]&\pop_0\arrow[d,equal]\\
&&\gm^2\pop_0&\gm \pop_1\arrow[l,"\gm d_1"]\arrow[r,"\gm d_0"]&\gm \pop_0&\pop_1 \arrow[l,"d_1"']\arrow[r,"d_0"']&\pop_0,
\end{tikzcd}
\end{equation*}
%--------------------------------------------------------------
which clearly makes the square
%---D-------I------A-------G-----R---------A-------M-------
\begin{equation}\label{dgm:TMulticat_Cat_C2d1restriction}
\begin{tikzcd}
\bar{\pop}_2' \arrow[d, "\bar{d}_1'"'] \arrow[r,"j_2'"] & \pop_2 \arrow[d, "d_1"]     \\
\bar{\pop}_1 \arrow[r, tail,"j_1"']                            &  \pop_1,                       
\end{tikzcd}
\phantom{aaa}
\begin{tikzcd}[row sep=2pt]
\text{and therefore}\\
\text{also the square}
\end{tikzcd}
\phantom{aaa}
\begin{tikzcd}
\bar{\pop}_2 \arrow[d, "\bar{d}_1"'] \arrow[r,"j_2"] & \pop_2 \arrow[d, "d_1"]     \\
\bar{\pop}_1 \arrow[r, tail,"j_1"']                            &  \pop_1,                       
\end{tikzcd}
\end{equation}
%--------------------------------------------------------------
commute, for the corresponding $j_2$. This says, roughly speaking, that composition in $\bar{\pop}$ is ``the same'' as composition in $\pop$, as is also clear in most of the examples.

 %-----------------------------------------------------------------------------------------------------------------------------------% %-------------------------------------------------------------associativity-------------------------------------------------------------%
%-----------------------------------------------------------------------------------------------------------------------------------%
%\subsubsection*{Associativity}
We have to check that composition is associative (\ref{dgm:Cat_Associativity}). As in the previous subsection, the proof of the following lemma is omitted; it can be read in \cite{Cebrian2}.
%As usual we define  $\tilde{C}_3$ as $\bar{C}_1\cx_{C_0}\bar{C}_1\cx_{C_0}\bar{C}_1$  with $\bar{d}_i$ its corresponding face maps for $i=0,1,2,3$.

\begin{lemma}  \label{lemma:barrestriction} There is a morphism $\bar{\pop}_3\xrightarrow{\;\;j_3\;\;} \pop_3$ such that the following diagrams commute
%---D-------I------A-------G-----R---------A-------M-------
\begin{equation}\label{dgm:barrestriction}
\begin{tikzcd}
\bar{\pop}_3 \arrow[d, "\bar{d}_1"'] \arrow[r,"j_3"] & \pop_3 \arrow[d, "d_1"]     \\
\bar{\pop}_2 \arrow[r,"j_2"']                            &  \pop_2,                       
\end{tikzcd}
\phantom{aaaaaaaaaaa}
\begin{tikzcd}
\bar{\pop}_3 \arrow[d, "\bar{d}_2"'] \arrow[r,"j_3"] & \pop_3 \arrow[d, "d_2"]     \\
\bar{\pop}_2 \arrow[r,"j_2"']                            &  \pop_2.                      
\end{tikzcd}
\end{equation}
%--------------------------------------------------------------
\end{lemma}
 \begin{proposition} Composition is associative. 
 \begin{proof}
 In view of Lemma \ref{lemma:barrestriction} there is a diagram
%---D-------I------A-------G-----R---------A-------M-------
 \begin{equation*}
  \begin{tikzcd}
\bar{C}_3\ar[rrr,"\bar{d}_2"]\ar[ddd,"\bar{d}_1"']\arrow[dr,"j_3"]&&&\bar{C}_2 \ar[ddd,"\bar{d}_1"]\arrow[dl,"j_2"]\\
 &C_3\arrow[r,"d_2"]\arrow[d,"d_1"']&C_2\arrow[d,"d_1"]\\
&C_2\arrow[r,"d_1"']&C_1\\
\bar{C}_2\ar[rrr,"\bar{d}_1"']\arrow[ur,"j_2"]& &&\bar{C}_1 \arrow[ul,tail,"j_1"].
 \end{tikzcd}
 \end{equation*}
  %--------------------------------------------------------------
 The four trapeziums are the commutative diagrams (\ref{dgm:TMulticat_Cat_C2d1restriction}) and (\ref{dgm:barrestriction}) of Lemma  \ref{lemma:barrestriction} respectively, and the inner square is associativity of composition in $C$ (\ref{dgm:PCat_Associativity}). Since $j_1$ is a monomorphism and all the inner diagrams commute, so does the outer square, as we wanted to see.
 \end{proof}
 \end{proposition}

 %-----------------------------------------------------------------------------------------------------------------------------------% %-------------------------------------------------------------unit-------------------------------------------------------------%
%-----------------------------------------------------------------------------------------------------------------------------------%
%\subsubsection*{Unit of $\bar{\pop}$}
 The unit morphism of the new category is easier to obtain than composition. Recall that the unit is a morphism $e\colon \pop_0\rightarrow \bar{\pop}_1$ such that this diagram (\ref{Cat_topbottom}) commutes
% \begin{center}
% \begin{tikzcd}[column sep=small]
% & C_0\ar[d,''e'']\ar[dl,"\eta_{C_0}"']\ar[dr,"id"] & \\
% C_0 &\overline{C}_1\ar[l,''\overline{s}'']\ar[r,''\widetilde{t}'']& C_0
% \end{tikzcd}
% \end{center}
%---D-------I------A-------G-----R---------A-------M-------
 \begin{center}
\begin{tikzcd}[column sep=small]
 & \pop_0\ar[d,"e"]\ar[dl,"\id"']\ar[dr,"\id"] & \\
 \pop_0 &\bar{\pop}_1\ar[l,"\bar{d}_1"]\ar[r,"\bar{d}_0"']& \pop_0.
 \end{tikzcd}
 \end{center}
 %--------------------------------------------------------------
 \begin{definition} The unit of $\bar{\pop}$ is given by the following arrow:
 %---D-------I------A-------G-----R---------A-------M-------
 \begin{equation}\label{dgm:TMulticat_Cat_unit_def}
 \begin{tikzcd}[row sep=20pt]
\pop_0\arrow[r,equal]\arrow[d,"\id"',"\phantom{aaaa}(A)"]& \pop_0\arrow[r,equal]\arrow[d,"F_{\pop_0}"]& \pop_0\arrow[r,equal]\arrow[d,"(B)\phantom{aaaa}"',"\eta_{\pop_0}\phantom{aaa}(C)"]& \pop_0\arrow[r,equal]\arrow[d,"e"]& \pop_0\arrow[d,"\id","(D)\phantom{aaa}"']\\
\pop_0&\Tx \pop_0\arrow[l,"R_{\pop_0}"]\arrow[r,"D_{\pop_0}"']&\gm \pop_0& \pop_1\arrow[l,"d_1"]\arrow[r,"d_0"']&\pop_0.
 \end{tikzcd}
 \end{equation}
 %--------------------------------------------------------------
 It is clear that all the diagrams commute: $(A)$ is obvious from the definitions of $R$ and $F$, $(B)$ is axiom (\ref{dgm:StrongMonad_AxiomEta}) for strong monads, and $(C)$ and $(D)$ are respectively the unit axioms (\ref{PCat_d0e}) and (\ref{PCat_d1e}) of $\pop$.
  \end{definition}
 
 % \subsubsection*{Composition with the unit}

 We have to check  that composition with the unit morphism is the identity (\ref{dgm:Cat_Unit}). To prove it we follow the same strategy as for associativity. That is, we project the diagrams into diagrams in the original $\gm$-operad $\pop$ containing the corresponding unit axioms. The proof of the following lemma can be found in \cite{Cebrian2}. Recall first that
$$\pop_2:=\gm\pop_1\lcx{\gm\pop_0}\pop_1\;\;\;\text{and}\;\;\; \bar{\pop}_2:=\bar{\pop}_1\cx_{\pop_0}\bar{\pop}_1.$$
 \begin{lemma} \label{lemma:barerestriction}We have commutative squares
 %---D-------I------A-------G-----R---------A-------M-------
  \begin{figure}[H]
\centering
 \begin{subequations}
\begin{minipage}[b]{0.45\linewidth}
 \begin{equation}\label{dgm:barerestrictiona}
 \begin{tikzcd}[row sep=20pt]
 \pop_0\lcx{\, \pop_0\,} \bar{\pop}_1\arrow[d,"\bar{e}\cx_{\id}\id"']\arrow[r,"j_1^l"]&\gm \pop_0 \lcx{\gm \pop_0}\pop_1\arrow[d,"Pe\cx_{\id}\id"']\\
 \bar{\pop}_2\arrow[r,"j_2"']&\pop_2,
 \end{tikzcd}
 \end{equation}
\end{minipage}
\quad
\begin{minipage}[b]{0.45\linewidth}
\begin{equation}\label{dgm:barerestrictionb}
 \begin{tikzcd}
  \bar{\pop}_1\lcx{\,\pop_0\,}\pop_0\arrow[d,"\id\cx_{\id}\bar{e}"']\arrow[r,"j_1^r"]& \pop_1 \lcx{\, \pop_0\,} \pop_0\arrow[d,"\eta_{\pop_1}\cx_{\eta_{\pop_0}} e"']\\
 \bar{\pop}_2\arrow[r,"j_2"']&\pop_2,
 \end{tikzcd}
 \end{equation}
 \end{minipage}
 \end{subequations}
\end{figure}
%--------------------------------------------------------------
\noindent where $j_1^l$ and $j_1^r$ are the morphisms corresponding to $j_1$.
 \end{lemma}

\begin{proposition} The unit morphism $\bar{e}$ of $\bar{\pop}$ satisfies the left and right composition axioms (\ref{dgm:Cat_Unit}).
 \begin{proof}
 For the left composition (\ref{dgm:Cat_Unita}), the required commutative triangle is the outline of the diagram
%---D-------I------A-------G-----R---------A-------M------- 
\begin{equation}\label{dgm:TMulticat_Cat_unit_axiom_leftproof}
\begin{tikzcd}[column sep=50pt]
\pop_0\lcx{\,\pop_0}\bar{\pop}_1 \arrow[rrrr, "\bar{e}\cx_{\id}\id"] \arrow[rd, "j_1^l\phantom{aaaaaaaaaaaaaaaaaaaa}(B)"] \arrow[rrddd,bend right,"\phantom{aaa}(A)" black] &   &\phantom{R}&&\bar{\pop}_2 \arrow[ld,"j_2"] \arrow[llddd, "\bar{d}_1", bend left] \\
& \gm \pop_0\lcx{\gm \pop_0}\pop_1 \arrow[rd,"\phantom{aaaaaa}(C)"] \arrow[rr, "\gm e\cx_{\id}\id"] && \gm \pop_1\lcx{\gm \pop_0}\pop_1 \arrow[ld, "d_1"] \arrow[ldd,white,"\phantom{aa}(D)" black]&\\
&& \pop_1&                                            &                                                                                                                   \\
                                                                                                                 &                                                              & \bar{\pop}_1,\arrow[u, tail,"j_1"] &                                            &                                                                                                                  
\end{tikzcd}
\end{equation}
 %--------------------------------------------------------------
We have that Diagram $(A)$ commutes by definition of $j_1^l$,  $(B)$ is precisely (\ref{dgm:barerestrictiona}) of Lemma \ref{lemma:barerestriction}, $(C)$ is the left composition with unit axiom in the $\gm$-operad $C$  (\ref{dgm:PCat_Unita})  and $(D)$ is the same as (\ref{dgm:TMulticat_Cat_C2d1restriction}). Since $j_1$ is a monomorphism and all the inner diagrams commute, so does the outer triangle, as we wanted to see.
 
 For the right composition (\ref{dgm:Cat_Unitb}), the required commutative triangle is the outline of the diagram 
%---D-------I------A-------G-----R---------A-------M------- 
\begin{equation}\label{dgm:TMulticat_Cat_unit_axiom_rightproof}
\begin{tikzcd}[column sep=50pt]
\bar{\pop}_1\lcx{\,\pop_0\,}\pop_0 \arrow[rrrr,"\id\cx_{\id}\bar{e}"] \arrow[rd,"j_1^r \phantom{aaaaaaaaaaaaaaaaaaaaaaa}(B)"] \arrow[rrddd, bend right,"\phantom{aaa}(A)" black] &&&& \bar{\pop}_2 \arrow[ld,"j_2"] \arrow[llddd, "\bar{d}_1", bend left] \\
& \pop_1\lcx{\;\pop_0\;}\pop_0 \arrow[rd,"\phantom{aaaaaa}(C)"] \arrow[rr, "\eta_{\pop_1}\cx_{\eta_{\pop_0}}e"] && \gm \pop_1\lcx{\gm \pop_0}\pop_1 \arrow[ld, "d_1"] \arrow[ldd,white,"\phantom{aa}(D)" black]&\\
&& \pop_1&&\\
&& \bar{\pop}_1, \arrow[u,tail,"j_1"] &&                                                                                                                  
\end{tikzcd}
 \end{equation}
%--------------------------------------------------------------
We have that Diagram $(A)$ commutes by definition of $j_1^r$,  $(B)$ is precisely (\ref{dgm:barerestrictionb}) of Lemma \ref{lemma:barerestriction}, $(C)$ is the right composition with unit axiom in the $\gm$-operad $\pop$ (\ref{dgm:PCat_Unitb}), and $(D)$ is the same as (\ref{dgm:TMulticat_Cat_C2d1restriction}), as before. Since $j_1$ is a monomorphism and all the inner diagrams commute so does the outer triangle, as we wanted to see.
 \end{proof}
 \end{proposition}

% \subsubsection*{Functoriality}
The last thing to check is that the construction is functorial. First of all we have to specify how the construction acts on morphisms. Let $\pop$ and $\pop'$ be two $\gm$-operads and $\pop\xrightarrow{f} B$ a morphism, that is a diagram
%---D-------I------A-------G-----R---------A-------M-------
 \begin{equation*}
 \begin{tikzcd}
 \gm \pop_0 \arrow[d,"\gm f_0"']&\pop_1\arrow[l,"d_1"']\arrow[r,"d_0"]\arrow[d,"f_1"]&\pop_0\arrow[d,"f_0"]\\
 \gm B_0&B_1\arrow[l,"d_1"']\arrow[r,"d_0"]&B_0
 \end{tikzcd}
 \end{equation*}
 %--------------------------------------------------------------
 satisfying the commutative squares of \ref{eq:Poperadsmorphism}. Then $\Tpc{\gm}{}{f}$ is the functor given by
%---D-------I------A-------G-----R---------A-------M-------
\begin{equation*}
\begin{tikzcd}[row sep=20pt]
\pmb{\Tpc{\gm}{}{\pop}\arrow[d,"\Tpc{\gm}{}{f}"]}&\pop_0\arrow[d,"f_0"']&\Tx \pop_0\arrow[l,"R_{\pop_0}"]\arrow[r,"D_{\pop_0}"']\arrow[d,"\Tx f_0"]&\gm \pop_0\arrow[d,"\gm f_0"]& \pop_1\arrow[l,"d_1"]\arrow[r,"d_0"']\arrow[d,"\gm f_1"]&\pop_0\arrow[d,"f_0"]\\
\pmb{\Tpc{\gm}{}{B}}&B_0&\Tx B_0\arrow[l,"R_{B_0}"]\arrow[r,"D_{B_0}"']&\gm B_0& B_1\arrow[l,"d_1"]\arrow[r,"d_0"']&B_0.
 \end{tikzcd}
 \end{equation*}
 %--------------------------------------------------------------
 It is a bit tedious but not difficult to see that $\Tpc{\gm}{}{f}$ satisfies again the commutative squares of \ref{eq:Poperadsmorphism} \cite{Cebrian2}. Moreover, given another morphism $B\xrightarrow{g} A$ it is clear that $\Tpc{\gm}{}{(g\circ f)}=\Tpc{\gm}{}{g}\circ \Tpc{\gm}{}{f}$, just because of the functoriality of $\gm$ and $\Tx$.

Since the construction is functorial, if the strength $D_{A}$ is a monomorphism for every object  $A\in\ambcat$ then $\Tpc{\gm}{}{}$ is in fact a functor from $\gm$-operads to categories internal to $\ambcat$. 
%-----------------------------------------------------------------------------------------------------------------------------------%
 %------------------------------Section: T-Multicat to T-multicat-----------------------------------------------%
  %-----------------------------------------------------------------------------------------------------------------------------------%

\subsection{The composite construction}
Since we have defined a construction from $\gm$-operads to categories and a construction from categories to $\gm$-operads, we obtain a composite construction from $\gm'$-operads to $\gm$-operads, for $\gm'$ and $\gm$ not necessarily the same monad. In particular, since a category is the same as an $\idm$-operad, the composite construction for $\gm'=\idm$ is the same as the functor from categories to $\gm$-operads. From now on we call $\tcons$-construction any of the three constructions, the context will suffice to distinguish, but we are mainly interested in landing on a $\gm$-operad, rather than a category.  To keep notation short, we denote by 
\label{TpcQQ}$$\Tpc{}{\gm}{\pop}:=\Tpc{}{\gm}{\Tpc{\gm'}{}{\pop}}$$
the composite construction that produces a $\gm$-operad from the $\gm'$-operad $\pop$. The monad $\gm'$ will be always clear from the context. 
%We may drop the superscript if $\gm=\idm$, since $\Tpc{\idm}{\gm}{}=\Tpc{}{\gm}{}$, and also if  $\gm$ is clear from the context. 

%-----------------------------------------------------------------------------------------------------------------------------------%
 %---------------------------------------Section: finiteness confditions-----------------------------------------------------%
  %-----------------------------------------------------------------------------------------------------------------------------------%

\subsection{Finiteness conditions}

In Section \ref{section:plethysmsandoperads} we will be interested in computing the incidence bialgebra of the bar construction of several $\gm$-operads in $\ambcat=\infgrpd$. Recall that to be able to take the homotopy cardinality, the bar construction has to be locally finite as a simplicial groupoid (in the sense of \cite{GKT:DSIAMI-2}). We now define the notion of locally finite operad (in the sense of (\cite{Kock-Weber}) in the setting of $\gm$-operads, which is the sufficient condition for its bar construction to be locally finite, and we give sufficient conditions on the $\tcons$-construction to preserve locally finiteness.

\begin{definition} A natural transformation is \emph{finite} if all its components are finite. A monad $(\gm,\mu,\eta)$ on $\infgrpd$ is \emph{locally finite} if $\mu$ and $\eta$ are finite natural transformations.
 A $\gm$-operad $\pop$ is locally finite if $\pop_1$ is locally finite, and the maps $d_1$ and $e$ are finite.
\end{definition}

In the special case of $\gm=\idm$, $\gm$-operads are just categories, and the notion of locally finite agrees with the standard notion. Notice that $\pop$ can be locally finite even if $\gm$ is not. The condition of $\gm$ being locally finite appears in the $\tcons$-construction.

\begin{example}\label{ex:asscommlocfinite} For a classical symmetric or nonsymmetric operad, the locally finiteness condition amounts to saying that every operation can be expressed as a composition of operations in a finite number of ways. For instance, the operads $\ass$ and $\comm$ are locally finite. For this it is important that nullary operations are excluded. The non-reduced versions, where there is a nullary operation, are {\em not} locally finite.
\end{example}

 The bar construction of $\pop$ is locally finite if  $\pop$ is locally finite and $\gm$ preserves locally finite groupoids and finite maps (see Section \ref{seccion:MonadsMulticategories}). Also, given another locally finite monad $\mathsf{R}$ on $\ambcat$ that preserves locally finite groupoids and finite maps, if there is a cartesian monad map $\gm\xRightarrow{\psi}\mathsf{R}$ with $\psi$ finite then the bar construction $\tsbar^{\mathsf{R}}$ is also locally finite. Let us see that the $\tcons$-construction interacts well with finiteness, as long as some simple conditions are satisfied.

\begin{lemma}\label{lemma:TCfinite} 
Let $\gm\colon\infgrpd\rightarrow\infgrpd$ be a locally finite strong monad that preserves locally finite groupoids, finite maps and fibrations. Assume moreover that  the strength $D$ is finite. Consider a locally finite category $C$ in $\infgrpd$ such that $C_0$ is discrete and $D_{C_0}$ is a monomorphism.
Then the $\gm$-operad $\Tpc{}{\gm}{C}$ is locally finite.
\begin{proof} Recall from Diagram~\eqref{dgm:Cat_TMulticat_Def} that $\tilde{C}_1$ is defined as the pullback 
%---D-------I------A-------G-----R---------A-------M-------
\begin{center}
\begin{tikzcd}
\tilde{C}_1\arrow[r]\arrow[d]\rdpbk & \Tx C_0\arrow[d,tail,"D_{C_0}"]\\
\gm C_1\arrow[r,"\gm d_0"']& \gm C_0.
\end{tikzcd}
\end{center}
%--------------------------------------------------------------
Notice that the pullback and the monomorphism refer to the $1$-categorical notions, while the finite map condition is a homotopy notion. 

Let us see first that $\tilde{C}_1$ is locally finite. Since $C_1$ is locally finite and $\gm$ preserves locally finite groupoids, $\gm C_1$ is locally finite. Now, an automorphism in  $\tilde{C}_1$ is a pair of automorphisms $(f,g)\in \gm C_1\cx \Tx C_0$ coinciding at $\gm C_0$, but there is only a finite number of $f$'s, since $\gm C_1$ is locally finite, and for each $f$ at most one $g$, since $D_{C_0}$ is a monomorphism.

We have to prove also that $\tilde{d}_1$ and $\tilde{e}$ are finite maps. This follows from their definitions,~\ref{def:tildecomp} and~\ref{def:tildeunit}: since $C_0$ is discrete, we have that $d_0$ is a fibration, and because $\gm$ (and also $\Tx$) preserves fibrations, all the right arrows  in diagrams~\eqref{dgm:tildeCcomposition} and~\eqref{dgm:Cat_TMulticat_unit_def} are fibrations. As a consequence their limit is equivalent to their homotopy limit. Finally, notice that all the vertical maps involved in these two diagrams are finite. This implies that their homotopy limit, and hence their limit, is also finite.
\end{proof}
\end{lemma}

\begin{lemma}\label{lemma:TQfinite}Let $\gm\colon\infgrpd\rightarrow\infgrpd$ be a locally finite strong monad that preserves locally finite groupoids, finite maps and fibrations. Assume moreover that  the strength $D$ is finite. Consider a locally finite  $\gm$-operad $\pop$ such that $\pop_0$ is discrete and $D_{\pop_0}$ is a monomorphism.
Then the $\gm$-operad $\Tpc{}{\gm}{C}$ is locally finite.
\begin{proof} The proof is analogous to the proof of Lemma~\ref{lemma:TCfinite}.
\end{proof}
\end{lemma}

In particular these results imply of course that if $\gm$ and $\gm'$ are monads satisfying the conditions of Lemmas~\ref{lemma:TCfinite} and~\ref{lemma:TQfinite} and 
$\pop$ is a locally finite $\gm'$-operad  then $\Tpc{}{\gm}{\pop}$ is locally finite.

\begin{remark}\label{rk:finiteness}
 In the sequel, we deal with the free semigroup monad $\fsg$ and the free symmetric semimonoidal category monad $\pfsmc$, which preserve locally finite groupoids, finite maps and fibrations, as required by Lemmas~\ref{lemma:TCfinite} and~\ref{lemma:TQfinite}. Moreover, their strength is finite, as can be easily seen from its definition (see Examples~\ref{example:plainoperadoperations} and~\ref{ex:fsmc}). Also, we use the reduced operads $\ass$ and $\comm$, as well as their colored versions. They are all locally finite and have discrete groupoid of colors.
\end{remark}

%-----------------------------------------------------------------------------------------------------------------------------------
%-----------------------------------------------------------EXAMPLES---------------------------------------------------------
%-----------------------------------------------------------EXAMPLES---------------------------------------------------------
%-----------------------------------------------------------EXAMPLES---------------------------------------------------------
%-----------------------------------------------------------EXAMPLES---------------------------------------------------------
%-----------------------------------------------------------EXAMPLES---------------------------------------------------------
%-----------------------------------------------------------EXAMPLES---------------------------------------------------------
%-----------------------------------------------------------EXAMPLES---------------------------------------------------------
%-----------------------------------------------------------------------------------------------------------------------------------

\section{$\tcons$-construction for $\fsg$ and $\pfsmc$-operads}\label{seccion:examples}
In this section we unravel the $\tcons$-construction with some of the main examples. We begin discussing the construction from categories to $\fsg$-operads and $\pfsmc$-operads. When the category is just a monoid we get the Giraudo $T$-construction, which we recall next.  Lastly we treat symmetric and nonsymmetric operads.

The choice of working with the reduced version of the operads (excluding nullary operations),
is irrelevant for the sake of the $\tcons$-construction itself, which is abstract enough to work with any operad.
The reason for preferring the reduced version is to stay within the realm of locally finite operads, as mentioned in Example~\ref{ex:asscommlocfinite} and Remark~\ref{rk:finiteness}. Moreover, it is also easy to see that the cartesian monad maps $\fsg\Rightarrow\fsmc$, $\pfsmc\Rightarrow\fsmc$ and $\fsg\Rightarrow\fm$ are finite.

\subsection{The $\tcons$-construction for categories} Let $C$ be a category internal to $\set$, represented by the span $C_0\leftarrow C_1\rightarrow C_0$, and take the free semigroup monad $\fsg$. The set of objects of $\Tpc{}{\fsg}{C}$ is again $C_0$, while $\tilde{C}_1$  is given by
%---D-------I------A-------G-----R---------A-------M-------
\begin{equation}
\begin{tikzcd}[sep={4em,between origins}]
 && \tilde{C}_1 \ar[dl,tail] \ar[dr] \arrow[ddll,bend right,"\tilde{d}_1"']\arrow[ddrr, bend left, "\tilde{d}_0"]\dpbk &&\\
 & \fsg C_1\ar[dl,"\gm d_1"]\ar[dr,"\gm d_0"] & & \Tx C_0 \ar[dl,tail,"D_{C_0}"'] \ar[dr,"R_{C_0}"']& \\
 \fsg C_0 && \fsg C_0 && C_0.
 \end{tikzcd}
\end{equation}
%--------------------------------------------------------------
Recall from Example \ref{ex:fsmc} that the strength is given by 
\begin{equation}\label{eq:strength1}
\begin{tikzcd}[row sep=3pt, column sep=40pt]
D_{C_0}\colon \Tx C_0 \arrow[r] & \fsg C_0\phantom{aaaaaaaaa}\\
\phantom{aaaaa}\big(c,(1,\dots,1)\big)\arrow[r] & \big((c,1),\dots,(c,1)\big).
\end{tikzcd}
\end{equation}
Therefore, the pullback condition means that the elements in $\tilde{C}_1$ that have input $c_1,\dots,c_n$ and output $c$ are the sequences of $n$ arrows in $C$ whose sources are $c_1,\dots,c_n$ and whose targets are all $c$. Hence 
$$\tilde{C}_1=\sum_{(c_1,\dots,c_n;c)}\prod_{i=1}^n\text{Hom}(c_i,c).$$
Substitution in $ \Tpc{}{\fsg}{C}$, 
$$\circ\colon \prod_{i=1}^k\text{Hom}(c_i,c)\times \prod_{i=1}^k\prod_{j=1}^{n_i}\text{Hom}(d^i_j,c_i)\longrightarrow \prod_{\substack{1\le i\le k\\1\le j \le n_i}} \text{Hom}(d^i_j,c),$$
goes as follows: for an operation $x\in \prod_{i=1}^k\text{Hom}(c_i,c)$ and a sequence of $k$ operations $y^i\in \prod_{j=1}^{n_i}\text{Hom}(d^i_j,c_i)$, with $i=1,\dots,k$,
$$x\circ (y^1,\dots,y^k)=(x_1\circ y^1_1,\dots, x_1\circ y^1_{n_1},\dots,x_k\circ y^k_1,\dots,x_k\circ y^k_{n_k})\in \prod_{\substack{1\le i\le k\\1\le j \le n_i}} \text{Hom}(d^i_j,c).$$
Note that now  the composition inside the parenthesis is composition of morphisms of $C$, while the composition on the left-hand side of the equation  is composition in $\fsg C$.  It is not difficult to see that the composition we get from \ref{def:tildecomp} agrees with the one defined above: both use the fact that $\tilde{C}_2$ is a subset of $(\fsg)^2 C$ together with $(\fsg)^2\circ$ and the monad multiplication. The identity elements of this operad are given by the identity morphisms of $C$. If the category $C$ has coproducts $(+)$ then 
$$\prod_{i=1}^n\text{Hom}(c_i,c)=\text{Hom}(c_1+\cdots+c_n,c),$$
so that the operations of $\Tpc{}{\fsg}{C}$ are in fact arrows of $C$.

Since $C$ can be considered as a category internal to $\infgrpd$, we can also compute $\Tpc{}{\pfsmc}{C}$ to get a symmetric operad. It is clear that $\Tpc{}{\fsg}{C}_1=\Tpc{}{\fsg}{C}_1//\mathfrak{S}$, where the action of the symmetric group $\mathfrak{S}_n$ is given by permutation of tuples, that is
\begin{center}
\begin{tabular}{cll}
  $\mathfrak{S}_n\times \prod_{i=1}^n\text{Hom}(c_i,c)$  & $\longrightarrow$ & $\prod_{i=1}^n\text{Hom}(c_{\sigma(i)};c)$  \\
$(\sigma,(x_1,\dots,x_n))$& $\longmapsto$ & $(x_{\sigma(1)},\dots,x_{\sigma(n)})$.
\end{tabular}
\end{center}

It is useful to picture elements $(c_1,\dots,c_n;c)$ as (picturing $n=3$)
%---D-------I------A-------G-----R---------A-------M-------
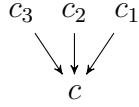
\begin{center}
\begin{tikzpicture}[<-,>=stealth',grow=up,level distance=30pt]
\tikzstyle{level 1}=[sibling distance=20pt]
%\tikzstyle{every node}=[fill=red!60,circle,inner sep=1pt]
\node at (0,0){$c$}
child {node{$c_1$}}
child {node{$c_2$}}
child {node{$c_3$}};
\end{tikzpicture}
\end{center}
%--------------------------------------------------------------
Under this representation, composition in $\Tpc{}{\pfsmc}{C}$ (or $\Tpc{}{\fsg}{C}$) looks like
%---D-------I------A-------G-----R---------A-------M-------
\begin{center}
\begin{tikzpicture}[<-,>=stealth',grow=up,level distance=30pt]
\tikzstyle{level 1}=[sibling distance=40pt]
\tikzstyle{level 2}=[sibling distance=20pt]
\node at (0,0){$c$}
child {node{$c_1$} child{node{$c^1_1$}} child{node{$c^2_1$}} child{node{$c^3_1$}}}
child {node{$c_2$} child{node{$c^1_2$}}}
child {node{$c_3$} child{node{$c^1_3$}} child{node{$c^2_3$}}};
\node at (3,1){$=$};
\end{tikzpicture}
\begin{tikzpicture}[<-,>=stealth',grow=up,level distance=50pt]
\tikzstyle{level 1}=[sibling distance=20pt]
\node at (0,0){$c$}
child {node{$c^1_1$}}
child {node{$c^2_1$}}
child {node{$c^3_1$}}
child {node{$c^1_2$}}
child {node{$c^1_3$}}
child {node{$c^2_3$}};
\end{tikzpicture}
\end{center}
%--------------------------------------------------------------

\begin{example}\label{ex:TassTsym}
 Take $C\!=\!\{\!\!\begin{tikzcd}0\!\arrow[r,bend left=20]&\!1\arrow[l,bend left=20]\end{tikzcd}\!\!\}$. For any pair of objects of $C$ there is exactly one morphism between them. Hence $\Tpc{}{\fsg}{C}$ has one operation for each given sequence of inputs and output, so that it is the  $2$-colored associative operad $\ass_2$. In the same way $\Tpc{}{\pfsmc}{C}$ is the $2$-colored symmetric operad $\comm_2$. In fact it is straightforward to see that the $\tcons$-constructions of the discrete connected groupoid of $n$ elements are $\ass_n$ and $\comm_n$.
\end{example} 

\begin{example}Consider the category $C=\{0\longrightarrow 1\}$.  Note that in this case there is either one or no morphism between two objects of $C$. Thus clearly
$$\Tpc{}{\pfsmc}{C}(c_1,\dots,c_n;c)=\begin{cases}
    (c\rightarrow c_1,\dots,c\rightarrow c_n) & \text{if } c=0\;\; \text{or } c=c_1=\dots=c_n=1 \\
    \emptyset & \text{otherwise. }
  \end{cases}$$
 Of course this operad is a suboperad of the previous one, since this category is a subcategory of the previous one. In particular composition is obvious.
 \end{example}
 
%A slight variation of this constructions comes from using the monad $\fm$ instead of $\fsg$ or the monad $\fsmc$ instead of $\pfsmc$. The only difference is that $\tilde{C}^{\fm}$ and $\tilde{C}^{\fsmc}$   have a nullary operation, unlike $\tilde{C}^{\fsg}$ and $\tilde{C}^{\pfsmc}$. 

\begin{example}\label{ex:giraudo}
We now specialize to the
case of categories with only one object,
that is monoids, recovering the $T$-construction of Giraudo. This construction was introduced by Giraudo \cite{Giraudo} as a generic method to build combinatorial operads from monoids.

Since a monoid is just a category with one object, it is represented by the span $1\leftarrow \monoid\rightarrow 1$,  and because the morphism $\Tx1\xrightarrow{D_1}\fsg1$ is an isomorphism, we have that $\Tpc{}{\fsg}{\monoid}$ is given by
%---D-------I------A-------G-----R---------A-------M-------
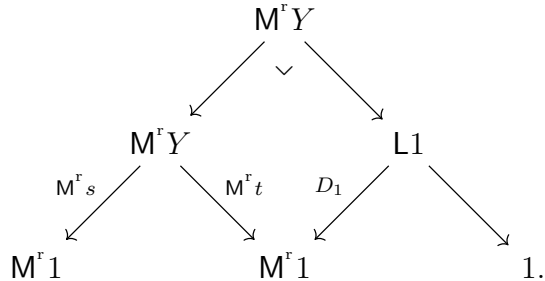
\begin{center}
\begin{tikzcd}[sep={4em,between origins}]
 &&\fsg \monoid \ar[dl] \ar[dr] \dpbk &&\\
 & \fsg \monoid\ar[dl,"\fsg s"']\ar[dr,"\fsg t"] & & \Tx 1 \ar[dl,"D_{1}"'] \ar[dr]& \\
 \fsg 1 && \fsg 1 && 1.
 \end{tikzcd}
\end{center}
%--------------------------------------------------------------
It is easy to see that this gives the same operad $T \monoid$ defined in the introduction, since $T\monoid$ is precisely $\fsg\monoid$, and both compositions are defined by using composition in $(\fsg)^2\monoid$ and the monad multiplitcation.

% It was already defined in the introduction: for $(\monoid,\cdot,1)$ a monoid, consider the operad
%$$T \monoid:=\bigsqcup_{n\ge1} T \monoid(n),$$
%where for all $n\ge 1$,
%$$T \monoid(n):=\{(x_1,\dots,x_n)\, |\, x_i\in \monoid\;\; \text{for all } i=1,\dots,n\}.$$
%The substitution operation in $T \monoid$,
%$$\circ_i:T \monoid(n)\times T \monoid(m)\longrightarrow T \monoid(n+m-1),$$
%is defined as follows: for all $x\in T \monoid(n)$, $y\in T \monoid(m)$, and $i=1,\dots, n$,
%$$x\circ_i y:=(x_1,\dots,x_{i-1},x_i\cdot y_1,\dots,x_i\cdot y_m,x_{i+1},\dots,x_n).$$
%

%Again, since $\monoid$ can be considered as a monoid in $\infgrpd$ we can compute $\Tpc{}{\fsg}{\monoid}$ too. It is clear that $\pfsmc \monoid=\fsg \monoid//\mathfrak{S}$, where the action of the symmetric group $\mathfrak{S}_n$ is given by permutation of tuples, that is
%\begin{center}
%\begin{tabular}{cll}
%  $\mathfrak{S}_n\times \fm \monoid(n)$  & $\longrightarrow$ & $\fm \monoid(n)$  \\
%$(\sigma,(x_1,\dots,x_n))$& $\longmapsto$ & $(x_{\sigma(1)},\dots,x_{\sigma(n)})$.
%\end{tabular}
%\end{center}
\end{example}
\begin{example}\label{ex:TassTsym1}If $\termon$ is the singleton monoid, then $\Tpc{}{\fsg}{\termon}=\ass$, the associative operad, and  $\Tpc{}{\pfsmc}{\termon}=\comm$, the commutative operad.
\end{example}

\subsection{The $\tcons$-construction for operads}
We now unravel the full $\tcons$-construction from nonsymmetric operads to $\pfsmc$-operads. As we already know, the first ones are the same as $\fsg$-operads in $\set$, but we view them as $\fsg$-operads in $\infgrpd$ with discrete groupoids of objects and arrows. At the end we will comment on other variations similar to this case, such as from  symmetric operads to $\pfsmc$-operads.

Let $\pop$ be an $\fsg$-operad represented by the span $\fsg \pop_0\leftarrow \pop_1 \rightarrow \pop_0$. Recall that elements of $\pop_1$ are depicted as
%---D-------I------A-------G-----R---------A-------M-------
 
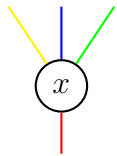
\begin{center}
 \begin{tikzpicture}[grow=up,level distance=30pt,thick]
\tikzstyle{level 1}=[sibling distance=20pt]
\node at (0,0){}
child[red] {node[black,circle,draw]{$x$}}{
child[green]
child[blue]
child[yellow]
};
\end{tikzpicture}
\end{center} 
%--------------------------------------------------------------
We apply first the $\tcons$-construction to get a category $\Tpc{\fsg}{}{Q}$:
%---D-------I------A-------G-----R---------A-------M-------
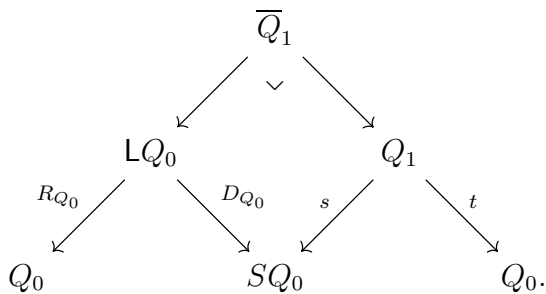
\begin{center}
 \begin{tikzcd}[sep={4em,between origins}]
 && \bar{\pop}_1 \ar[dl] \ar[dr] \dpbk &&\\
 & \Tx \pop_0\ar[dl,"R_{\pop_0}"']\ar[dr,"D_{\pop_0}"] & & \pop_1 \ar[dl,"s"'] \ar[dr,"t"]& \\
 \pop_0 && S\pop_0 && \pop_0.
 \end{tikzcd}
\end{center}
 %--------------------------------------------------------------

The strength morphism is the same as in (\ref{eq:strength1}). Therefore the elements of $\bar{\pop}_1$ are the elements of $\pop_1$ such that all the input objects coincide,
 
 %---D-------I------A-------G-----R---------A-------M-------
 
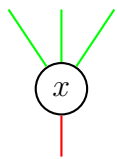
\begin{center}
 \begin{tikzpicture}[grow=up,level distance=30pt,thick]
\tikzstyle{level 1}=[sibling distance=20pt]
\node at (0,0){}
child[red] {node[black,circle,draw]{$x$}}{
child[green]
child[green]
child[green]
};
\end{tikzpicture}
\end{center}
 %--------------------------------------------------------------
so that $x$ is an arrow ${\color{green}\bullet}\xrightarrow{x}{\color{red}\bullet}$ in $\Tpc{\fsg}{}{\pop}$. Notice that $\bar{\pop}_2$ is a subset of $\pop_2$. Therefore composition in $\Tpc{\fsg}{}{\pop}$ is the same as composition in $\pop$. For example

 %---D-------I------A-------G-----R---------A-------M-------
\begin{equation}\label{op:TMQcomposition}
\begin{tikzpicture}[grow=up,level distance=30pt,thick]
\tikzstyle{level 1}=[sibling distance=20pt]
\tikzstyle{level 2}=[sibling distance=15pt]
\node at (0,0){}
child[red] {node[black,circle,draw]{$x$}}{
child[green]
child[green]
child[green]
};
\node at (0.85,1){$\circ$};
\node at (1.7,0){}
child[blue] {node[black,circle,draw,inner sep=3pt]{$y$}}{
child[red]
child[red]
};
\node at (2.5,1){$=$};
\tikzstyle{level 1}=[sibling distance=40pt]
\tikzstyle{level 2}=[sibling distance=15pt]
\node at (4,-0.4){}
child[blue] {node[black,circle,draw,inner sep=3pt]{$y$}}{
	child[red] {node[black,circle,draw]{$x$} child[green] child[green] child[green] }
	child[red] {node[black,circle,draw]{$x$} child[green] child[green] child[green]}
};
\node at (5.5,1){$=$};
\end{tikzpicture}
\begin{tikzpicture}[grow=up,level distance=45pt,thick]
\tikzstyle{level 1}=[sibling distance=20pt]
\node at (0,0){}
child[blue] {node[black,ellipse,draw]{$y\circ (x,x)$}}{
child[green]
child[green]
child[green]
child[green]
child[green]
child[green]
};
\end{tikzpicture}
\end{equation}

%--------------------------------------------------------------
where $y\circ (x,x)$ is composition in $\pop$. Hence the recipe is to repeat $x$ for each input of $y$ and use composition in $\pop$. Now we have to apply again the $\tcons$-construction to get a $\pfsmc$-operad from the category $\Tpc{\fsg}{}{\pop}$. This step was made above for any category: the objects of $\tilde{\pop}_1$ are sequences $(x_1,\dots,x_n)$ of elements $x_i\in \bar{\pop}_1$. For instance the pair
 %---D-------I------A-------G-----R---------A-------M-------
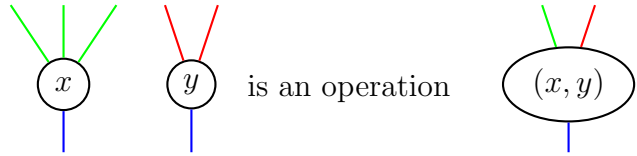
\begin{center}
\begin{tikzpicture}[grow=up,level distance=30pt,thick]
\tikzstyle{level 1}=[sibling distance=20pt]
\tikzstyle{level 2}=[sibling distance=15pt]
%\tikzstyle{every node}=[fill=red!60,circle,inner sep=1pt]
\node at (0,0){}
child[blue] {node[black,circle,draw]{$x$}}{
child[green]
child[green]
child[green]
};
\node at (1.7,0){}
child[blue] {node[black,circle,draw,inner sep=3pt]{$y$}}{
child[red]
child[red]
};
\node at (4,1){is an operation\phantom{aa}};
\end{tikzpicture}
\begin{tikzpicture}[grow=up,level distance=30pt,thick]
\tikzstyle{level 1}=[sibling distance=20pt]
\tikzstyle{level 2}=[sibling distance=15pt]
%\tikzstyle{every node}=[fill=red!60,circle,inner sep=1pt]
\node at (0,0){}
child[blue] {node[black,ellipse,draw,minimum size=20pt]{$(x,y)$}}{
child[red]
child[green]
};
\end{tikzpicture}
\end{center}
%--------------------------------------------------------------
in $\Tpc{}{\pfsmc}{\pop}$. Clearly $\Tpc{}{\pfsmc}{\pop}$ is a symmetric operad, since the groupoid of objects is discrete and the morphisms in the groupoid $\tilde{\pop}_1$ are given by permutation of tuples. 

\begin{example}\label{ex:Tass} If the starting $\fsg$-operad is $\ass$, which is a noncolored operad, then it is easy to see that the monoid $\Tpc{\fsg}{}{\ass}$ is isomorphic to $(\N^+,\times)$. Therefore the operations of $\Tpc{}{\pfsmc}{\ass}$ are sequences
of natural numbers and composition is given by multiplication. For example
$$\big((2,3),(4,7)\big)\circ (5,9)=(5\cdot 2,5\cdot 3,9\cdot 4,9\cdot 7)=(10,15,36,63).$$
If the starting $\fsg$-operad is $\ass_2$ the $2$-colored associative operad, then the category $\Tpc{\fsg}{}{\ass_2}$ has two objects and a morphism $\xrightarrow{n}$ for every pair of objects and positive natural number $n$. Composition is given by multiplication. The operations of $\Tpc{}{\pfsmc}{\ass_2}$ are thus sequences of such arrows with the same output.
\end{example}

Suppose we start instead from a symmetric operad $\pop$. Recall from Example~\ref{ex:fsmc} that a symmetric operad is an $\fsmc$-operad in $\infgrpd$ such that $\pop_0$ is discrete and $ \pfsmc \pop_0\xleftarrow{s}\pop_1$ is discrete fibration. The $\tcons$-construction to get another $\pfsmc$-operad is completely analogous to the previous case, but in this case the groupoid $\tilde{\pop}_1$ inherits morphisms from $\pop$, so that for instance the element 
%---D-------I------A-------G-----R---------A-------M-------
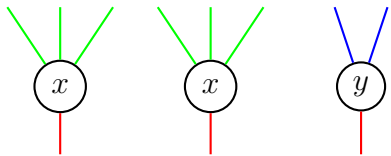
\begin{center}
\begin{tikzpicture}[grow=up,level distance=30pt,thick]
\tikzstyle{level 1}=[sibling distance=20pt]
\node at (0,0){}
child[red] {node[black,circle,draw]{$x$}}{
child[green]
child[green]
child[green]
};
\node at (2,0){}
child[red] {node[black,circle,draw]{$x$}}{
child[green]
child[green]
child[green]
};
\node at (4,0){}
child[red] {node[black,circle,draw,inner sep=3pt]{$y$}}{
child[blue]
child[blue]
};
\end{tikzpicture}
\end{center}
%--------------------------------------------------------------
has $2!\cdot 3!^2 \cdot 2!$ automorphisms, corresponding to $2!$ invariant permutations on $(x,x,y)$ and permutations of the inputs. The latter contribution did not appear in the previous case, since $\pop$ was a planar operad. Notice that this means that $\Tpc{}{\pfsmc}{\pop}$ is not a symmetric operad, but just an $\pfsmc$-operad in $\infgrpd$.

\begin{example}\label{ex:Tsym} If the starting $\pfsmc$-operad is $\comm$, which is a noncolored symmetric operad, then it is easy to see that the monoid $\Tpc{\pfsmc}{}{\comm}$ is isomorphic to the monoid $(\N^+,\times)$ internal to groupoids where $\aut (n)\cong \mathfrak{S}_n$. The objects of $\Tpc{}{\pfsmc}{\comm}$ are the same as the objects in $\Tpc{}{\pfsmc}{\ass}$, and the morphisms are given by permutation of tuples (as in $\Tpc{}{\pfsmc}{\ass}$) plus the ones given by $\aut (n)$ for each $n$. The colored case is analogous.
\end{example}

\subsection{The opposite convention}\label{subsection:oppositeconvention} 
It is not difficult to see that the categories $\Tpc{\pfsmc}{}{\ass}$ and $\Tpc{\pfsmc}{}{\comm}$ (as well as  $\Tpc{\fsg}{}{\ass}$ and $\Tpc{\fsg}{}{\comm}$) are self dual.  In the case of $\ass$ this means that the monoid $(\N^+,\times)$ is commutative. An obvious consequence of this self duality is that we get equivalent operads by applying the $\tcons$-construction to their opposite categories. Nevertheless, when dealing with plethysm it is in fact more natural, from a combinatorial point of view, to apply the $\tcons$-construction to the opposite categories. This is particularly apparent when we interpret the simplicial groupoid $T\sur$~\cite{Cebrian} as an operad (see Examples~\ref{ex:TSTSym2simplex} and ~\ref{ex:TS2BTSym2}). 

We end this chapter by  developing this variant in a general context, since we believe it is interesting in its own right. 
From a formal perspective there is not much to say, since in the context of internal categories if $C$ is represented by 
%---D-------I------A-------G-----R---------A-------M-------
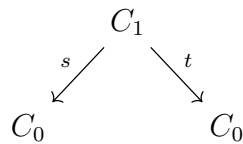
\begin{center}
  \begin{tikzcd}[column sep=small]
       & C_1\arrow[dl,"s"']\arrow[dr,"t"] &       \\
 C_0 &                                            & C_0
 \end{tikzcd}
 \end{center}
 then $C^{\text{op}}$ is represented by
  
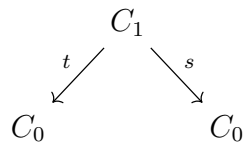
\begin{center} 
   \begin{tikzcd}[column sep=small]
       & C_1\arrow[dl,"t"']\arrow[dr,"s"] &       \\
 C_0 &                                            & C_0
 \end{tikzcd}
 \end{center}
 %--------------------------------------------------------------
and thus the $\tcons$-construction can be applied the same way. Let us see what $\Tpc{}{\pfsmc}{C^{\text{op}}}$ looks like.  We have that  
 $$\tilde{C_1^{\text{op}}}=\sum_{(c_1,\dots,c_n;c)}\prod_{i=1}^n\text{Hom}_{C^{\text{op}}}(c_i,c)=\sum_{(c_1,\dots,c_n;c)}\prod_{i=1}^n\text{Hom}_C(c,c_i),$$
 for each tuple $(c_1,\dots,c_n;c)$ of elements of $C_0$. 
In this case elements $(c_1,\dots,c_n;c)$ can be pictured as (picturing $n=3$)
%---D-------I------A-------G-----R---------A-------M-------
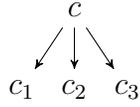
\begin{center}
\begin{tikzpicture}[->,>=stealth',grow=down,level distance=30pt]
\tikzstyle{level 1}=[sibling distance=20pt]
%\tikzstyle{every node}=[fill=red!60,circle,inner sep=1pt]
\node at (0,0){$c$}
child {node{$c_1$}}
child {node{$c_2$}}
child {node{$c_3$}};
\end{tikzpicture}
\end{center}
%--------------------------------------------------------------
and under this representation, composition in $\Tpc{}{\pfsmc}{C^{\text{op}}}$ looks like
%---D-------I------A-------G-----R---------A-------M-------
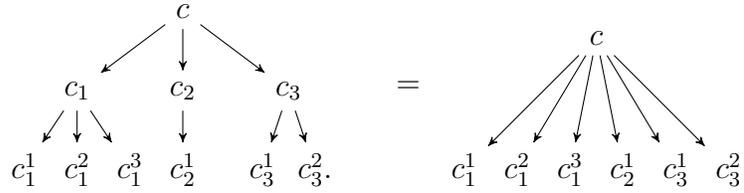
\begin{center}
\begin{tikzpicture}[->,>=stealth',grow=down,level distance=30pt]
\tikzstyle{level 1}=[sibling distance=40pt]
\tikzstyle{level 2}=[sibling distance=20pt]
\node at (0,0){$c$}
child {node{$c_1$} child{node{$c^1_1$}} child{node{$c^2_1$}} child{node{$c^3_1$}}}
child {node{$c_2$} child{node{$c^1_2$}}}
child {node{$c_3$} child{node{$c^1_3$}} child{node{$c^2_3$.}}};
\node at (3,-1){$=$};
\end{tikzpicture}
\begin{tikzpicture}[->,>=stealth',grow=down,level distance=50pt]
\tikzstyle{level 1}=[sibling distance=20pt]
\node at (0,0){$c$}
child {node{$c^1_1$}}
child {node{$c^2_1$}}
child {node{$c^3_1$}}
child {node{$c^1_2$}}
child {node{$c^1_3$}}
child {node{$c^2_3$}};
\end{tikzpicture}
\end{center}
%--------------------------------------------------------------
Furthermore, if $C$ has products then
  $$\prod_{i=1}^n\text{Hom}_{C}(c,c_i)=\text{Hom}_{C}(c,c_1\cx\cdots\cx c_n).$$

Suppose now that we start from an $\fsg$-operad $\pop$. The first step is the same as before: we obtain a category $\Tpc{\fsg}{}{\pop}$ whose arrows are operations of $\pop$
 %---D-------I------A-------G-----R---------A-------M-------
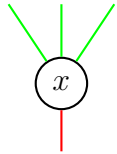
 \begin{center}
 \begin{tikzpicture}[grow=up,level distance=30pt,thick]
\tikzstyle{level 1}=[sibling distance=20pt]
\node at (0,0){}
child[red] {node[black,circle,draw]{$x$}}{
child[green]
child[green]
child[green]
};
\end{tikzpicture}
\end{center}
 %--------------------------------------------------------------
all of whose inputs coincide. Next, we take the opposite category $\Tpc{\fsg}{}{\pop}^{\text{op}}$, and depict its arrows as
%---D-------I------A-------G-----R---------A-------M-------
 
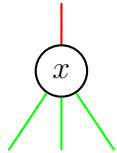
\begin{center}
 \begin{tikzpicture}[grow=down,level distance=30pt,thick]
 \tikzstyle{level 1}=[sibling distance=20pt]
\tikzstyle{level 2}=[sibling distance=15pt]
\node at (0,0){}
child[red] {node[black,circle,draw]{$x$}}{
child[green]
child[green]
child[green]
};
\end{tikzpicture}
\end{center}
%--------------------------------------------------------------
The $\tcons$-construction $\Tpc{}{\pfsmc}{\pop^{\text{op}}}$ has as operations sequences $(x_1,\dots,x_n)$ of arrows $x_i\in \bar{\pop_1^\text{op}}$ with the same output. For instance the pair 
 %---D-------I------A-------G-----R---------A-------M-------
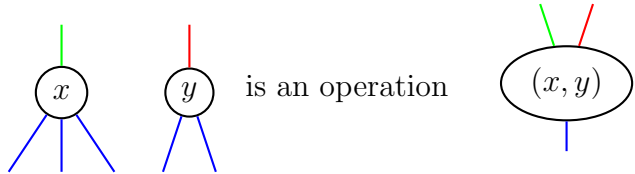
\begin{center}
\begin{tikzpicture}[grow=down,level distance=30pt,thick]
\tikzstyle{level 1}=[sibling distance=20pt]
\tikzstyle{level 2}=[sibling distance=15pt]
%\tikzstyle{every node}=[fill=red!60,circle,inner sep=1pt]
\node at (0,0){}
child[green] {node[black,circle,draw]{$x$}}{
child[blue]
child[blue]
child[blue]
};
\node at (1.7,0){}
child[red] {node[black,circle,draw,inner sep=3pt]{$y$}}{
child[blue]
child[blue]
};
\node at (4,-1){is an operation\phantom{aa}};
\end{tikzpicture}
\begin{tikzpicture}[grow=up,level distance=30pt,thick]
\tikzstyle{level 1}=[sibling distance=20pt]
\tikzstyle{level 2}=[sibling distance=15pt]
%\tikzstyle{every node}=[fill=red!60,circle,inner sep=1pt]
\node at (0,0){}
child[blue] {node[black,ellipse,draw,minimum size=20pt]{$(x,y)$}}{
child[red]
child[green]
};
\end{tikzpicture}
\end{center}
  %--------------------------------------------------------------
in $\Tpc{}{\pfsmc}{\pop^{\text{op}}}$. We show once and for all an example of composition in $\Tpc{}{\pfsmc}{\pop^{\text{op}}}$:
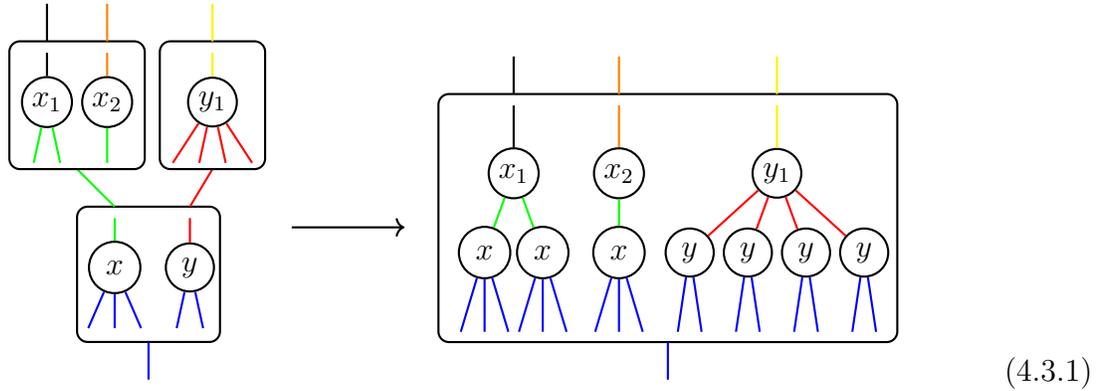
\begin{equation}\label{eq:TQopcomposition}
\begin{tikzpicture}[grow=down,level distance=23pt,thick]
\tikzstyle{level 1}=[sibling distance=10pt]
%\tikzstyle{level 2}=[sibling distance=15pt]
%\tikzstyle{every node}=[fill=red!60,circle,inner sep=1pt]
\node at (0,0){}
child[green] {node[black,circle,draw]{$x$}}{
child[blue]
child[blue]
child[blue]
};
\node at (1,0){}
child[red] {node[black,circle,draw,inner sep=3pt]{$y$}}{
child[blue]
child[blue]
};
\draw [rounded corners] (-0.5,-1.8)--(-0.5,0)--(1.4,0)--(1.4,-1.8)--cycle;
\draw [green](0,0)--(-0.5,0.5); \draw [red](1,0)--(1.3,0.5); \draw [blue](0.45,-1.8)--(0.45,-2.3);
\draw [rounded corners] (0.4,0.5)--(0.4,2.2)--(-1.4,2.2)--(-1.4,0.5)--cycle;
\draw [rounded corners] (0.6,0.5)--(0.6,2.2)--(2,2.2)--(2,0.5)--cycle;
\node at (-0.9,2.2){}
child {node[black,circle,draw,inner sep=2pt]{$x_1$}}{
child[green]
child[green]
};
\node at (-0.1,2.2){}
child [orange]{node[black,circle,draw,inner sep=2pt]{$x_2$}}{
child[green]
};
\node at (1.3,2.2){}
child [yellow]{node[black,circle,draw,inner sep=2pt]{$y_1$}}{
child[red]
child[red]
child[red]
child[red]
};
\draw (-0.9,2.2)--(-0.9,2.7); \draw [orange](-0.1,2.2)--(-0.1,2.7); \draw [yellow](1.3,2.2)--(1.3,2.7);
\end{tikzpicture}
\begin{tikzpicture}
\node at (-0.2,-1.9){};
\node at (1.8,0){};
\draw [->,thick](0,0)--(1.5,0);
\end{tikzpicture}
\begin{tikzpicture}[grow=down,level distance=30pt,thick]
\tikzstyle{level 1}=[sibling distance=22pt]
\tikzstyle{level 2}=[sibling distance=9pt]
\node at (0,0){}
child {node[black,circle,draw,inner sep=2pt]{$x_1$}}{
	child[green] {node[black,circle,draw]{$x$} child[blue] child[blue] child[blue] }
	child[green] {node[black,circle,draw]{$x$} child[blue] child[blue] child[blue]}
};
\node at (1.4,0){}
child[orange] {node[black,circle,draw,inner sep=2pt]{$x_2$}}{
	child[green] {node[black,circle,draw]{$x$} child[blue] child[blue] child[blue] }
};
\node at (3.5,0){}
child [yellow]{node[black,circle,draw,inner sep=2pt]{$y_1$}}{
	child[red] {node[black,circle,draw,inner sep=3pt]{$y$} child[blue] child[blue]}
	child[red] {node[black,circle,draw,inner sep=3pt]{$y$} child[blue] child[blue]}
	child[red] {node[black,circle,draw,inner sep=3pt]{$y$} child[blue] child[blue]}
	child[red] {node[black,circle,draw,inner sep=3pt]{$y$} child[blue] child[blue]}
};
\draw [rounded corners] (-1,-3.3)--(-1,0)--(5.1,0)--(5.1,-3.3)--cycle;
\draw (0,0)--(0,0.5); \draw [orange](1.4,0)--(1.4,0.5); \draw [yellow] (3.5,0)--(3.5,0.5);
\draw [blue](2.05,-3.3)--(2.05,-3.8);
\end{tikzpicture}
\end{equation}
and then of course we compose in $\pop$, as in~\eqref{op:TMQcomposition}.

%-----------------------------------------------------------------------------------------------------------------------------------
%--------------------------------------------PLETHYSMS AND OPERADS-----------------------------------------------
%-----------------------------------------------------------------------------------------------------------------------------------

%---------------------Section title-------------------------------------------------------------

\section{Plethysms and operads}\label{section:plethysmsandoperads}

%-----------------------Section introduction---------------------------------------------------------

Let us present the relation between the several plethystic bialgebras, operads and the $\tcons$-construction. Some proofs are omitted, since most of them are similar. The operads involved are the  reduced symmetric operad $\comm$, the reduced associative  operad $\ass$ and their $2$-colored versions. Also, playing the same role as these operads, we have a locally finite monoid $Y$. On the other hand, the $\tcons$-constructions are taken with respect to the monads $\pfsmc$ and $\fsg$, as in Section~\ref{seccion:examples}, and everything is internal to $\ambcat=\infgrpd$. 

Let us stress again that by Proposition~\ref{prop:segalsquares}, Lemmas~\ref{lemma:TCfinite} and~\ref{lemma:TQfinite} and the discussion of Section~\ref{seccion:examples} all the bar constructions featuring in the present section are locally finite Segal groupoids, so that we can take cardinality to arrive at
their incidence bialgebra in the classical sense
of vector spaces.

It is appropriate to begin with the classical bialgebras, which are the main cases.
The following standard notation is used:\label{plethysticnotation}
\begin{multicols}{2}
\begin{itemize}
\item $\x=(x_1,x_2,\dots)$,
\item $\Lambda$: set of infinite vectors of natural numbers with $\l_i=0$ for all $i$ large enough,
\item $\Lambda\ni\l=(\lambda_1,\l_2,\dots)$,
\item $\x^{\l}=x_1^{\l_1}x_2^{\l_2}\cdots$,
\item $\autiv (\l)=1!^{\l_1}\l_1!\cdot 2!^{\l_2}\l_2!\cdots$,
\item $\l!=\l_1!\cdot\l_2!\cdots$,
\item $W$: set of finite words of positive natural numbers,
\item $W\ni \omega=\omega_1\dots\omega_n$,
\item $\x_{\omega}=x_{\omega_1}\cdots x_{\omega_n},$
\item $\omega!=\omega_1!\cdots\omega_n!$.
\end{itemize}
\end{multicols}

\subsection{The classical case}\label{subsection:classicalcase}

%-------------P--------------P----------------P--------------P-------------

The classical Fa\`a di Bruno bialgebra $\FdB$ \cite{Doubilet, Joyal:1981} is obtained from the substitution of power series in one variable. Let $\Q\pow{x}$ be the ring of formal power series with coefficients in $\Q$ without constant term. Elements of $\Q\pow{x}$ are written
$$F(x)=\sum_{n\ge 1} \frac{F_n}{n!}x^n.$$
The set $\Q\pow{x}$ forms  a (noncommutative) monoid with substitution of power series. The \emph{Fa\`a di Bruno bialgebra} $\FdB$ is the free polynomial algebra $\Q[A_1,A_2,\dots]$ generated by the linear maps
\begin{equation*}
\begin{tikzcd}[row sep=0pt]
A_i:\Q\pow{x}\arrow[r]& \Q&\\
\phantom{AAAA}F\arrow[r,mapsto]&F_i&
\end{tikzcd}
\end{equation*}
together with the comultiplication induced by substitution, meaning that 
$$\Delta (A_n)(F,G)=A_n(G\circ F),$$
and counit given by $\epsilon (A_n)=A_n(x)$. The comultiplication of the generators can be explicitly described through the (exponential) Bell polynomials $B_{n,k}$, which count the number of partitions of an $n$-element set into $k$ blocks:
$$\Delta(A_n)=\sum_{k=1}^n A_k\otimes B_{n,k}(A_1,A_2,\cdots).$$

\begin{theorem}[Joyal, cf. modern reformulation in \cite{GKT:FdB}] \label{thm:FdB}The Fa\`a di Bruno bialgebra $\FdB$ is isomorphic to the homotopy cardinality of the incidence bialgebra of $\tsbar\comm$.
\end{theorem}

Note that $\comm$ is of course the same as $\Tpc{}{\idm}{\comm}$ and, as explained in Section \ref{seccion:examples}, it is also $\Tpc{}{\pfsmc}{}$ of the trivial monoid. This connects the Fa\`a di Bruno bialgebra to the $\tcons$-construction in an analogous way as the plethystic bialgebras.

%--------------P----------------P---------------P------------------P

%\subsubsection*{Plethystic bialgebra $\plethB$}\label{subsubsection:plethB}

Let us recall how the classical plethystic substitution works \cite{polya1937, Joyal:1981,Nava-Rota}. 
Let $\Q\pow{x_1,x_2,\dots}$ be the ring of power series in infinitely many variables without constant term and coefficients in $\Q$. 
\noindent Elements of $\Q\pow{\x}$ are written
$$F(\x)=\sum_{\l\in \Lambda}\frac{F_{\l}}{\autiv (\l)}\x^{\l}.$$
Given two power series $F,G\in\Q\pow{\x}$, their \emph{plethystic substitution} is defined as 
 \begin{align}\label{eq:plethsubs}
\nonumber (G\oast F)(x_1,x_2,\dots):=&\,G(F_1,F_2,\dots),\;\;\;\;\;\text{where}\\
F_k(x_1,x_2,\dots):=&\,F(x_k,x_{2k},\dots).&
\end{align}
%The power series $F_k$ is often described by using the Verschiebung operators: the \emph{$n$-th Verschiebung operator} $V^n$ is defined as
%
%\begin{eqnarray*}
%    (V^n\l)_i=\left \{ \begin{array}{ll} \l_{i/n} & \mbox{if} \; n\mid i \\
%        0 & \mbox{otherwise,} 
%    \end{array} \right. \text{ for }  \l=(\l_1,\l_2,\dots).
%\end{eqnarray*}
%
%\noindent For example
%$V^2(5,9,2,0,\dots)=(0,5,0,9,0,2,0,\dots).$  It is clear that $V^n$ preserves sums of infinite vectors. Clearly we have that $F_k$ can be expressed as
%$$F_k(x_1,x_2,\dots)=F(x_k,x_{2k},\dots)=\sum_{\l} F_{\l}\frac{\x^{V^k\l}}{\autiv (\l)}.$$

The set $\Q\pow{\x}$ forms  a (noncommutative) monoid with plethystic substitution. 
 The \emph{plethystic bialgebra} $\plethB$ \cite{Cebrian, Nava} is the free polynomial algebra $\mathcal{P}=\Q\big[\{A_{\l}\}_{\l}\big]$ generated by the set maps
\begin{equation*}
\begin{tikzcd}[row sep=0pt]
A_{\l}:\Q\pow{\x}\arrow[r]& \Q&\\
\phantom{AAAA}F\arrow[r,mapsto]&F_{\l}&
\end{tikzcd}
\end{equation*}
together with the comultiplication induced by substitution, meaning that 
$$\Delta (A_{\l})(F,G)=A_{\l}(G\oast F),$$
and counit given by $\epsilon (A_{\l})=A_{\l}(x_1)$. The comultiplication of the generators can be explicitly described through the polynomials $P_{\sigma,\l}$, a plethystic version of the Bell polynomials which, in the terminology of Nava--Rota \cite{Nava-Rota}, count transversals of partitions.
$$\Delta(A_{\sigma})=\sum_{\l} A_{\l}\otimes P_{\sigma,\l}\big(\{A_{\mu}\}_{\mu}\big).$$
\begin{example} \label{example:classiccomultiplication}
\begin{align}\label{eq:Psigmalabmda}
\nonumber P_{(0,0,0,1,0,2),(1,2)}\big(\{A_{\mu}\}_{\mu}\big)=&\\[5pt]
=\frac{6!^22!4!\cdot 2!}{2!^22!\cdot 4!3!^22!}&A_{(0,0,0,1)}A^2_{(0,0,1)}+\frac{6!^22!4!\cdot 2!}{2!^22!\cdot 6!3!2!}A_{(0,0,0,0,0,1)}A_{(0,0,1)}A_{(0,1)}.
\end{align}
We can give a heuristic explanation of this. Define the \label{Verschiebung}\emph{$k$-th Verschiebung operator} $V^k\colon \Lambda\rightarrow\Lambda$ as
\begin{eqnarray}\label{eq:classicalVerschiebung}
    (V^k\l)_i=\left \{ \begin{array}{ll} \l_{i/k} & \mbox{if} \; k\mid i \\
        0 & \mbox{otherwise,} 
    \end{array} \right.
\end{eqnarray}
Then Equation~\ref{eq:Psigmalabmda} indicates that the vector $(0,0,0,1,0,2)$ can be obtained in the following two ways,
\begin{align*}
(0,0,0,1,0,2)&=V^1(0,0,0,1)+V^2(0,0,1)+V^2(0,0,1),\\
(0,0,0,1,0,2)&=V^1(0,0,0,0,0,1)+V^2(0,0,1)+V^2(0,1),
\end{align*}
where the Verschiebung operators used are determined by $\l$: one $V^1$ and two $V^2$. The coefficients are given by the automorphisms of the vectors involved. See~\cite{Cebrian} for a detailed explanation.
\end{example}

\begin{theorem}\label{thm:plethB}
 The plethystic bialgebra $\plethB$ is isomorphic to the homotopy cardinality of the incidence bialgebra of $\tsbar \Tpc{}{\pfsmc}{\comm}$.
 \begin{proof} The comparison between these two incidence bialgebras was made in \cite{Cebrian}, where the simplicial interpretation of plethysm was established.  
 In Section \ref{seccion:oldTrelation} we will see that indeed $\tsbar \Tpc{}{\pfsmc}{\comm}$ is equivalent  to $T\sur$, the simplicial groupoid of \cite{Cebrian}.
 \end{proof}
\end{theorem}
\begin{example} \label{example:TSymcomultiplication}
Let us see the interpretation of $P_{(0,0,0,1,0,2),(1,2)}\big(\{A_{\mu}\}_{\mu}\big)$ (see Examples~\ref{example:classiccomultiplication}) from the point of view of $\tsbar^{\fsmc} \Tpc{}{\pfsmc}{\comm}$. The vectors $\sigma=(0,0,0,1,0,2)$ and $\l=(1,2)$ are represented by
 
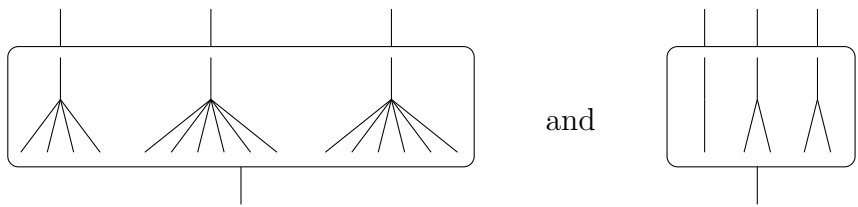
\begin{center}
 \begin{tikzpicture}[grow=down,level distance=20pt]
\tikzstyle{level 1}=[sibling distance=10pt]
\node at (0,0){}
child {
child 
child 
child 
child
};
\node at (2,0){}
child {
child 
child 
child
child
child
child
};
\node at (4.4,0){}
child {
child 
child 
child 
child
child
child
};
\draw [rounded corners] (-0.7,-1.6)--(-0.7,0)--(5.5,0)--(5.5,-1.6)--cycle;
\draw (0,0)--(-0,0.5); \draw (2,0)--(2,0.5); \draw (4.4,0)--(4.4,0.5);
\draw (2.4,-1.6)--(2.4,-2.1);
\end{tikzpicture}
\begin{tikzpicture}
\node at (1,0){}; \node at (0,-1){}; \node at (-1,0){};
\node at (0,0){and};
\end{tikzpicture}
\begin{tikzpicture}[grow=down,level distance=20pt]
\tikzstyle{level 1}=[sibling distance=10pt]
\node at (0,0){}
child {
child 
};
\node at (0.7,0){}
child {
child 
child 
};
\node at (1.5,0){}
child {
child 
child
};
\draw [rounded corners] (-0.5,-1.6)--(-0.5,0)--(2,0)--(2,-1.6)--cycle;
\draw (0,0)--(0,0.5); \draw (0.7,0)--(0.7,0.5); \draw (1.5,0)--(1.5,0.5);
\draw (0.7,-1.6)--(0.7,-2.1);
\end{tikzpicture}
\end{center}
respectively. We have used the opposite convention (Subection~\ref{subsection:oppositeconvention}), which in this case does not affect the result. What we want to count is, roughly speaking, the number of ways we can obtain $\sigma$ as a composition of $\l$ with three operations. It is straightforward to see that there are essentially two choices:
 
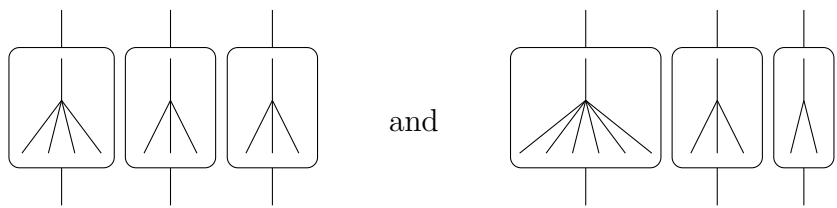
\begin{center}
 \begin{tikzpicture}[grow=down,level distance=20pt]
\tikzstyle{level 1}=[sibling distance=10pt]
\node at (0,0){}
child {
child 
child 
child 
child
};
\draw [rounded corners] (-0.7,-1.6)--(-0.7,0)--(0.7,0)--(0.7,-1.6)--cycle;
\draw (0,0)--(-0,0.5);
\draw (0,-1.6)--(0,-2.1);
\end{tikzpicture}
 \begin{tikzpicture}[grow=down,level distance=20pt]
\tikzstyle{level 1}=[sibling distance=10pt]
\node at (0,0){}
child {
child 
child 
child 
};
\draw [rounded corners] (-0.6,-1.6)--(-0.6,0)--(0.6,0)--(0.6,-1.6)--cycle;
\draw (0,0)--(-0,0.5);
\draw (0,-1.6)--(0,-2.1);
\end{tikzpicture}
 \begin{tikzpicture}[grow=down,level distance=20pt]
\tikzstyle{level 1}=[sibling distance=10pt]
\node at (0,0){}
child {
child 
child 
child 
};
\draw [rounded corners] (-0.6,-1.6)--(-0.6,0)--(0.6,0)--(0.6,-1.6)--cycle;
\draw (0,0)--(-0,0.5);
\draw (0,-1.6)--(0,-2.1);
\end{tikzpicture}
\begin{tikzpicture}
\node at (1,0){}; \node at (0,-1){}; \node at (-1,0){};
\node at (0,0){and};
\end{tikzpicture}
 \begin{tikzpicture}[grow=down,level distance=20pt]
\tikzstyle{level 1}=[sibling distance=10pt]
\node at (0,0){}
child {
child 
child 
child 
child
child
child
};
\draw [rounded corners] (-1,-1.6)--(-1,0)--(1,0)--(1,-1.6)--cycle;
\draw (0,0)--(-0,0.5);
\draw (0,-1.6)--(0,-2.1);
\end{tikzpicture}
 \begin{tikzpicture}[grow=down,level distance=20pt]
\tikzstyle{level 1}=[sibling distance=10pt]
\node at (0,0){}
child {
child 
child 
child 
};
\draw [rounded corners] (-0.6,-1.6)--(-0.6,0)--(0.6,0)--(0.6,-1.6)--cycle;
\draw (0,0)--(-0,0.5);
\draw (0,-1.6)--(0,-2.1);
\end{tikzpicture}
 \begin{tikzpicture}[grow=down,level distance=20pt]
\tikzstyle{level 1}=[sibling distance=10pt]
\node at (0,0){}
child {
child 
child 
};
\draw [rounded corners] (-0.4,-1.6)--(-0.4,0)--(0.4,0)--(0.4,-1.6)--cycle;
\draw (0,0)--(-0,0.5);
\draw (0,-1.6)--(0,-2.1);
\end{tikzpicture}
\end{center}
which clearly coincide with the ones of Example~\ref{example:classiccomultiplication}. This example could be misleading, in the sense that each operation above contains only one operation of $\comm$. This happens of course because $|\sigma|=|\l|$. For instance, 
if we take instead $\l=(1,1)$ we obtain three possible choices:
 
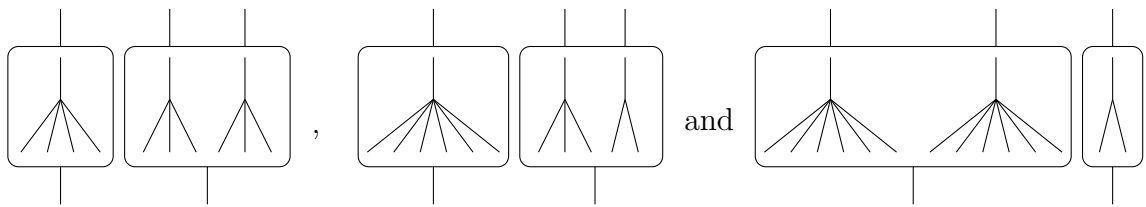
\begin{center}
 \begin{tikzpicture}[grow=down,level distance=20pt]
\tikzstyle{level 1}=[sibling distance=10pt]
\node at (0,0){}
child {
child 
child 
child 
child
};
\draw [rounded corners] (-0.7,-1.6)--(-0.7,0)--(0.7,0)--(0.7,-1.6)--cycle;
\draw (0,0)--(-0,0.5);
\draw (0,-1.6)--(0,-2.1);
\end{tikzpicture}
 \begin{tikzpicture}[grow=down,level distance=20pt]
\tikzstyle{level 1}=[sibling distance=10pt]
\node at (0,0){}
child {
child 
child 
child 
};
%\draw [rounded corners] (-0.6,-1.6)--(-0.6,0)--(0.6,0)--(0.6,-1.6)--cycle;
%\draw (0,0)--(-0,0.5);
%\draw (0,-1.6)--(0,-2.1);
\node at (1,0){}
child {
child 
child 
child 
};
\draw [rounded corners] (-0.6,-1.6)--(-0.6,0)--(1.6,0)--(1.6,-1.6)--cycle;
\draw (0,0)--(-0,0.5); \draw (1,0)--(1,0.5); 
\draw (0.5,-1.6)--(0.5,-2.1);
\end{tikzpicture}
\begin{tikzpicture}
\end{tikzpicture}
\begin{tikzpicture}
\node at (0.2,0){}; \node at (0,-1){};
\node at (0,-0.2){\!\!,};
\end{tikzpicture}
\begin{tikzpicture}[grow=down,level distance=20pt]
\tikzstyle{level 1}=[sibling distance=10pt]
\node at (0,0){}
child {
child 
child 
child 
child
child
child
};
\draw [rounded corners] (-1,-1.6)--(-1,0)--(1,0)--(1,-1.6)--cycle;
\draw (0,0)--(-0,0.5);
\draw (0,-1.6)--(0,-2.1);
\end{tikzpicture}
 \begin{tikzpicture}[grow=down,level distance=20pt]
\tikzstyle{level 1}=[sibling distance=10pt]
\node at (0,0){}
child {
child 
child 
child
};
%\draw [rounded corners] (-0.6,-1.6)--(-0.6,0)--(0.6,0)--(0.6,-1.6)--cycle;
%\draw (0,0)--(-0,0.5);
%\draw (0,-1.6)--(0,-2.1);
\node at (0.8,0){}
child {
child 
child 
};
\draw [rounded corners] (-0.6,-1.6)--(-0.6,0)--(1.3,0)--(1.3,-1.6)--cycle;
\draw (0,0)--(-0,0.5); \draw (0.8,0)--(0.8,0.5); 
\draw (0.4,-1.6)--(0.4,-2.1);
\end{tikzpicture}
\begin{tikzpicture}
\node at (0.3,0){}; \node at (0,-1){}; \node at (-0.3,0){};
\node at (0,0){and};
\end{tikzpicture}
 \begin{tikzpicture}[grow=down,level distance=20pt]
\tikzstyle{level 1}=[sibling distance=10pt]
\node at (0,0){}
child {
child 
child 
child
child
child
child
};
\node at (2.2,0){}
child {
child 
child 
child 
child
child
child
};
\draw [rounded corners] (-1,-1.6)--(-1,0)--(3.2,0)--(3.2,-1.6)--cycle;
\draw (0,0)--(-0,0.5); \draw (2.2,0)--(2.2,0.5);
\draw (1.1,-1.6)--(1.1,-2.1);
\end{tikzpicture}
 \begin{tikzpicture}[grow=down,level distance=20pt]
\tikzstyle{level 1}=[sibling distance=10pt]
\node at (0,0){}
child {
child 
child 
};
\draw [rounded corners] (-0.4,-1.6)--(-0.4,0)--(0.4,0)--(0.4,-1.6)--cycle;
\draw (0,0)--(-0,0.5);
\draw (0,-1.6)--(0,-2.1);
\end{tikzpicture}
\end{center}

\end{example}
%---------Summary------------------------

\subsection{Overview of variations}

We proceed to introduce the variations of the plethystic bialgebra we explore. For the set of variables $(x_1,x_2,\dots)$, there are three sources of variations. At the level of power series they are the following:
\begin{enumerate}[(i)]
\item Commuting or noncommuting variables: of course in the classical case the variables commute. When the variables do not commute we will index them by $\omega\in W$, rather than $\l\in \Lambda$.
\item Commuting or noncommuting coefficients.
\item Two types of automorphisms: $\autiv(\l)$ or $\l!$ for commuting variables, and $\omega!$ or $1$ for noncommuting variables.
\end{enumerate}
These variations are not independent: if the variables commute then the coefficients commute. Analogous variations can be obtained of the Fa\`a di Bruno bialgebra, except in this case there is only one variable. 

At the objective level, these three variations correspond (respectively) to the following choices:
\begin{enumerate}[(i)]
\item $\tcons$-construction over $\pfsmc$ or over $\fsg$.
\item Bar construction over $\pfsmc$ or over $\fsg$.
\item Taking $\comm$ or $\ass$ as input operads.
\end{enumerate}
The reason why they are not independent is clear here: there is a cartesian natural transformation $\fsg\Rightarrow\pfsmc$ that allows taking $\tsbar^{\pfsmc}$ of a $\fsg$-operad (see Section \ref{seccion:MonadsMulticategories}), but no natural transformation in the opposite direction.

Let us give a brief justification of these correspondences. Consider the following sequence of operations:
%---D-------I------A-------G-----R---------A-------M-------
 
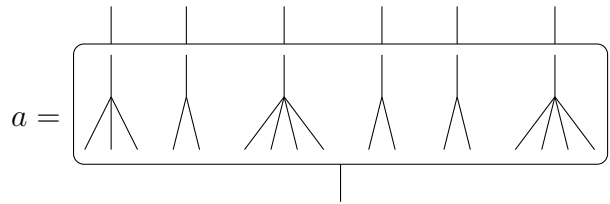
\begin{center}
 \begin{tikzpicture}[grow=down,level distance=20pt]
\tikzstyle{level 1}=[sibling distance=10pt]
\node at (-1,-1){$a=$};
\node at (0,0){}
child {
child 
child 
child 
};
\node at (1,0){}
child {
child 
child 
};
\node at (2.3,0){}
child {
child 
child 
child 
child
};
\node at (3.6,0){}
child {
child 
child 
};
\node at (4.6,0){}
child {
child 
child 
};
\node at (5.9,0){}
child {
child 
child 
child 
child
};
\draw [rounded corners] (-0.5,-1.6)--(-0.5,0)--(6.6,0)--(6.6,-1.6)--cycle;
\foreach \i in {0,1,2.3,3.6,4.6,5.9}
{
	\draw (\i,0)--(\i,0.5);
}
\draw (3.05,-1.6)--(3.05,-2.1);
\end{tikzpicture}
\end{center}
%--------------------------------------------------------------
This could be either an operation in one of the following operads:
\begin{enumerate}[(i)]
\item $\Tpc{}{\pfsmc}{\comm}$: in this case each operation has automorphisms, coming from the action of the symmetric group on $\comm$, and since the $\tcons$-construction is over $\pfsmc$ we can permute the operations. This means that the isomorphism class of $a$ is given by $\l=(0,3,1,2)$, since the order of the operations does not matter, and it has $\autiv(\l)=2!^33!\cdot 3!^11!\cdot 4!^22!$ automorphisms.  The corresponding bialgebra is thus $\plethB$ and this particular operation corresponds to $A_{(0,3,1,2)}$, the linear map returning the coefficient of $x_2^3x_3x_4^2/\autiv(\l)$. 

\item $\Tpc{}{\fsg}{\comm}$: in this case the operations have automorphisms again, but since the $\tcons$-construction is over $\fsg$ we cannot permute them. This means that the isomorphism class of $a$ is given by $\omega=(3,2,4,2,2,4)$, so that it corresponds to noncommuting variables. Clearly it has $3!2!4!2!2!4!$ automorphisms. Now, depending on the bar construction it corresponds to commuting or noncommuting coefficients. This particular operation corresponds to $A_{(3,2,4,2,2,4)}$, the linear map returning the coefficient of $x_3x_2x_4x_2x_2x_4/\omega!$.

\item $\Tpc{}{\fsg}{\ass}$: in this case the operations do not have automorphisms, and since the $\tcons$-construction is over $\fsg$ we cannot permute them. This means that the isomorphism class of $a$ is given by $\omega=(3,2,4,2,2,4)$, so that it corresponds to noncommuting variables,  and it has no automorphisms. Now, depending on the bar construction it corresponds to commuting or noncommuting coefficients, as in the previous case. This particular operation corresponds to $a_{(3,2,4,2,2,4)}$, the linear map returning the coefficient of $x_3x_2x_4x_2x_2x_4$.

\item $\Tpc{}{\pfsmc}{\ass}$: in this case the operations do not have automorphisms, and since the $\tcons$-construction is over $\pfsmc$ we can permute them. This means that the isomorphism class of $a$ is given by $\l=(0,3,1,2)$, and it has $\l!=3!\cdot 1!\cdot 2!$ automorphisms. Therefore it corresponds to commuting variables and coefficients. This particular operation corresponds to $a_{(0,3,1,2)}$, the linear map returning the coefficient of $x_2^3x_3x_4^2/\l!$.
\end{enumerate}

The cases of $\comm$ and $\ass$ are developed in Subsections $\ref{subsection:comm}$ and $\ref{subsection:ass}$ respectively. In Subsection \ref{subsection:monoid} we generalize $\EplethB$ to power series in the set of variables $(x_m\,|\,m\in Y)$ indexed over a locally finite monoid.

In Subsections \ref{subsection:comm} and \ref{subsection:ass} we also study the Fa\`a di Bruno bialgebra in two variables and the plethystic bialgebra in the two sets of variables $(x_1,x_2,\dots),(y_1,y_2,\dots)$. For the plethystic case we only consider commuting variables and coefficients. 
Let us give a similar digression as above for the plethystic cases. Consider the following $2$-colored operation:
%---D-------I------A-------G-----R---------A-------M-------
 
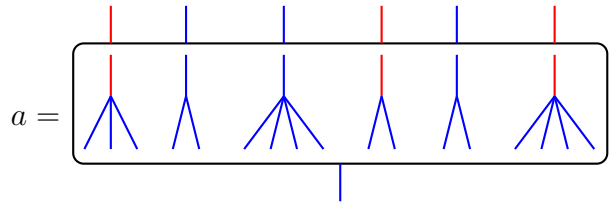
\begin{center}
 \begin{tikzpicture}[grow=down,level distance=20pt,thick]
\tikzstyle{level 1}=[sibling distance=10pt]
\node at (-1,-1){$a=$};
\node at (0,0){}
child[red] {
child[blue]
child[blue]
child[blue] 
};
\node at (1,0){}
child[blue] {
child[blue] 
child[blue] 
};
\node at (2.3,0){}
child[blue] {
child[blue]
child[blue] 
child[blue] 
child[blue]
};
\node at (3.6,0){}
child[red] {
child[blue] 
child[blue] 
};
\node at (4.6,0){}
child [blue]{
child[blue] 
child [blue]
};
\node at (5.9,0){}
child[red] {
child[blue] 
child [blue]
child [blue]
child[blue]
};
\draw [rounded corners] (-0.5,-1.6)--(-0.5,0)--(6.6,0)--(6.6,-1.6)--cycle;
\foreach \i in {0,3.6,5.9}
{
	\draw [red](\i,0)--(\i,0.5);
}
\foreach \i in {1,2.3,4.6}
{
	\draw [blue](\i,0)--(\i,0.5);
}
\draw [blue](3.05,-1.6)--(3.05,-2.1);
\end{tikzpicture}

\end{center}

%--------------------------------------------------------------
The isomorphism class of this operation is given by $(\l^1,\l^2)=\big({\color{blue}(0,2,0,2)},{\color{red}(0,1,1,1)}\big)$ (since everything commutes now), and it can either be an operation in $\Tpc{}{\pfsmc}{\ass}_2$ or $\Tpc{}{\pfsmc}{\ass}_2$. It thus corresponds to $A_{\big({\color{blue}(0,2,0,2)},{\color{red}(0,1,1,1)}\big)}\in \IIplethB$ or to $a_{\big({\color{blue}(0,2,0,2)},{\color{red}(0,1,1,1)}\big)}\in \EIIplethB$, the linear maps returning the coefficients of $x_2^2x_4^2y_2y_3y_4/\autiv(\l^x)\autiv(\l^y)$.

%-------------SYM------------SYM-------------SYM-----------------SYM

\subsection{Bialgebras from $\comm$ and $\comm_2$}\label{subsection:comm}

We have already seen two bialgebras arising from $\comm$ in Subsection \ref{subsection:classicalcase}, the Fa\`a di Bruno bialgebra $\FdB$ and the plethystic bialgebra $\plethB$. Let us see the aforementioned variations.

%---------------PX----------------PX----------------PX-----------------PX
Replace $\Q\pow{\x}$ by $\Q\ncpow{\x}$, that is, noncommuting variables. Elements of $\Q\ncpow{\x}$ are written
$$F(\x)=\sum_{\omega\in W}\frac{F_{\omega}}{\omega !}\x^{\omega},$$
Substitution of power series in $\Q\ncpow{\x}$ is defined in the same way as before (\ref{eq:plethsubs}). The \emph{plethistic bialgebra with noncommuting variables} $\plethBX$ is defined as the free polynomial algebra $\Q\big[\{A_{\omega}\}_{\omega}\big]$ on the set maps $A_{\omega}$ and comultiplication and counit as usual.

\begin{theorem}\label{thm:plethBX}
 The  plethystic bialgebra with noncommuting variables $\plethBX$ is isomorphic to the homotopy cardinality of the incidence bialgebra of $\tsbar^{\pfsmc} \Tpc{}{\fsg}{\comm}$.
 \end{theorem}
 
 %---------------XPX----------------XPX----------------XPX-----------------XPX

 If we take $R\ncpow{\x}$ with $R$ a noncommutative unital ring, then we get the \emph{noncommutative plethystic bialgebra with noncommuting variables} $\XplethBX$, which is the free associative unital algebra $\Q\langle \{A_{\omega}\}_{\omega}\rangle$ together with the usual comultiplication and counit. In this case, substitution of power series is defined in the same way but it is not associative. However the comultiplication is still associative. A proof of this can be found in \cite{BFK:0406117} for the one variable case, which is obtained below.
 
\begin{theorem}\label{thm:XplethBX}
 The noncommutative plethystic bialgebra with noncommuting variables $\XplethBX$ is isomorphic to the homotopy cardinality of the incidence bialgebra of $\tsbar \Tpc{}{\fsg}{\comm}$.
\end{theorem}

Let us move forward to power series in two variables. All the results are also valid for any number of variables, but for simplicity and notation we have chosen to show the two variables case. Also, for the bivariate plethystic bialgebras, we do not enter into noncommutativity of the variables or of the coefficients.

%------------------FDB2----------------FDB2--------------------FDB2------------------FDB2

Let $\Q\pow{x,y}$ be the ring of formal power series in the variables $x$ and $y$ with coefficients in $\Q$ without constant term. Elements of $\Q\pow{x,y}$ are written
$$F(x,y)=\sum_{n+m\ge 1} \frac{F_{n,m}}{n!m!}x^ny^m.$$
The set $\Q\pow{x,y}\cx \Q\pow{x,y}$ forms  a (noncommutative) monoid with substitution of power series:
\begin{equation*}
\begin{tikzcd}[row sep=0pt,column sep=30pt]
\big(\Q\pow{x,y}\cx \Q\pow{x,y}\big)\cx \big(\Q\pow{x,y}\cx \Q\pow{x,y}\big)\arrow[r,"\circ"]&\Q\pow{x,y}\cx \Q\pow{x,y}\phantom{AAAAA}&\\
\phantom{AAAAAAAAAA}\big((F^1,F^2),(G^1,G^2)\big)\arrow[r]&\big(G^1(F^1,F^2),G^2(F^1,F^2)\big).&
\end{tikzcd}
\end{equation*}
 We define the \emph{Fa\`a di Bruno bialgebra in two variables} $\IIFdB$ as the free polynomial algebra $\Q\big[\{A^i_{n,m}\}_{n+m\ge 1}^{i=1,2}\big]$ generated by the set maps
\begin{equation*}
\begin{tikzcd}[row sep=0pt]
A^i_{n,m}:\Q\pow{x,y}\cx\Q\pow{x,y}\arrow[r]& \Q&\\
\phantom{AAAA}(F^1,F^2)\arrow[r]&F_{n,m}^i&
\end{tikzcd}
\end{equation*}
together with the comultiplication induced by substitution, meaning that 
$$\Delta (A_{n,m}^i)\big((F^1,F^2),(G^1,G^2)\big)=A_{n,m}^i\big((G^1,G^2)\circ (F^1,F^2)\big),$$
and counit given by $\epsilon (A_{n,m}^i)=A^i_{n,m}(x,y)$. 

\begin{theorem}\label{thm:2FdB}
The Fa\`a di Bruno bialgebra in two variables $\IIFdB$ is isomorphic to the homotopy cardinality of the incidence bialgebra of $\tsbar \comm_2$. The same holds for $n$ variables and $\comm_n$.
\end{theorem}

Notice that $\comm_2$ is the same as $\Tpc{}{\idm}{\comm}$ and, as explained in Example \ref{ex:TassTsym}, it is also $\Tpc{}{\pfsmc}{C}$, where $C\!=\!\{\!\!\begin{tikzcd}0\!\arrow[r,bend left=20]&\!1\arrow[l,bend left=20]\end{tikzcd}\!\!\}$. This connects the Fa\`a di Bruno bialgebra in two variables to the $\tcons$-construction in an analogous way as the plethystic bialgebras.

%-----------------P2---------------P2-------------------P2------------------P2

We can do the same with the power series ring in two sets of infinitely many variables $\Q\pow{\x,\y}$ with coefficients in $\Q$. We shall write 
$$\X=(\x,\y), \;\;\l=(\l^1,\l^2)\in \Lambda^2,\;\; \autiv(\l)=\autiv(\l^1)\autiv(\l^2) \;\;\text{and} \;\; \X^{\l}=\x^{\l^1}\y^{\l^2},$$
so that elements of $\Q\pow{\X}$ are written
$$F(\X)=\sum_{\l} \frac{F_{\l}}{\autiv (\l)}\X^{\l}.$$
The set $\Q\pow{\X}\cx \Q\pow{\X}$ forms  a (noncommutative) monoid with plethystic substitution of power series:
\begin{equation*}
\begin{tikzcd}[row sep=0pt,column sep=30pt]
\big(\Q\pow{\X}\cx \Q\pow{\X}\big)\cx \big(\Q\pow{\X}\cx \Q\pow{\X}\big)\arrow[r,"\oast"]&\Q\pow{\X}\cx \Q\pow{\X}\phantom{AAAAA}&\\
\phantom{AAAAAAAAAA}\big((F^1,F^2),(G^1,G^2)\big)\arrow[r,mapsto]&\big(G^1(F^1,F^2),G^2(F^1,F^2)\big)&
\end{tikzcd}
\end{equation*}
 The \emph{plethystic bialgebra in two variables} $\IIplethB$ is defined as the free polynomial algebra $\Q\big[\{A^i_{\l}\}^{i=1,2}\big]$ generated by the set maps
\begin{equation*}
\begin{tikzcd}[row sep=0pt]
A^i_{\l}:\Q\pow{\X}\cx\Q\pow{\X}\arrow[r]& \Q&\\
\phantom{AAAA}(F^1,F^2)\arrow[r,mapsto]&F_{\l}^i&
\end{tikzcd}
\end{equation*}
together with the comultiplication induced by substitution, meaning that 
$$\Delta (A_{\l}^i)\big((F^1,F^2),(G^1,G^2)\big)=A_{\l}^i\big((G^1,G^2)\circ (F^1,F^2)\big),$$
and counit given by $\epsilon (A_{\l}^i)=A^i_{\l}(x,y)$. 

\begin{theorem}\label{thm:2plethB}
The  plethystic bialgebra in two variables $\IIplethB$ is isomorphic to the homotopy cardinality of the incidence bialgebra of $\tsbar \Tpc{}{\pfsmc}{\comm_2}$.
\end{theorem}

\begin{example} \label{example:TSymcomultiplication}
We could define polynomials $P^2_{\sigma,\l}\big(\{A_{\mu}\}^i_{\mu}\big)$ to express the comultiplication of $A^i_{\sigma}$, in analogy to the univariate case. Let us see again the interpretation of $P^2_{\sigma,\l}\big(\{A_{\mu}\}^i_{\mu}\big)$, for $\sigma=\big({\color{blue}(0,0,0,0,0,1)},{\color{red}(0,0,0,1,0,1)}\big)$ and $\l=\big({\color{blue}(1,1)},{\color{red}(0,1)}\big)$, from the point of view of $\tsbar^{\fsmc} \Tpc{}{\pfsmc}{\comm_2}$. These two vectors are represented by
 
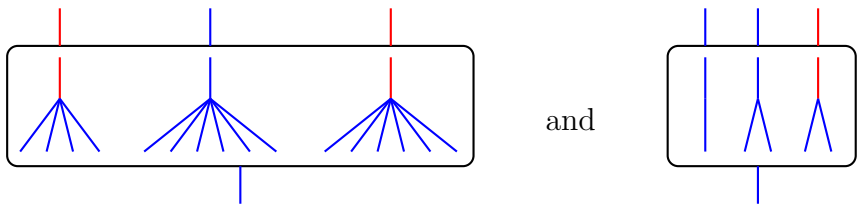
\begin{center}
 \begin{tikzpicture}[grow=down,level distance=20pt,thick]
\tikzstyle{level 1}=[sibling distance=10pt]
\node at (0,0){}
child[red] {
child[blue] 
child [blue]
child [blue]
child [blue]
};
\node at (2,0){}
child [blue]{
child [blue]
child [blue]
child [blue]
child [blue]
child [blue]
child [blue]
};
\node at (4.4,0){}
child [red]{
child [blue]
child [blue]
child [blue]
child [blue]
child [blue]
child [blue]
};
\draw [rounded corners] (-0.7,-1.6)--(-0.7,0)--(5.5,0)--(5.5,-1.6)--cycle;
\draw [red](0,0)--(-0,0.5); \draw [blue](2,0)--(2,0.5); \draw [red](4.4,0)--(4.4,0.5);
\draw [blue](2.4,-1.6)--(2.4,-2.1);
\end{tikzpicture}
\begin{tikzpicture}
\node at (1,0){}; \node at (0,-1){}; \node at (-1,0){};
\node at (0,0){and};
\end{tikzpicture}
\begin{tikzpicture}[grow=down,level distance=20pt,thick]
\tikzstyle{level 1}=[sibling distance=10pt]
\node at (0,0){}
child [blue]{
child [blue]
};
\node at (0.7,0){}
child [blue]{
child [blue]
child [blue]
};
\node at (1.5,0){}
child [red]{
child [blue]
child [blue]
};
\draw [rounded corners] (-0.5,-1.6)--(-0.5,0)--(2,0)--(2,-1.6)--cycle;
\draw [blue](0,0)--(0,0.5); \draw [blue](0.7,0)--(0.7,0.5); \draw [red](1.5,0)--(1.5,0.5);
\draw [blue](0.7,-1.6)--(0.7,-2.1);
\end{tikzpicture}
\end{center}
respectively. The output color depends on whether we are computing the comultiplication of $A^1_{\sigma}$ or $A^2_{\sigma}$. We assume the former, without loss of generality. It is easy to see that there are essentially three options, which are the only possible colorings of the solutions for the analogous case of Example~\ref{example:TSymcomultiplication}:
 \begin{center}
 \begin{tikzpicture}[grow=down,level distance=20pt,thick]
\tikzstyle{level 1}=[sibling distance=10pt]
\node at (0,0){}
child [red]{
child [blue]
child [blue]
child [blue]
child [blue]
};
\draw [rounded corners] (-0.7,-1.6)--(-0.7,0)--(0.7,0)--(0.7,-1.6)--cycle;
\draw [red](0,0)--(-0,0.5);
\draw [blue](0,-1.6)--(0,-2.1);
\end{tikzpicture}
 \begin{tikzpicture}[grow=down,level distance=20pt,thick]
\tikzstyle{level 1}=[sibling distance=10pt]
\node at (0,0){}
child [blue]{
child [blue]
child [blue]
child [blue]
};
\draw [rounded corners] (-0.6,-1.6)--(-0.6,0)--(0.6,0)--(0.6,-1.6)--cycle;
\draw [blue](0,0)--(-0,0.5);
\draw [blue](0,-1.6)--(0,-2.1);
\end{tikzpicture}
 \begin{tikzpicture}[grow=down,level distance=20pt,thick]
\tikzstyle{level 1}=[sibling distance=10pt]
\node at (0,0){}
child [red]{
child [red]
child [red]
child [red]
};
\draw [rounded corners] (-0.6,-1.6)--(-0.6,0)--(0.6,0)--(0.6,-1.6)--cycle;
\draw [red](0,0)--(-0,0.5);
\draw [red](0,-1.6)--(0,-2.1);
\end{tikzpicture}
\begin{tikzpicture}
\node at (0.2,0){}; \node at (0,-1){};
\node at (0,-0.2){\!\!,};
\end{tikzpicture}
 \begin{tikzpicture}[grow=down,level distance=20pt,thick]
\tikzstyle{level 1}=[sibling distance=10pt]
\node at (0,0){}
child [blue]{
child [blue]
child [blue]
child [blue]
child[blue]
child[blue]
child[blue]
};
\draw [rounded corners] (-1,-1.6)--(-1,0)--(1,0)--(1,-1.6)--cycle;
\draw [blue](0,0)--(-0,0.5);
\draw [blue](0,-1.6)--(0,-2.1);
\end{tikzpicture}
 \begin{tikzpicture}[grow=down,level distance=20pt,thick]
\tikzstyle{level 1}=[sibling distance=10pt]
\node at (0,0){}
child [red]{
child [blue]
child [blue]
child [blue]
};
\draw [rounded corners] (-0.6,-1.6)--(-0.6,0)--(0.6,0)--(0.6,-1.6)--cycle;
\draw [red](0,0)--(-0,0.5);
\draw [blue](0,-1.6)--(0,-2.1);
\end{tikzpicture}
 \begin{tikzpicture}[grow=down,level distance=20pt,thick]
\tikzstyle{level 1}=[sibling distance=10pt]
\node at (0,0){}
child [red]{
child [red]
child [red]
};
\draw [rounded corners] (-0.4,-1.6)--(-0.4,0)--(0.4,0)--(0.4,-1.6)--cycle;
\draw [red](0,0)--(-0,0.5);
\draw [red](0,-1.6)--(0,-2.1);
\end{tikzpicture}
\begin{tikzpicture}
\node at (0.3,0){}; \node at (0,-1){}; \node at (-0.3,0){};
\node at (0,0){and};
\end{tikzpicture}
 \begin{tikzpicture}[grow=down,level distance=20pt,thick]
\tikzstyle{level 1}=[sibling distance=10pt]
\node at (0,0){}
child [red]{
child [blue]
child [blue]
child [blue]
child[blue]
child[blue]
child[blue]
};
\draw [rounded corners] (-1,-1.6)--(-1,0)--(1,0)--(1,-1.6)--cycle;
\draw [red](0,0)--(-0,0.5);
\draw [blue](0,-1.6)--(0,-2.1);
\end{tikzpicture}
 \begin{tikzpicture}[grow=down,level distance=20pt,thick]
\tikzstyle{level 1}=[sibling distance=10pt]
\node at (0,0){}
child [blue]{
child [blue]
child [blue]
child [blue]
};
\draw [rounded corners] (-0.6,-1.6)--(-0.6,0)--(0.6,0)--(0.6,-1.6)--cycle;
\draw [blue](0,0)--(-0,0.5);
\draw [blue](0,-1.6)--(0,-2.1);
\end{tikzpicture}
 \begin{tikzpicture}[grow=down,level distance=20pt,thick]
\tikzstyle{level 1}=[sibling distance=10pt]
\node at (0,0){}
child [red]{
child [red]
child [red]
};
\draw [rounded corners] (-0.4,-1.6)--(-0.4,0)--(0.4,0)--(0.4,-1.6)--cycle;
\draw [red](0,0)--(-0,0.5);
\draw [red](0,-1.6)--(0,-2.1);
\end{tikzpicture}
\end{center}
\end{example}

%---------------ASS-----------------ASS-----------------ASS-----------------ASS

\subsection{Bialgebras from $\ass$ and $\ass_2$}\label{subsection:ass}

%The main difference between $\comm$ and $\ass$ is that the groupoid of operations is discrete in the second case. At the level of bialgebras this translates to a change of basis with respect to the four bialgebras obtained above which concerns the automorphisms appearing in the power series, as we shall see. However, at the objective level, the corresponding simplicial groupoids are not equivalent. Besides these bialgebras, we will also obtain the noncommutative Fa\`a di Bruno bialgebra.

%-----------------OFDB-----------------OFDB-----------------OFDB-----------------OFDB

Take again $\Q\pow{x}$, but write now elements of $\Q\pow{x}$ as
$$F(x)=\sum_{n\ge 1} f_nx^n.$$
The \emph{ordinary Fa\`a di Bruno bialgebra} $\OFdB$ is the free polynomial algebra $\Q[a_1,a_2,\dots]$ generated by the linear maps $a_i(F)=f_i$ 
%\begin{equation*}
%\begin{tikzcd}[row sep=0pt]
%a_i:\Q\pow{x}\arrow[r]& \Q&\\
%\phantom{AAAA}F\arrow[r,mapsto]&f_i&
%\end{tikzcd}
%\end{equation*}
together with the comultiplication induced by substitution and counit given by $\epsilon (a_n)=a_n(x)$, as before. 
%The comultiplication of the generators can be explicitly described through the ordinary Bell polynomials $\hat{B}_{n,k}$, which count the number of monotone surjections $\underline{n}\twoheadrightarrow \underline{k}$:
%\begin{align*}
%\hat{B}_{n,k}(a_1,a_2,\dots)&=\frac{k!}{n!}B_{n,k}(A_1,A_2,\dots),\\
%\Delta(a_n)&=\sum_{k=1}^n a_k\otimes \hat{B}_{n,k}(a_1,a_2,\cdots).
%\end{align*}
%{\color{red}REFS}

\begin{theorem}\label{thm:OFdB}
The ordinary Fa\`a di Bruno bialgebra $\OFdB$ is isomorphic to the homotopy cardinality of the incidence bialgebra of $\tsbar^{\pfsmc} \ass $.
\end{theorem}
It is clear that $\FdB$ and $\OFdB$ are isomorphic bialgebras, since we have only changed the basis. However their combinatorial meaning is slightly different, and indeed $\tsbar^{\pfsmc} \comm $ and $\tsbar^{\pfsmc} \ass $ are not equivalent.
Note that $\ass$ is of course the same as $\Tpc{}{\idm}{\ass}$ and, as explained in Section \ref{seccion:examples}, it is also $\Tpc{}{\fsg}{}$ of the trivial monoid. This connects the ordinary Fa\`a di Bruno bialgebra to the $\tcons$-construction. 

%---------------XFDB----------------XFDB----------------XFDB--------------XFDB

If we replace above $\Q$ by $R$ (a noncommutative unital ring),  we obtain the \emph{noncommutative Fa\`a di Bruno bialgebra} $\XFdB$ \cite{BFK:0406117,EbrahimiFard-Lundervold-Manchon:1402.4761,Lundervold-MuntheKaas:0905.0087}, the free associative unital algebra $\Q\langle a_1,a_2,\dots\rangle$ generated by the set maps $a_i(F)=f_i$,
%\begin{equation*}
%\begin{tikzcd}[row sep=0pt]
%a_i:R\pow{x}\arrow[r]& R&\\
%\phantom{AAAA}F\arrow[r]&f_i&
%\end{tikzcd}
%\end{equation*}
together with the comultiplication induced by substitution and counit  $\epsilon (a_n)=a_n(x)$, as before. In this case, substitution of power series is not associative, but the comultiplication is still coassociative \cite{BFK:0406117}. 
%The comultiplication of the generators can be explicitly described through the noncommutative Bell polynomials $\overline{B}_{n,k}$ \cite{BFK:0406117}, which again count the number of monotone surjections from an $n$-element ordered set into a $k$-element ordered set, but with noncommutative variables:
%\begin{align*}
%\overline{B}_{n,k}(a_1,a_2,\dots)&=\frac{k!}{n!}B_{n,k}(A_1,A_2,\dots),\\
%\Delta(a_n)&=\sum_{k=1}^n \overline{B}_{n,k}(a_1,a_2,\cdots).
%\end{align*}
It is clear that $\FdB$ and $\OFdB$ are the abelianization of $\XFdB$ \cite{BFK:0406117}. 

\begin{theorem}\label{thm:XFdB}
The noncommutative Fa\`a di Bruno bialgebra $\XFdB$ is isomorphic to the homotopy cardinality of the incidence bialgebra of $\tsbar^{\fsg} \ass $.
\end{theorem}

%-----------------EP---------------EP----------------EP-----------------EP

We now move to the plethystic bialgebras.  The \emph{exponential plethystic bialgebra} $\EplethB$ is the same bialgebra as $\plethB$, but in this case $\autiv(\l)=\l!=\l_1!\l_2!\cdots$ \cite{Nava}. The generators of this bialgebra are denoted by $a_{\l}$.
 
 \begin{theorem}\label{thm:EplethB}
  The exponential plethystic bialgebra $\EplethB$ is isomorphic to the homotopy cardinality of the incidence bialgebra of $\tsbar \Tpc{}{\pfsmc}{\ass}$.
  \end{theorem}
  
  %----------------LPX--------------LPX----------------LPX----------------LPX
  
 The \emph{linear plethystic bialgebra with noncommuting variables} $\LplethBX$ is the same bialgebra as $\plethBX$ but without automorphisms of $\omega$. The generators for this bialgebra are denoted $a_{\omega}$.

\begin{theorem}\label{thm:LplethBX}
 The linear plethystic bialgebra with noncommuting variables $\LplethBX$ is isomorphic to the homotopy cardinality of the incidence bialgebra of $\tsbar^{\pfsmc} \Tpc{}{\fsg}{\ass}$.
\end{theorem}
 
  %----------------XLPX--------------XLPX----------------XLPX----------------XLPX
 
The \emph{noncommutative linear plethystic bialgebra with non-commutating variables} $\LXplethBX$ is the same as $\XplethBX$ but without automorphisms on $\omega$. We write $a_{\omega}$ for its generators. Contrary to what it may seem, the noncommutativity simplifies the explicit  formula for the comultiplication of the generators. Denote by $|w|$ the length of a word. Let also $W^W_n$ be the set of length $n$ words of words of $W$. Finally, for $k\in\N$ and $\omega=\omega_1\dots\omega_n\in W$, define the $k$th Verschiebung operator as   $$k\omega=(k\omega_1)\dots(k\omega_n).$$
\begin{proposition} 
The comultiplication of $\XplethBX$ is given by 
$$\Delta(a_{\nu})=\sum_{\omega\in W}\sum_{\kappa\in W^W_{|\omega|}}T_{\nu,\omega}^{\kappa}\left(\prod_{i=1}^{|\omega|}a_{\kappa_i}\right)\otimes a_{\omega},$$
where 
\[
    T_{\nu,\omega}^{\pmb{\kappa}} = \left\{\begin{array}{lr}
        1 & \text{if } \nu=\sum_{i=1}^n\omega_i\kappa_i\\
        0 & \text{otherwise.}
        \end{array}\right.
  \]
\end{proposition}
\noindent This proposition is analogous to \cite[Proposition 3.3]{Cebrian}.

\begin{theorem} \label{thm:LXplethBX}
 The noncommutative linear plethystic bialgebra with noncommuting variables $\LXplethBX$ is isomorphic to the homotopy cardinality of the incidence bialgebra of $\tsbar \Tpc{}{\fsg}{\ass}$.
 \begin{proof}
Notice that $\tsbar_1 \Tpc{}{\fsg}{\ass}$ is discrete. Its elements are given by sequences of tuples 
$$(m^1_1,\dots,m^1_{n_1}),\dots,(m^k_1,\dots,m^k_{n_k})$$
 of elements of positive natural numbers (see Example \ref{ex:Tass}), but there is only the identity morphisms between them. Thus juxtaposition of sequences gives $\tsbar \Tpc{}{\fsg}{\ass}$ a (nonsymmetric) monoidal structure. Sequences containing one tuple are called connected, and form an algebra basis of the incidence bialgebra. The subgroupoid of connected sequences is denoted \label{tsbarconnectednsimplices}$\tsbar^{^\circ}_1\Tpc{}{\fsg}{\ass}$. It is clear that $\pi_0\tsbar^{^\circ}_1\Tpc{}{\fsg}{\ass}=\tsbar^{^\circ}_1\Tpc{}{\fsg}{\ass}$ is isomorphic to $W$, and that $\pi_0\tsbar_1\Tpc{}{\fsg}{\ass}=\tsbar_1 \Tpc{}{\fsg}{\ass}$ is isomorphic to $W^W$. Although $\pi_0\tsbar_1\Tpc{}{\fsg}{\ass}=\tsbar_1\Tpc{}{\fsg}{\ass}$ we keep using the notation $\delta_{\omega}$ for the isomorphism class of $\omega\in\tsbar_1\Tpc{}{\fsg}{\ass}$. It only remains to compute the comultiplication:
 $$\Delta(\delta_{\nu})=\sum_{\omega}\sum_{\kappa}|\iso(d_0\kappa,d_1\omega)_{\nu}|\delta_{\kappa}\otimes \delta_{\omega}.$$
 By the discussion above we only have to check that
 $$|\iso(d_0\kappa,d_1\omega)_{\nu}|=T_{\nu,\omega}^{\kappa},$$
 but this is clear because there is only one morphism between $d_0\kappa$ and $d_1\omega$ and fibering over $\nu$ means taking the subset of those morphisms that give $\nu$ after composing, hence $|\iso(d_0\kappa,d_1\omega)_{\nu}|=1$ if $d_1(\kappa,\omega)=\nu$ and $0$ otherwise, exactly as  $T_{\nu,\omega}^{\kappa}$.
\end{proof}
 
\end{theorem}

Let us move forward to power series in two variables. Again, all the results are also valid for any number of variables, but for simplicity and notation we have chosen to show the two variables case. 
%-----------------FDB2X-----------------FDB2X------------------FDB2X-----------------FDB2X

Let $\Q\ncpow{x,y}$ be the ring of formal power series in the noncommutative variables $x$ and $y$ with coefficients in $\Q$ without constant term. Elements of $\Q\ncpow{x,y}$ are written
$$F(x,y)=\sum_{\omega} f_{\omega}\omega,$$
where $\omega$ is a nonempty word in $x$ and $y$.
The set $\Q\ncpow{x,y}$ forms  a noncommutative monoid with substitution of power series.

 We define the \emph{Fa\`a di Bruno bialgebra in two noncommuting variables} $\IIXFdB$ as the free polynomial algebra $\Q\big[\{a^i_{\omega}\}\big]$ generated by the set maps
\begin{equation*}
\begin{tikzcd}[row sep=0pt]
a^i_{\omega}:\Q\ncpow{x,y}\cx\Q\ncpow{x,y}\arrow[r]& \Q&\\
\phantom{AAAA}(F^1,F^2)\arrow[r]&f_{\omega}^i&
\end{tikzcd}
\end{equation*}
together with   the counit given by $\epsilon (a_{\omega}^i)=a^i_{\omega}(x,y)$ and the comultiplication induced by substitution. 

\begin{theorem}\label{thm:2XFdB}
The Fa\`a di Bruno bialgebra in two noncommuting variables $\IIXFdB$ is isomorphic to the homotopy cardinality of the incidence bialgebra of $\tsbar^{\pfsmc} \ass_2$.
\end{theorem}

%----------------XFDB2X----------------XFDB2X----------------XFDB2X----------------XFDB2X

 We obtain the \emph{noncommutative Fa\`a di Bruno bialgebra in two noncommuting variables} $\XIIXFdB$ by taking above power series with coefficients in $R$.

\begin{theorem}\label{thm:X2XFdB}
The  noncommutative Fa\`a di Bruno bialgebra in two noncommuting variables $\XIIXFdB$ is isomorphic to the homotopy cardinality of the incidence bialgebra of $\tsbar \ass_2$.
\end{theorem}

%There are two more bialgebras, that in principle can not be obtained: OFDB2 and XFDB2

%----------------EP2------------------EP2---------------------EP2--------------------EP2

Finally, the \emph{exponential plethystic bialgebra in two variables} $\EIIplethB$ is the same as $\IIplethB$ but with exponential automorphisms $\autiv (\l)=\l_1!\l_2!\cdots$. The generators of this bialgebra are denoted $a_{\l}^i$.

\begin{theorem}\label{thm:E2plethB}
The  exponential plethystic bialgebra in two variables $\EIIplethB$ is isomorphic to the homotopy cardinality of the incidence bialgebra of $\tsbar \Tpc{}{\pfsmc}{\ass_2}$.
\end{theorem}

%-------------------Monoid------------------Monoid-------------------Monoid

\subsection{$\monoid$-plethysm and bialgebras from $\monoid$} \label{subsection:monoid}

In Subsection \ref{subsection:ass} we could have taken the locally finite monoid $(\N,\times)$ instead of $\ass$, since $\Tpc{\fsg}{}{\ass}=(\N,\times)$ (Example \ref{ex:Tass}). In fact, we have indirectly done so in the proof of Theorem \ref{thm:LXplethBX}. It is the case that the three plethystic bialgebras of Subsection \ref{subsection:ass} can be generalized to any locally finite monoid. In this section we explain the generalization of $\EplethB$, which arises from $\monoid$-plethysm, introduced by M\'endez and Nava \cite{Mendez-Nava} in the context of colored species.  

Let $\monoid$ be a locally finite monoid; this means that any $m\in \monoid$ there has a finite number of two-step factorizations $m=nk$. This is the same as the finite decomposition property of Cartier--Foata \cite{Cartier-Foata}. Consider the ring of formal power series $\Q\pow{x_m|m\in \monoid}$ without constant term. Following the same conventions as above, the set of variables $\{x_m\}_{m\in \monoid}$ is denoted $\x$. Elements of $\Q\pow{\x}$ are written
$$F(\x)=\sum_{\l\in\Lambda}\frac{f_{\l}}{\l!}\x^{\l},$$
where now the sum is indexed by the subset $\Lambda\subseteq \text{Hom}_{\set}(\monoid,\N)$ of maps with finite support,  and $\x^{\l}$ is the obvious monomial, for $\l\in\Lambda$. In this case $\l!=\prod \l_m!$.

The monoid structure of $\monoid$ defines an operation $x_n\oast x_m=x_{mn}$, which extends to a binary operation on $\Q\pow{\x}$ as
 \begin{align*}
 (G\oast F)(x_m|m\in \monoid):=&\,G(F_m|m\in \monoid),\;\;\;\;\;\text{where}\\
F_m(x_n|n\in \monoid):=&\,F(x_{mn}|n\in \monoid).&
\end{align*}
This substitution operation was introduced in \cite{Mendez-Nava} in the context of species colored over a monoid, although their conditions on the monoid are more restrictive. The main example comes from the monoid $(\N^+,\times)$, which gives ordinary plethysm. Another relevant example is $(\N,+)$, which gives $F_k(\x)=F(x_k,x_{k+1},\dots)$, which appears in \cite{Mendez}. The power series $F_m$  can be described by using the Verschiebung operators: for each $m\in \monoid$ we define the $m$th Verschiebung operator $V^m$ on $\text{Hom}_{\set}(\monoid,\N)$  as follows: for each $\l\in \text{Hom}_{\set}(\monoid,\N)$ and $n\in\monoid$,
$$V^m\l(n)=\sum_{mk=n}\l_k.$$
Clearly if $\monoid=(\N^+,\times)$ this gives the usual Verschiebung operators \cite{Nava,Nava-Rota,Cebrian}. The power series $F_m$ can be expressed as 
$$F_m(\x)=\sum_{\l}\frac{f_{\l}}{\autiv(\l)}\x^{V^m\l}.$$
As usual, we define the $\monoid$-plethystic bialgebra $\MplethB$ as the polynomial algebra $\Q\big[\{a_{\l}\}_{\l}\big]$ on the set maps $a_{\l}:\Q\pow{\x}\rightarrow \Q$ defined by $a_{\l}(F)=f_{\l}$, with comultiplication dual to plethystic substitution, that is 
$$\Delta(a_{\l})(F, G)=a_{\l}(G\oast F),$$
and counit given by $\epsilon(a_{\l})=a_{\l}(x_1)$.

What follows is devoted to express the comultiplication of $\MplethB$.
Consider a list $\pmb{\mu}\in \Lambda^n$ of $n$ infinite vectors, regarded as a representative element of a multiset $\overline{\pmb{\mu}}\in \Lambda^n/\mathfrak{S}_n$. We denote by $\rep(\pmb{\mu})\subseteq \mathfrak{S}_n$ the set of automorphisms that maps the list $\pmb{\mu}$ to itself. For example if $\pmb{\mu}=\{\alpha,\alpha,\beta,\gamma,\gamma,\gamma\}$ then $\rep(\pmb{\mu})$ has $2!\cdot 1!\cdot 3!$ elements. Notice that if $\pmb{\mu},\pmb{\mu'}\in\Lambda^n$ are representatives of the same multiset then there is an induced bijection $\rep(\pmb{\mu})\cong \rep(\pmb{\mu'})$. We may thus refer to $\rep(\pmb{\mu})$ for a multiset $\overline{\pmb{\mu}}\in \Lambda^n/\mathfrak{S}_n$ by taking a representative, since we are only interested in its cardinality. 
%\begin{remark} Observe that $\sum_n \Lambda^n/\mathfrak{S}_n \simeq \pi_0 T_1\sur$. Furthermore, the number of automorphisms of a representative element in $T_1\sur$ of the image of $\overline{\pmb{\mu}}$ under this bijection is precisely
%$$\autiv(\pmb{\mu})=| \rep(\pmb{\mu}) | \cdot\prod_{\mu\in \pmb{\mu}}\autiv(\mu).$$
%\end{remark}

Fix two infinite vectors, $\sigma,\l\in \Lambda$, and a list of infinite vectors $\pmb{\mu}\in \Lambda^n$, with $n=|\l|$. We define the set of $(\l,\pmb{\mu})-$\emph{decompositions} of $\sigma$ as
$$T_{\sigma,\l}^{\pmb{\mu}}:=\left\{p\colon \pmb{\mu}\xrightarrow{\;\;\sim\;\;} \sum_{m\in\monoid} \{1,\dots,\l_m\}\; |\;\sigma=\sum_{\mu\in\pmb{\mu}}V^{q(\mu)}\mu\right\},$$
where $p$ is a bijection of $n$-element sets and $q$ returns the index of $p(\mu)$ in the sum. A useful way to visualize an element of this set is as a placement of the elements of $\pmb{\mu}$ over a grid with $\l_m$ cells in the $m$th column such that if we apply $V^m$ to the $m$th column and sum the cells the result is $\sigma$. For example, if $\l=(\l_{m_1},\l_{m_2},\l_{m_3})=(2,1,3)$ and $\pmb{\mu}=\{\alpha,\alpha,\beta,\gamma,\gamma,\gamma\}$ the placement

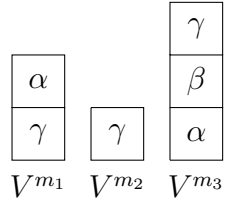
\begin{center}
\begin{tikzpicture}[scale=0.7]
\draw (1,1)--(1,0) -- (0,0) -- (0,1)--(1,1)--(1,2)--(0,2)--(0,1);
\draw (1.5,1)--(2.5,1)--(2.5,0)--(1.5,0)--(1.5,1);
\draw (3,1)--(3,0);
\draw (3,0)--(4,0)--(4,1)--(3,1)--(3,2)--(4,2)--(4,1)  (3,2)--(3,3)--(4,3)--(4,2);
\node at (0.5,0.5) {$\gamma$}; \node at (0.5,1.5) {$\alpha$}; \node at (0.5,-0.5) {$V^{m_1}$};
\node at (2,0.5) {$\gamma$};  \node at (2,-0.5) {$V^{m_2}$};
\node at (3.5,0.5) {$\alpha$}; \node at (3.5,1.5) {$\beta$}; \node at (3.5,2.5) {$\gamma$}; \node at (3.5,-0.5){$V^{m_3}$};
\end{tikzpicture}
\end{center}

\noindent belongs to $T_{\sigma,\l}^{\pmb{\mu}}$ if $\sigma=V^{m_1}(\gamma+\alpha)+V^{m_2}(\gamma)+V^{m_3}(\alpha+\beta+\gamma)$, where the sum is a pointwise vector sum in $\Lambda$. Note that each such placement appears $|\rep (\pmb{\mu})|$ times in $T_{\sigma,\l}^{\pmb{\mu}}$. Observe also that if $\pmb{\mu},\pmb{\mu'}\in\Lambda^n$ are representatives of the same multiset then there is an induced bijection $T_{\sigma,\l}^{\pmb{\mu}}\cong T_{\sigma,\l}^{\pmb{\mu'}}$. We may thus refer to $T_{\sigma,\l}^{\pmb{\mu}}$ for a class $\overline{\pmb{\mu}}\in \Lambda^{|\l|}/\mathfrak{S}_{|\l|}$ by taking a representative, since we are only interested in its cardinality.

\begin{proposition}The comultiplication of $\MplethB$ is given by 
\begin{equation}\label{eq:pletform} 
\Delta(\sigma)=\sum_{\l}\sum_{\overline{\pmb{\mu}}} \frac{\autiv(\sigma)\cdot|T_{\sigma,\l}^{\pmb{\mu}}|}{\autiv(\l)\cdot \displaystyle\autiv(\pmb{\mu}) }\prod_{\mu\in \pmb{\mu}} a_{\mu}.
\end{equation}
\end{proposition}
\noindent This proposition is analogous to \cite[Proposition 3.3]{Cebrian}.
\begin{theorem}\label{thm:MplethB} The $\monoid$-plethystic bialgebra $\MplethB$ is isomorphic to the homotopy cardinality of the incidence bialgebra of $\tsbar \Tpc{}{\pfsmc}{\monoid}$.

\begin{proof}[Proof of \ref{thm:MplethB}]
Let us compute the homotopy cardinality of the incidence bialgebra of $\tsbar \Tpc{}{\pfsmc}{\monoid}$. First of all, notice that the elements of $\tsbar_1 \Tpc{}{\pfsmc}{\monoid}=\pfsmc \Tpc{}{\pfsmc}{\monoid}$ are sequences of tuples 
$$(m^1_1,\dots,m^1_{n_1}),\dots,(m^k_1,\dots,m^k_{n_k})$$
 of elements of $\monoid$. Juxtaposition of sequences gives $\tsbar \Tpc{}{\pfsmc}{\monoid}$ a symmetric monoidal structure. Sequences containing only one tuple are called connected, and form an algebra basis of the incidence bialgebra. Since the morphisms between tuples are given by permutations, it is clear that the set of isomorphism classes of connected elements $\pi_0 \tsbar^{^\circ}_1 \Tpc{}{\pfsmc}{\monoid}$ is isomorphic to $\Lambda$, the subset of $\text{Hom}_{\set}(\monoid,\N)$ consisting of maps with finite support. 
 The isomorphism class $\delta_{\l}$ of  a connected element $\l$
is given by the map $\monoid\xrightarrow{\l}\N$ such that $\l_m$ is the number of times $m$ appears in $\l$. Be aware that the same notation is used for either the connected elements of   $ \tsbar_1 \Tpc{}{\pfsmc}{\monoid}$ and the maps representing their isomorphism class. Moreover, 
$$\pi_0 \tsbar_1 \Tpc{}{\pfsmc}{\monoid}\cong \sum_n \Lambda^n//\mathfrak{S}_n,$$
so that an element $\tau \in \pi_0 \tsbar_1 \Tpc{}{\pfsmc}{\monoid}$ may be identified with a multiset $\overline{\pmb{\mu}}$ of maps. With these identifications we clearly have
$$|\aut (\l)|=\l!   \;\;\;\text{ and }\;\;\; |\aut (\tau)|=\autiv (\pmb{\mu}),$$
for $\l$ connected and $\tau$ not necessarily connected. The left hand sides refer to the automorphisms groups in $ \tsbar_1 \Tpc{}{\pfsmc}{\monoid}$, while the right hand sides were introduced above.

\renewcommand{\arraystretch}{1}
The assignment 
\begin{center}
\begin{tabular}{rll}
$\Q_{\pi_0 \tsbar_1 \Tpc{}{\pfsmc}{\monoid}}$& $\longrightarrow$ &$\EplethB$\\
$\delta_{\l}$& $\longmapsto$ & $a_{\l}$\\
$\delta_{\l+\mu}=\delta_{\l}\delta_{\mu}$&$\longmapsto$ & $a_{\l}a_{\mu}$,
\end{tabular}
\end{center} 
for $\l$ and $\mu$ connected, defines  an isomorphism of algebras. Notice that $\l+\mu$ is the monoidal sum in $\tsbar_1\Tpc{}{\pfsmc}{\monoid}$, which does not correspond to the pointwise sum of their corresponding infinite vectors, since it has two connected components.

 We have to compute the coproduct in $\Q_{\pi_0\tsbar_1 \Tpc{}{\pfsmc}{\monoid}}$. It is enough to compute it for connected elements. From Lemma \ref{SegalCom} we have, for $\sigma$ connected,
 \begin{equation}\label{eq:MplethBcom}
\Delta(\delta_{\sigma})=\sum_{\l} \sum_{\tau}\frac{|\iso(d_0\tau,d_1\l)_{\sigma}|}{|\aut(\l)||\aut(\tau)|}\delta_{\tau}\otimes\delta_{\l}.
\end{equation}
In view of the discussion above, it only remains to show that 
$$|\iso(d_0\tau,d_1\l)_{\sigma}|=\autiv (\sigma)\cdot|T_{\sigma,\l}^{\pmb{\mu}}|.$$
Consider representatives for $\tau$ and $\l$, 
\begin{align*}
\tau &= \big((m^1_1,\dots,m^1_{n_1}),\dots,(m^k_1,\dots,m^k_{n_k})\big)\\
\l&=(m_1,\dots,m_k),
\end{align*}
then $d_0\tau=d_1\l=(1,\dots,1)$, $k$ times. This means that  
$$\iso(d_0\tau,d_1\l)=\aut (1,\dots,1)\cong \mathfrak{S}_k.$$
Any element $\phi \in \iso(d_0\tau,d_1\l)$ induces a map between sequences
$$\big((m^1_1,\dots,m^1_{n_1}),\dots,(m^k_1,\dots,m^k_{n_k})\big)\xrightarrow{\;\;\phi\;\;}(m_1,\dots,m_k).$$
We express it as a permutation on $\tau$ and write
$$\phi(\tau)=\big((m^{\phi(1)}_1,\dots,m^{\phi(1)}_{n_{\phi(1)}}),\dots,(m^{\phi(k)}_1,\dots,m^{\phi(k)}_{n_{\phi(k)}})\big).$$
Now, consider the subset 
$$\Big\{\phi \in \iso(d_0\tau,d_1\l)\; |  \; d_1\big((\phi(\tau),\l)\big)\simeq \sigma\Big\}.$$
It is straightforward to see that this subset is isomorphic to 
$$T_{\sigma,\l}^{\pmb{\mu}}:=\left\{p\colon \pmb{\mu}\xrightarrow{\;\;\sim\;\;} \sum_{m\in M} \{1,\dots,\l_m\}\; |\;\sigma=\sum_{\mu\in\pmb{\mu}}V^{q(\mu)}\mu\right\},$$
under the identifications $\tau \rightarrow \pmb{\mu}$ and  $\phi\rightarrow p$. The summation of the Verschiebung operators is precisely composition of $\phi(\tau)$ and $\l$. Finally, since $\iso(d_0\tau,d_1\l)_{\sigma}$ is a homotopy fiber we have that
$$\iso(d_0\tau,d_1\l)_{\sigma}\cong \aut(\sigma)\times \Big\{\phi \in \iso(d_0\tau,d_1\l)\; |  \; d_1\big((\phi(\tau),\l)\big)\simeq \sigma\Big\}\cong \aut(\sigma)\times T_{\sigma,\l}^{\pmb{\mu}}$$
and therefore 
$$|\iso(d_0\tau,d_1\l)_{\sigma}|=\autiv (\sigma)\cdot |T_{\sigma,\l}^{\pmb{\mu}}|,$$
as we wanted to see.
 \end{proof}
This proves also Theorem \ref{thm:EplethB} by taking the monoid $(\N^+,\times)$.
\end{theorem}

%---------------------------------------------------------------------------------------------------------------------------------------
%-------------------------------------RELATION WITH OLD T-----------------------------------------------------------------
%-------------------------------------RELATION WITH OLD T-----------------------------------------------------------------
%-------------------------------------RELATION WITH OLD T-----------------------------------------------------------------
%-------------------------------------RELATION WITH OLD T-----------------------------------------------------------------
%-------------------------------------RELATION WITH OLD T-----------------------------------------------------------------
%-------------------------------------RELATION WITH OLD T-----------------------------------------------------------------
%-------------------------------------RELATION WITH OLD T-----------------------------------------------------------------
%---------------------------------------------------------------------------------------------------------------------------------------

\section{Relation with $T\sur$}\label{seccion:oldTrelation}
We end this work by exploring the relations between the $\tcons$-construction and the simplicial groupoid $T\sur$ \label{TS} of \cite{Cebrian}. We first recall what this simplicial groupoid looks like. Then we prove that $T\sur$ and $\tsbar \Tpc{}{\pfsmc}{\comm}$ are equivalent simplicial groupoids. This proves in particular Theorem \ref{thm:plethB}. Finally we  show that the operads of Section \ref{section:plethysmsandoperads} arising from $\ass$ or $\comm$ are also equivalent to similar simplicial groupoids.

\subsection{The simplicial groupoid $T\sur$}
It can be defined through a general construction \cite{Cebrian}, but we content ourselves with a brief description: objects in $T_1\sur$ and $T_2\sur$ (1 and 2-simplices of $T\sur$) are, respectively, diagrams of finite sets and surjections
%---D-------I------A-------G-----R---------A-------M-------
 
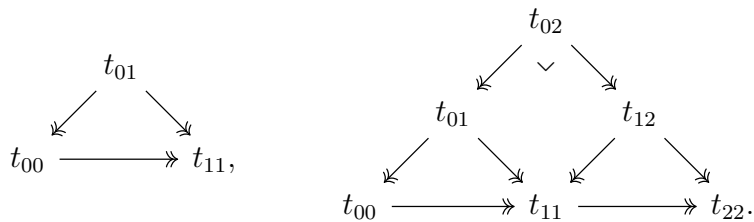
\begin{equation*}
 \begin{tikzcd}[sep={3em,between origins}]
 &{t_{01}}\arrow[ld,twoheadrightarrow] \arrow[rd,twoheadrightarrow]&\\
 	     {t_{00}} \arrow[rr,twoheadrightarrow]& &{t_{11}},
\end{tikzcd}
 \hspace{30pt}
\begin{tikzcd}[sep={3em,between origins}]    
& & {t_{02}}\dpbk \arrow[ld,twoheadrightarrow] \arrow[rd,twoheadrightarrow]& &\\                             
&{t_{01}}\arrow[ld,twoheadrightarrow] \arrow[rd,twoheadrightarrow]& &{t_{12}}\arrow[ld,twoheadrightarrow] \arrow[rd,twoheadrightarrow] & \\
	     {t_{00}} \arrow[rr,twoheadrightarrow]& &{t_{11}} \arrow[rr,twoheadrightarrow]& &{t_{22}}.      
\end{tikzcd}
\end{equation*}
%--------------------------------------------------------------
Morphisms of such shapes are levelwise bijections $t_{ij}\xrightarrow{\sim} t_{ij}'$ compatible with the diagram. In general $T_n\sur$ is an analogous pyramid, with $t_{0n}$ in the peak, all of whose squares are pullbacks of sets. The  face maps $d_i$ remove all the sets containing an $i$ index, and the  degeneracy maps $s_i$ repeat the $i$th diagonals. Diagrams whose last set is
  singleton are called \emph{connected}. It is not difficult to see that $T\sur$ is a Segal groupoid \cite{Cebrian}.

We now prove that $T\sur\simeq \tsbar \Tpc{}{\pfsmc}{\comm}$. We prove the equivalence by constructing an intermediate simplicial groupoid. More precisely, we find a subsimplicial groupoid of $T\sur$ which is equivalent to $T\sur$ and isomorphic to $\tsbar \Tpc{}{\pfsmc}{\comm}$. First of all we need some notation and elementary results.
 
 \begin{definition} \label{def:monotonesquare}Consider the category of finite ordinals $[n]=\{1,\dots,n\}$ and set maps. We say that a square
 %---D-------I------A-------G-----R---------A-------M-------
 \begin{equation}\label{Tpullback}
\begin{tikzcd}
{[m]} \arrow[d,"p"'] \arrow[r,"q"] \rdpbk&{[n]}  \arrow[d,"f"] \\
              {[l]} \arrow[r,"g"']&{[k]}
\end{tikzcd}   
\end{equation}
%--------------------------------------------------------------
is monotone if it is a pullback of sets, $p$ is monotone and $q$ is monotone at each fiber over $p$, that is, $q_{|p^{-1}(i)|}$ is monotone for all $i\in [l]$.

\end{definition}
 \begin{lemma}\label{lemma:monotonesquare} Consider the category of finite ordinals and set maps.
 \begin{enumerate}[(i)]
 \item The class of monotone pullback squares is closed under composition of squares.
 \item Given a diagram $[l]\xrightarrow{g}[k] \xleftarrow{f} [n]$, there is a unique monotone square as \ref{Tpullback}.
 \end{enumerate}
 \begin{proof} (i) is clear, and (ii) follows fom the fact that we can totally order the pullback,
 $$P=\sum_{i\in[k]} [l]_i \times [n]_i,$$
 by using the orders of $[l]$ and $[n]$. That is, given $a,b\in P$, then $a<b$ if $p(a)<p(b)$ or $p(a)=p(b)$ and $q(a)<q(b)$.
 \end{proof}
 \end{lemma}
Consider the full subsimplicial groupoid $\mathcal{V}\subseteq T\sur$ containing only the simplices whose entries are the finite ordinals $[k]$, whose left-down-arrows and right-arrows are monotone surjections and whose left-down arrows are fiber-monotone in the sense of Definition \ref{def:monotonesquare}, and whose pullback squares are monotone.
Note that Lemma \ref{lemma:monotonesquare} ensures that $\mathcal{V}$ is well defined, meaning that the inclusion $\mathcal{V}\hookrightarrow T\sur$ is a morphism of simplicial groupoids.
\begin{lemma} \label{lemma:MeqTS}$\mathcal{V}\hookrightarrow T\sur$ is an equivalence of simplicial groupoids.
\begin{proof} Given an element of $T_n\sur$,
%---D-------I------A-------G-----R---------A-------M-------
\begin{equation*}
\begin{tikzcd}[sep={3em,between origins}]
& &  \arrow[ld,dashed] && \arrow[rd,dashed]&&\\
&t_{01}\arrow[ld] \arrow[rd]& && & t_{n-1,n} \arrow[ld] \arrow[rd]&\\
             t_{00} \arrow[rr]&&t_{11}\arrow[rr,dashed]&&t_{n-1,n-1} \arrow[rr]&&t_{nn},
\end{tikzcd}
\end{equation*}
%--------------------------------------------------------------
it is clear we can choose an ordering of the $t_{ii}$ and the $t_{i,i+1}$ such that all the arrows between them are monotone. Then by Lemma \ref{lemma:monotonesquare} there exists a unique ordering on the rest of the $t_{ij}$'s making the pullback squares monotone. Hence the inclusion is essentially surjective. Since we have taken the full inclusion, the automorphism group of any elment of $\mathcal{V}_n$ is equal to to its automorphism group as an element of $T_n\sur$. Hence the inclusion is an equivalence.
\end{proof}
\end{lemma}
Note that in $\mathcal{V}$ the uniqueness of the monotone squares implies that the Segal maps are in fact isomorphisms, 
$$\mathcal{V}_n\cong \mathcal{V}_1\cx_{\mathcal{V}_0} \cdots \cx_{\mathcal{V}_0}  \mathcal{V}_1.$$
In other words, there is a well-defined composition $d_1\colon \mathcal{V}_1\cx_{\mathcal{V}_0} \mathcal{V}_1\rightarrow \mathcal{V}_1$. In view of this we may drop the elements $t_{ij}$ with $j\ge i+2$ from the diagrams.

\begin{lemma} \label{lemma:MeqbarT}Let $\mathsf{V}$ be the operad whose $n$-ary operations are diagrams
%---D-------I------A-------G-----R---------A-------M-------
\begin{equation*}
\begin{tikzcd}[sep={3em,between origins}]
&{[m]}\arrow[ld] \arrow[rd]& \\
             {[n]} \arrow[rr]&&1 
\end{tikzcd}
\end{equation*}
%--------------------------------------------------------------
where $[m]\twoheadrightarrow [n]$ is monotone, whose morphisms are entrywise bijections, and whose composition is given by monotone pullback squares. Then $\mathcal{V}\cong \tsbar\mathsf{V}$.
\begin{proof} The isomorphism is given by
%---D-------I------A-------G-----R---------A-------M-------
\begin{equation*}
\begin{tikzcd}[sep={3em,between origins}]
&{[m_1]}\arrow[ld] \arrow[rd]& & &{[m_k]}\arrow[ld] \arrow[rd]&\\
             {[n_1]} \arrow[rr]&&1,\arrow[r,dash,dotted]&,{[n_2]} \arrow[rr]&&1,
\end{tikzcd}
\mapsto
\begin{tikzcd}[sep={3em,between origins}]
&{[m_1+\dots+m_k]}\arrow[ld] \arrow[rd]&\\
			{[n_1+\dots+n_k]} \arrow[rr]&&{[k]}
\end{tikzcd}
\end{equation*}
%--------------------------------------------------------------
at the level of $1$-simplices and similarly in general. 
\end{proof}
\end{lemma}

\begin{lemma} $\mathsf{V}$ is isomorphic to $\Tpc{}{\pfsmc}{\comm}$.
\begin{proof}
An operation of $\Tpc{}{\pfsmc}{\comm}$ is a family of operations of $\comm$, which is equivalent to a monotone surjection $[m]\twoheadrightarrow [n]$. It is also clear that morphsims between operations of $\Tpc{}{\pfsmc}{\comm}$ are the same as morphisms in $\mathsf{V}$. Thus we only need to see that composition coincides. Let us denote by $x$ the unique $x$-ary operation of $\comm$. Thus a general element of $\Tpc{}{\pfsmc}{\comm}$ is a tuple $(x_1,\dots,x_n)$. By definition of the $\tcons$-construction
$$(x_1,\dots,x_n)\oast ((y_1^1,\dots,y_{k_1}^1),\dots,(y_1^n,\dots,y_{k_n}^n))=(y_1^1\cdot x_1,\dots,y_{k_1}^1\cdot x_1,\dots,y_1^n\cdot x_n,\dots,y_{k_n}^n\cdot x_n),$$
which is nothing but the pullback 
%---D-------I------A-------G-----R---------A-------M-------
\begin{equation*}
\begin{tikzcd}[sep={4em,between origins}]
& & \left[\sum_{i,j}y_j^ix_i\right] \arrow[ld,twoheadrightarrow] \arrow[rd,twoheadrightarrow] \dpbk& &\\
&\left[\sum_{i,j}y_j^i\right]\arrow[ld,twoheadrightarrow] \arrow[rd,twoheadrightarrow]& &\left[\sum_ix_i\right]\arrow[ld,twoheadrightarrow] \arrow[rd,twoheadrightarrow]&\\
             \left[\sum_i k_i^i\right] \arrow[rr,twoheadrightarrow]&&\phantom{a}[n]\phantom{a} \arrow[rr,twoheadrightarrow]&&1,
\end{tikzcd}
\end{equation*}
%--------------------------------------------------------------
the composition of their corresponding operations in $\mathsf{V}$.
\end{proof}
\end{lemma}
\begin{proposition} \label{prop:TSTSymequivalence}
The simplicial groupoids $T\sur$ and $\tsbar\Tpc{}{\pfsmc}{\comm}$ are equivalent.
\begin{proof} It is direct from Lemmas \ref{lemma:monotonesquare}, \ref{lemma:MeqTS} and \ref{lemma:MeqbarT}.
\end{proof}
\end{proposition}

\begin{example} \label{ex:TSTSym2simplex}
Consider the following $2$-simplex of $\tsbar^{\fsmc}\Tpc{}{\pfsmc}{\comm}$:
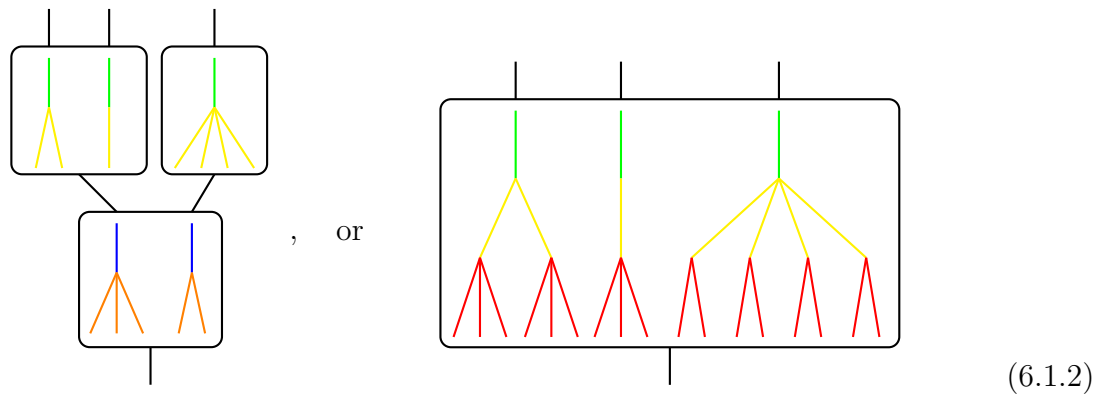
\begin{equation}\label{ex:2simplexBTSym}
\begin{tikzpicture}[grow=down,level distance=23pt,thick]
\tikzstyle{level 1}=[sibling distance=10pt]
%\tikzstyle{level 2}=[sibling distance=15pt]
%\tikzstyle{every node}=[fill=red!60,circle,inner sep=1pt]
\node at (0,0){}
child [blue]{
child [orange]
child [orange]
child [orange]
};
\node at (1,0){}
child [blue]{
child [orange]
child [orange]
};
\draw [rounded corners] (-0.5,-1.8)--(-0.5,0)--(1.4,0)--(1.4,-1.8)--cycle;
\draw (0,0)--(-0.5,0.5); \draw (1,0)--(1.3,0.5); \draw (0.45,-1.8)--(0.45,-2.3);
\draw [rounded corners] (0.4,0.5)--(0.4,2.2)--(-1.4,2.2)--(-1.4,0.5)--cycle;
\draw [rounded corners] (0.6,0.5)--(0.6,2.2)--(2,2.2)--(2,0.5)--cycle;
\node at (-0.9,2.2){}
child [green]{
child [yellow]
child [yellow]
};
\node at (-0.1,2.2){}
child [green]{
child [yellow]
};
\node at (1.3,2.2){}
child [green]{
child [yellow]
child [yellow]
child [yellow]
child [yellow]
};
\draw (-0.9,2.2)--(-0.9,2.7); \draw (-0.1,2.2)--(-0.1,2.7); \draw (1.3,2.2)--(1.3,2.7);
\end{tikzpicture}
\begin{tikzpicture}
\node at (-0.2,-1.9){};
\node at (1.8,0){};
\node at (0,-0.1){,};
\node at (0.75,0){or};
\end{tikzpicture}
\begin{tikzpicture}[grow=down,level distance=30pt,thick]
\tikzstyle{level 1}=[sibling distance=27pt]
\tikzstyle{level 2}=[sibling distance=10pt]
\node at (0,0){}
child[green] {}{
	child[yellow] {  child[red] child[red] child[red] }
	child [yellow]{ child[red] child[red] child[red]}
};
\node at (1.4,0){}
child[green]{
	child[yellow] { child[red] child[red] child[red] }
};
\tikzstyle{level 1}=[sibling distance=22pt]
\node at (3.5,0){}
child[green] {}{
	child[yellow] { child[red] child[red]}
	child[yellow] { child[red] child[red]}
	child[yellow] { child[red] child[red]}
	child[yellow] {child[red] child[red]}
};
\draw [rounded corners] (-1,-3.3)--(-1,0)--(5.1,0)--(5.1,-3.3)--cycle;
\draw (0,0)--(0,0.5); \draw (1.4,0)--(1.4,0.5); \draw  (3.5,0)--(3.5,0.5);
\draw (2.05,-3.3)--(2.05,-3.8);
\end{tikzpicture}
\end{equation}
We use colors here only to make the comparison more pleasant, but of course this is not a colored operad. This $2$-simplex corresponds, in $T\sur$, to
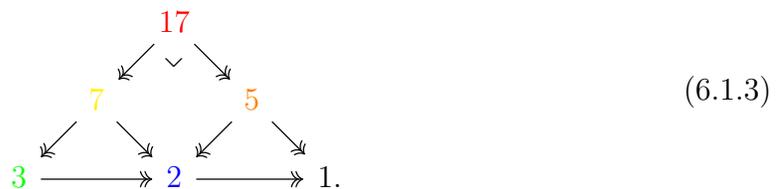
\begin{equation}\label{ex:2simplexBS}
\begin{tikzcd}[sep={2.5em,between origins}]
& & {\color{red}17} \arrow[ld,twoheadrightarrow] \arrow[rd,twoheadrightarrow] \dpbk& &\\
&{\color{yellow}7}\arrow[ld,twoheadrightarrow] \arrow[rd,twoheadrightarrow]& &{\color{orange}5}\arrow[ld,twoheadrightarrow] \arrow[rd,twoheadrightarrow]&\\
            {\color{green}3} \arrow[rr,twoheadrightarrow]&&{\color{blue}2}\arrow[rr,twoheadrightarrow]&&1.
\end{tikzcd}
\end{equation}
It is opportune in this example to show that indeed the opposite convention comes out more naturally in order to interpret $T\sur$ as an operad. First of all, observe that at the level of finite sets and surjections the Verschiebung operators~\ref{eq:classicalVerschiebung} can be regarded as a scalar multiplication, 
$$V^S(X\rightarrow B)=S\times X\rightarrow X\rightarrow B,$$
in the sense that if $\l$ represents the class of $X\rightarrow B$, then $V^{|S|}\l$ represents the class of $V^S(X\rightarrow B)$. Under this perspective we can write the information on~\eqref{ex:2simplexBS} as
\begin{equation}\label{eq:exV}
{\color{red}17}\twoheadrightarrow {\color{green}3}=({\color{red}9}\twoheadrightarrow{\color{green}2})+({\color{red}8}\twoheadrightarrow{\color{green}1})={\color{blue}V}^{\color{orange}3}({\color{yellow}3}\twoheadrightarrow{\color{green}2})+{\color{blue}V}^{\color{orange}2}({\color{yellow}4}\twoheadrightarrow{\color{green}1})=({\color{orange}3}\times{\color{yellow}3}\twoheadrightarrow{\color{green}2})+({\color{orange}2}\times{\color{yellow}4}\twoheadrightarrow{\color{green}1}),
\end{equation}
We can clearly see this in~\eqref{ex:2simplexBTSym}. On the contrary, it is not difficult to check that without the opposite convention Equation~\eqref{eq:exV} would rather appear as 
$${\color{red}17}\twoheadrightarrow {\color{green}3}=({\color{red}9}\twoheadrightarrow{\color{green}2})+({\color{red}8}\twoheadrightarrow{\color{green}1})=\big(({\color{yellow}2}\times{\color{orange}3}\twoheadrightarrow{\color{green}1})+({\color{yellow}1}\times{\color{orange}3}\twoheadrightarrow{\color{green}1})\big)+({\color{yellow}4}\times{\color{orange}2}\twoheadrightarrow{\color{green}1}).$$
\end{example}

%%%%%%%%%%%%%%%%%%%%%%%%%%%%%%%%%%%%%%%%%%%%%%%%%%%%%%%%%
\subsection{Other $T\sur$-like simplicial groupoids}
We now present other equivalences between variations of  $T\sur$ and some of the bar constructions treated before. First of all we introduce some notation: monotone surjections between ordered sets are denoted $\begin{tikzcd}[column sep=1.5em]a\arrow[r, twoheadrightarrow, pos=0.5, "\bullet" on top]&b.\end{tikzcd}$ We call \emph{linear surjection}  $\begin{tikzcd}[column sep=1.5em]a\arrow[r, twoheadrightarrow,pos=0.5, "\circ" on top]&b\end{tikzcd}$ a surjection between finite sets $f: a\twoheadrightarrow b$ with an order on $f^{-1}(r)$ for each $r\in b$. 

Notice that the composite of two monotone surjections is again a monotone surjection, and the composite of two linear surjections is again a linear surjection, with the obvious order. Moreover, given pullback squares
 \begin{equation*}
\begin{tikzcd}[sep={3.5em,between origins}]
\phantom{\cdot} \arrow[d,twoheadrightarrow,pos=0.5, "\bullet" on top,"p"'] \arrow[r,twoheadrightarrow] \rdpbk&\phantom{\cdot}  \arrow[d,twoheadrightarrow,pos=0.5, "\bullet" on top,"f"] \\
            \phantom{\cdot} \arrow[r,twoheadrightarrow]&\phantom{\cdot}
\end{tikzcd}  
\phantom{aaa}
\text{and}
\phantom{aaa}
\begin{tikzcd}[sep={3.5em,between origins}]
\phantom{\cdot} \arrow[d,twoheadrightarrow,pos=0.5, "\circ" on top,"q"'] \arrow[r,twoheadrightarrow] \rdpbk&\phantom{\cdot}  \arrow[d,twoheadrightarrow,pos=0.5, "\circ" on top,"g"] \\
            \phantom{\cdot} \arrow[r,twoheadrightarrow]&\phantom{\cdot}
\end{tikzcd},
\end{equation*}
we say that $p$ and $f$ are \emph{compatible} if the order of $p$ is induced by the order of $f$, in the sense of Lemma~\ref{lemma:monotonesquare}. Similarly, we say that $q$ and $g$ are compatible if the order of $q$ is induced by the order of $g$.

The proofs of all the following results are similar to the one of Proposition~\ref{prop:TSTSymequivalence}. To avoid repetitiveness we give only intuitive explanations.

\begin{example}\label{ex:TAsslineartransversal}The simplicial groupoid $\tsbar^{\fsmc}\Tpc{}{\pfsmc}{\ass}$ is equivalent to the simplicial groupoid constructed as $T\sur$ but with the additional structure that all  the left-down surjections  are linear and compatible. Morphisms are order-preserving levelwise bijections. Hence the $1$-simplices are diagrams
%---D-------I------A-------G-----R---------A-------M-------
\begin{equation*}
\begin{tikzcd}[sep={3em,between origins}]
& t_{01} \arrow[ld, twoheadrightarrow,pos=0.5, "\circ" on top]  \arrow[rd,twoheadrightarrow]& \\
             t_{00} \arrow[rr,twoheadrightarrow]& &t_{11}.
\end{tikzcd}
\end{equation*}
%--------------------------------------------------------------
The isomorphism classes of connected diagrams are again infinite vectors $\l=(\l_1,\l_2,\dots)$ as in $T\sur$, and the number of automorphisms of a connected element of class $\l$ is precisely $\l_1!\cdot \l_2! \cdots $, since $t_{01}$ is fixed. 

%Recall from Section~\ref{partitionals}  the definition of linear transversal. 
%In the language of surjections, a linear partition corresponds to a surjection $\sigma\colon E\twoheadrightarrow B$ with a partial order on $E$ consisting of linear orders on each fiber of $\sigma$. We shall name these surjections \emph{linear}, and denote them $\begin{tikzcd}[column sep=1.5em]\sigma\colon E\arrow[r, twoheadrightarrow, pos=0.5, "\circ" on top]&B.\end{tikzcd}$ A linear transversal of $\sigma$ corresponds to a transversal 
%\begin{equation}\label{eq:lineartransversal}
%\begin{tikzcd}[sep={3em,between origins}]
%& & E\arrow[dl,twoheadrightarrow,"\pi"',pos=0.5, "\circ" on top]\arrow[dr,twoheadrightarrow,"\tau"]\arrow[ddll,twoheadrightarrow, bend right,"\sigma"',pos=0.5, "\circ" on top]\dpbk& &\\
%&S\arrow[dl,twoheadrightarrow,"\gamma",pos=0.5, "\circ" on top]\arrow[dr,twoheadrightarrow]& &X\arrow[dl,twoheadrightarrow, "\theta", pos=0.5, "\circ" on top]\arrow[dr,twoheadrightarrow]&\\
%B\arrow[rr,twoheadrightarrow]&&I\arrow[rr,twoheadrightarrow]&&1
%\end{tikzcd}
%\end{equation}
%such that all the left-down surjections are linear and whose linear orders are compatible. By compatible we mean that the order of $\sigma$ is the one induced by the orders of $\pi$ and $\gamma$ and the order of $\pi$ is induced by the one of $\theta$ along the pullback. 
%
%We can clearly recognize this as an object in the groupoid of $2$-simplices of $\tsbar^{\fsmc}\Tpc{}{\pfsmc}{\ass}$.
%

\end{example}

\begin{example}\label{ex:BSTMAss}
The simplicial groupoid $\tsbar^{\fsmc}\Tpc{}{\fsg}{\ass}$ is equivalent to the simplicial grou-poid constructed as $T\sur$ but  with the additional structure that the left-down surjections and the right surjections are linear and compatible. Morphisms are order-preserving levelwise bijections. Hence the $1$-simplices are diagrams
%---D-------I------A-------G-----R---------A-------M--------
\begin{equation*}
\begin{tikzcd}[sep={3em,between origins}]
& t_{01} \arrow[ld, twoheadrightarrow, pos=0.5, "\circ" on top]  \arrow[rd,twoheadrightarrow]& \\
             t_{00} \arrow[rr, twoheadrightarrow, pos=0.5, "\circ" on top]& &t_{11}.
\end{tikzcd}
\end{equation*}
%--------------------------------------------------------------
Observe that for a connected element, $t_{00}$ is totally ordered. Thus the isomorphism classes of connected elements are given by words $\omega=\omega_1\omega_2\dots\omega_n$ where $\omega_i$ is the size of the $i$th fiber. It does not have any automorphisms, since $t_{01}$ and $t_{00}$ are fixed.
\end{example}

\begin{example}\label{ex:BSTMSym}
The simplicial groupoid $\tsbar^{\fsmc}\Tpc{}{\fsg}{\comm}$ is equivalent to the simplicial groupoid constructed as $T\sur$ but  with the additional structure that the right surjections  are linear and compatible. Morphisms are order-preserving levelwise bijections. Hence the $1$-simplices are diagrams
%---D-------I------A-------G-----R---------A-------M-------
\begin{equation*}
\begin{tikzcd}[sep={3em,between origins}]
& t_{01} \arrow[ld,twoheadrightarrow]  \arrow[rd,twoheadrightarrow]& \\
             t_{00} \arrow[rr, twoheadrightarrow,pos=0.5, "\circ" on top]& &t_{11}.
\end{tikzcd}
\end{equation*}
%--------------------------------------------------------------
Observe that for a connected element, $t_{00}$ is totally ordered. Thus the isomorphism classes of connected elements are given by finite words $\omega=\omega_1\omega_2\dots\omega_n$ where $\omega_i>0$ is the size of the $i$th fiber. It has  $\omega!:=\omega_1!\omega_2!\cdots \omega_n!$ automorphisms, since $t_{00}$ is fixed. 
\end{example}
\begin{example}
The simplicial groupoid $\tsbar^{\fm}\Tpc{}{\fsg}{\comm}$ is equivalent to the simplicial groupoid constructed as $T\sur$ but with the additional structure that the right surjections are monotone. Morphisms are order-preserving levelwise bijections. Hence the $1$-simplices are diagrams
%---D-------I------A-------G-----R---------A-------M-------
\begin{equation*}
\begin{tikzcd}[sep={3em,between origins}]
& t_{01} \arrow[ld,twoheadrightarrow]  \arrow[rd,twoheadrightarrow]& \\
             t_{00} \arrow[rr, twoheadrightarrow,pos=0.5, "\bullet" on top]& &t_{11}.
\end{tikzcd}
\end{equation*}
%--------------------------------------------------------------
Observe that for a connected element, $t_{00}$ is totally ordered. Thus the isomorphism classes of connected elements are given by finite words $\omega=\omega_1\omega_2\dots\omega_n$ where $\omega_i>0$ is the size of the $i$th fiber. It has  $\omega!:=\omega_1!\omega_2!\cdots \omega_n!$ automorphisms, since $t_{00}$ is fixed.  The difference between this simplicial groupoid and the one of Example~\ref{ex:BSTMSym} is that in this case $t_{11}$ is also ordered. As a consequence the monoidal structure is not symmetric, so that the resulting incidence bialgebra is not commutative.
\end{example}

\begin{example}
The simplicial groupoid $\tsbar^{\fm}\Tpc{}{\fsg}{\ass}$ is equivalent to the simplicial groupoid constructed as $T\sur$ but with the additional structure that the left-down surjections and the right surjections are monotone and compatible. Morphisms are order-preserving levelwise bijections. Hence the $1$-simplices are diagrams
%---D-------I------A-------G-----R---------A-------M-------
\begin{equation*}
\begin{tikzcd}[sep={3em,between origins}]
& t_{01} \arrow[ld, twoheadrightarrow,pos=0.5, "\bullet" on top]  \arrow[rd,twoheadrightarrow]& \\
             t_{00} \arrow[rr, twoheadrightarrow,pos=0.5, "\bullet" on top]& &t_{11}.
\end{tikzcd}
\end{equation*}
%--------------------------------------------------------------
Observe that for a connected element, $t_{00}$ is totally ordered. Thus the isomorphism classes of connected elements are given by words $\omega=\omega_1\omega_2\dots\omega_n$ where $\omega_i$ is the size of the $i$th fiber. It does not have any automorphisms, since $t_{01}$ and $t_{00}$ are fixed. Again, the difference between this simplicial groupoid and the one of Example~\ref{ex:BSTMAss} is that in this case $t_{11}$ is ordered.
\end{example}
\begin{example}\label{ex:TS2BTSym2}
Finally, the simplicial groupoid $\tsbar^{\fsmc} \Tpc{}{\pfsmc}{\comm_2}$  is equivalent to the simplicial groupoid constructed as $T\sur$ but with the additional structure that the objects are $2$-colored and the right-down surjections are color preserving. Morphisms are color-preserving levelwise bijections. 

For instance, the following $2$-simplex of $\tsbar^{\fsmc} \Tpc{}{\pfsmc}{\comm_2}$,
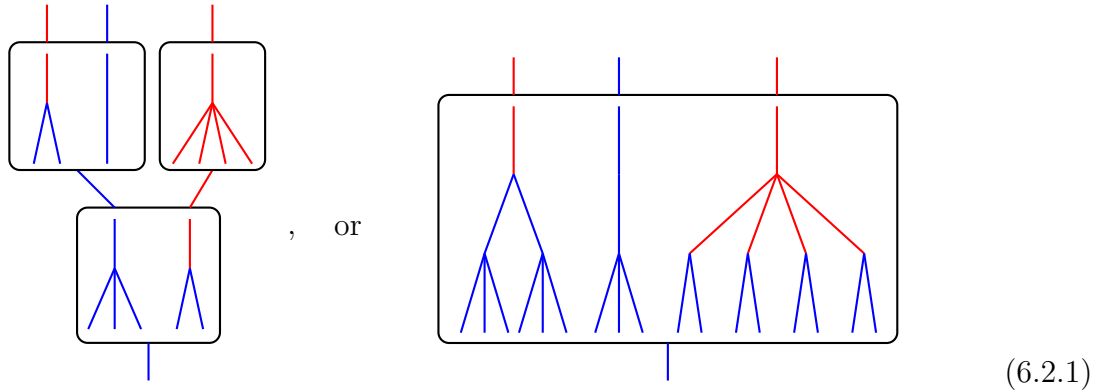
\begin{equation}\label{eq:2simplexBTSym2}
\begin{tikzpicture}[grow=down,level distance=23pt,thick]
\tikzstyle{level 1}=[sibling distance=10pt]
%\tikzstyle{level 2}=[sibling distance=15pt]
%\tikzstyle{every node}=[fill=red!60,circle,inner sep=1pt]
\node at (0,0){}
child[blue] {}{
child[blue]
child[blue]
child[blue]
};
\node at (1,0){}
child[red] {}{
child[blue]
child[blue]
};
\draw [rounded corners] (-0.5,-1.8)--(-0.5,0)--(1.4,0)--(1.4,-1.8)--cycle;
\draw [blue](0,0)--(-0.5,0.5); \draw [red](1,0)--(1.3,0.5); \draw [blue](0.45,-1.8)--(0.45,-2.3);
\draw [rounded corners] (0.4,0.5)--(0.4,2.2)--(-1.4,2.2)--(-1.4,0.5)--cycle;
\draw [rounded corners] (0.6,0.5)--(0.6,2.2)--(2,2.2)--(2,0.5)--cycle;
\node at (-0.9,2.2){}
child[red] {}{
child[blue]
child[blue]
};
\node at (-0.1,2.2){}
child [blue]{}{
child[blue]
};
\node at (1.3,2.2){}
child [red]{}{
child[red]
child[red]
child[red]
child[red]
};
\draw [red](-0.9,2.2)--(-0.9,2.7); \draw [blue](-0.1,2.2)--(-0.1,2.7); \draw [red](1.3,2.2)--(1.3,2.7);
\end{tikzpicture}
\begin{tikzpicture}
\node at (-0.2,-1.9){};
\node at (1.8,0){};
\node at (0,-0.1){,};
\node at (0.75,0){or};
\end{tikzpicture}
\begin{tikzpicture}[grow=down,level distance=30pt,thick]
\tikzstyle{level 1}=[sibling distance=22pt]
\tikzstyle{level 2}=[sibling distance=9pt]
\node at (0,0){}
child [red]{}{
	child[blue] { child[blue] child[blue] child[blue] }
	child[blue] { child[blue] child[blue] child[blue]}
};
\node at (1.4,0){}
child[blue] {}{
	child[blue] { child[blue] child[blue] child[blue] }
};
\node at (3.5,0){}
child [red]{}{
	child[red] { child[blue] child[blue]}
	child[red] {child[blue] child[blue]}
	child[red] { child[blue] child[blue]}
	child[red] { child[blue] child[blue]}
};
\draw [rounded corners] (-1,-3.3)--(-1,0)--(5.1,0)--(5.1,-3.3)--cycle;
\draw [red](0,0)--(0,0.5); \draw [blue](1.4,0)--(1.4,0.5); \draw [red] (3.5,0)--(3.5,0.5);
\draw [blue](2.05,-3.3)--(2.05,-3.8);
\end{tikzpicture}
\end{equation}
where now the colors do refer to the input and output colors, corresponds to the following $2$-simplex:
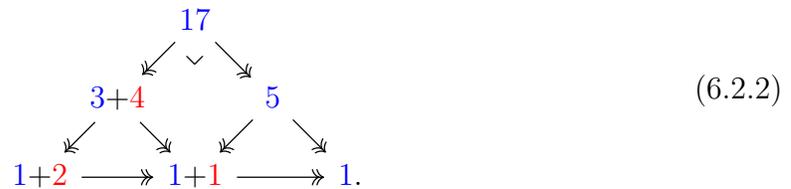
\begin{equation}\label{eq:2simplexTS2}
\begin{tikzcd}[sep={2.5em,between origins}]
& & {\color{blue}17} \arrow[ld,twoheadrightarrow] \arrow[rd,twoheadrightarrow] \dpbk& &\\
&{\color{blue}3}{+}{\color{red}4}\arrow[ld,twoheadrightarrow] \arrow[rd,twoheadrightarrow]& &{\color{blue}5}\arrow[ld,twoheadrightarrow] \arrow[rd,twoheadrightarrow]&\\
            {\color{blue}1}{+}{\color{red}2} \arrow[rr,twoheadrightarrow]&&{\color{blue}1}{+}{\color{red}1}\arrow[rr,twoheadrightarrow]&&{\color{blue}1}.
\end{tikzcd}
\end{equation}
Observe that indeed the right-down surjections are color-preserving. Notice also that if we had not used here the opposite convention the colors of~\eqref{eq:2simplexBTSym2} would not match the colors of~\eqref{eq:2simplexTS2} in such a direct way.
\end{example}

%--------------------------------------------------------------------------------------------------------------------------------
%----------------------------------------------APPENDICES---------------------------------------------------------------  
%----------------------------------------------APPENDICES---------------------------------------------------------------  
%----------------------------------------------APPENDICES---------------------------------------------------------------  
%----------------------------------------------APPENDICES---------------------------------------------------------------  
%----------------------------------------------APPENDICES---------------------------------------------------------------  
%----------------------------------------------APPENDICES---------------------------------------------------------------  
%----------------------------------------------APPENDICES---------------------------------------------------------------  
%--------------------------------------------------------------------------------------------------------------------------------

\appendices

\section{Appendices}

 \subsection{Axioms for internal category}\label{Cat_Axioms}
 Let $\ambcat$ be a cartesian category.
  A category $C$ internal to $\ambcat$ can be described by objects and arrows of $\ambcat$
\begin{center} 
 \begin{tikzcd}[column sep=small]
       & C_1\arrow[dl,"s"']\arrow[dr,"t"] &       \\
 C_0 &                                            & C_0
 \end{tikzcd}
 \begin{tikzcd}
 C_1\cx_{C_0} C_1 \arrow[r,"\comp"]& C_1\\
 \phantom{aaaaa}C_0 \arrow[r,"e"] &C_1
 \end{tikzcd}
 \end{center}
 where the pullback is taken along $C_1\xrightarrow{\;\;s\;\;}C_0\xleftarrow{\;\;t\;\;}C_1$, satisfying the following commutative diagrams:

 \begin{figure}[ht]
\centering
 \begin{subequations}
\begin{minipage}[b]{0.48\linewidth}
 \begin{equation}\label{Cat_d1d1}
 \begin{tikzcd}
 C_1\cx_{C_0}C_1\arrow[d,"p_1"']\arrow[r,"\comp"]&C_1\arrow[d,"s"]\\
 C_1\arrow[r,"s"']&C_0\\
 \end{tikzcd}
 \end{equation}
\end{minipage}
\quad
\begin{minipage}[b]{0.48\linewidth}
\begin{equation}\label{Cat_d0d1}
  \begin{tikzcd}
 C_1\cx_{C_0}C_1\arrow[d,"p_2"']\arrow[r,"\comp"]&C_1\arrow[d,"t"]\\
 C_1\arrow[r,"t"']&C_0\\
 \end{tikzcd}
 \end{equation}
\end{minipage}
 \end{subequations}
\end{figure}

 \vspace{50pt}
 
  \begin{figure}[ht]
\centering
 \begin{subequations}\label{Cat_topbottom}
\begin{minipage}[b]{0.48\linewidth}
 \begin{equation}\label{Cat_d1e}
  \phantom{aaa} \begin{tikzcd}[column sep=4em]
 C_0\arrow[dr,"\id"']\arrow[r,"e"]&C_1\arrow[d,"s"]\\
   &C_0
 \end{tikzcd}
 \end{equation}
\end{minipage}
\quad
\begin{minipage}[b]{0.48\linewidth}
  \begin{equation}\label{Cat_d0e}
 \phantom{aaa} \begin{tikzcd}[column sep=4em]
 C_0\arrow[dr,"\id"']\arrow[r,"e"]&C_1\arrow[d,"t"]\\
   &C_0
 \end{tikzcd}
 \end{equation}
\end{minipage}
 \end{subequations}
\end{figure}

 \begin{equation}\label{dgm:Cat_Associativity}
  \begin{tikzcd}[column sep=60pt, row sep=40pt]
 (C_1\cx_{C_0}C_1)\cx_{C_0}C_1\arrow[r,"\comp\cx_{C_0}C_1"]\arrow[d]&C_1\cx_{C_0}C_1\arrow[dd,"\comp"]\\
 C_1\cx_{C_0}(C_1\cx_{C_0}C_1)\arrow[d,"C_1\cx_{C_0}\comp"']&\\
 C_1\cx_{C_0}C_1\arrow[r,"\comp"']& C_1
 \end{tikzcd}
 \end{equation}
 
 \vspace{20pt}

   \begin{figure}[ht]
\centering
 \begin{subequations}\label{dgm:Cat_Unit}
\begin{minipage}[b]{0.475\linewidth}
\begin{equation}\label{dgm:Cat_Unita}
   \begin{tikzcd}[column sep=20pt, row sep=20pt]
 C_0\cx_{C_0}C_1\arrow[dr,"p_2"']\arrow[rr,"e\cx_{C_0}C_1"]&&C_1\cx_{C_0}C_1\arrow[ld,"\comp"]\\
   &C_1&
 \end{tikzcd}
 \end{equation}
\end{minipage}
\quad
\begin{minipage}[b]{0.475\linewidth}
\begin{equation}\label{dgm:Cat_Unitb}
   \begin{tikzcd}[column sep=20pt, row sep=20pt]
 C_1\cx_{C_0}C_0\arrow[dr,"p_1"']\arrow[rr,"C_1\cx_{C_0}e"]&&C_1\cx_{C_0}C_1\arrow[ld,"\comp"]\\
   &C_1&
 \end{tikzcd}
 \end{equation}
 \end{minipage}
 \end{subequations}
\end{figure}

 \subsection{Axioms for $\gm$-operad}\label{PCat_Axioms}
 Let $\mathcal{E}$ be a cartesian category and $(\gm ,\mu,\eta)$ a cartesian monad.
  A $\gm $-multicategory $\pop$ can be described by objects and arrows of $\ambcat$
\begin{center} \begin{tikzcd}[column sep=small]
       & \pop_1\arrow[dl,"s"']\arrow[dr,"t"] &       \\
 \gm \pop_0 &                                            & \pop_0
 \end{tikzcd}
 \begin{tikzcd}
\gm \pop_1\cx_{\gm \pop_0} \pop_1 \arrow[r,"\comp"]& \pop_1\\
 \phantom{aaaaa}\pop_0 \arrow[r,"e"] &\pop_1
 \end{tikzcd}
  \end{center}
 where the pullback is taken along $\gm \pop_1\xrightarrow{t}\gm \pop_0\xleftarrow{s}\pop_1$, satisfying the following commutative diagrams:

 \begin{figure}[H]
\centering
 \begin{subequations}
\begin{minipage}[b]{0.48\linewidth}
  \begin{equation}\label{PCat_d1d1}
 \begin{tikzcd}[column sep=3em,row sep=2em]
 \gm \pop_1\cx_{\gm \pop_0}\pop_1\arrow[r,"p_1"]\arrow[dd,"\comp"']&\gm \pop_1\arrow[d,"\gm s"]\\
 &\gm ^2\pop_0 \arrow[d,"\mu_{\pop_0}"]\\
 \pop_1\arrow[r,"s"']&\gm \pop_0
 \end{tikzcd}
 \end{equation}
\end{minipage}
\quad
\begin{minipage}[b]{0.48\linewidth}
 \begin{equation}\label{PCat_d0d1}
  \begin{tikzcd}
 \gm \pop_1\cx_{\pop_0}\pop_1\arrow[d,"p_2"']\arrow[r,"\comp"]&\pop_1\arrow[d,"t"]\\
 \pop_1\arrow[r,"t"']&\pop_0
 \end{tikzcd}
 \end{equation}
\end{minipage}
 \end{subequations}
\end{figure}
 
\vspace{20pt}
 
  \begin{figure}[H]
\centering
 \begin{subequations}\label{PCat_topbottom}
\begin{minipage}[b]{0.48\linewidth}
 \begin{equation}\label{PCat_d1e}
 \begin{tikzcd}
 \pop_0\arrow[dr,"\eta_{\pop_0}"']\arrow[r,"e"]&\pop_1\arrow[d,"s"]\\
   &\gm \pop_0
 \end{tikzcd}
\end{equation}
\end{minipage}
\quad
\begin{minipage}[b]{0.48\linewidth}
 \begin{equation}\label{PCat_d0e}
  \begin{tikzcd}
 \pop_0\arrow[dr,"\id"']\arrow[r,"e"]&\pop_1\arrow[d,"t"]\\
   &\pop_0
 \end{tikzcd}
 \end{equation}
\end{minipage}
 \end{subequations}
\end{figure}

 \begin{equation}\label{dgm:PCat_Associativity}
  \begin{tikzcd}[column sep=60pt, row sep=40pt]
 (\gm ^2\pop_1\cx_{\gm ^2\pop_0}\gm \pop_1)\cx_{\gm \pop_0}\pop_1\arrow[r,"\gm \comp\cx_{\pop_0}\pop_1"]\arrow[d]&\gm \pop_1\cx_{\gm \pop_0}\pop_1\arrow[dd,"\comp"]\\
 \gm ^2\pop_1\cx_{\gm ^2\pop_0}(\gm \pop_1\cx_{\gm \pop_0}\pop_1)\arrow[d,"\mu_{\pop_1}\cx_{\mu{\pop_0}}\comp"']&\\
 \gm \pop_1\cx_{\gm \pop_0}\pop_1\arrow[r,"\comp"']& \pop_1
 \end{tikzcd}
 \end{equation}
 
 \vspace{20pt}

   \begin{figure}[ht]
\centering
 \begin{subequations}\label{dgm:PCat_Unit}
\begin{minipage}[b]{0.475\linewidth}
 \begin{equation}\label{dgm:PCat_Unita}
   \begin{tikzcd}[column sep=20pt, row sep=20pt]
 \gm \pop_0\cx_{\gm \pop_0}\pop_1\arrow[dr,"p_2"']\arrow[rr,"\gm e\cx_{\pop_0}\pop_1"]&&\gm \pop_1\cx_{\gm \pop_0}\pop_1\arrow[dl,"\comp"]\\
   &\pop_1&
 \end{tikzcd}
 \end{equation}
\end{minipage}
\quad
\begin{minipage}[b]{0.475\linewidth}
 \begin{equation}\label{dgm:PCat_Unitb}
   \begin{tikzcd}[column sep=20pt, row sep=20pt]
\pop_1\cx_{\pop_0}\pop_0\arrow[dr,"p_1"']\arrow[rr,"\pop_1\cx_{\pop_0}e"]& &\gm \pop_1\cx_{\gm \pop_0}\pop_1\arrow[dl,"\comp"]\\
   &\pop_1&
 \end{tikzcd}
 \end{equation}\end{minipage}
 \end{subequations}
\end{figure}

%---------------------------------------------------------------------------------------------------------------------------------------
%-------------------------------------------REFERENCES------------------------------------------------------------------------
%-------------------------------------------REFERENCES------------------------------------------------------------------------
%-------------------------------------------REFERENCES------------------------------------------------------------------------
%-------------------------------------------REFERENCES------------------------------------------------------------------------
%-------------------------------------------REFERENCES------------------------------------------------------------------------
%-------------------------------------------REFERENCES------------------------------------------------------------------------
%-------------------------------------------REFERENCES------------------------------------------------------------------------
%---------------------------------------------------------------------------------------------------------------------------------------

%Personal information:
\noindent \textsc{Departament de Matem\`atiques,\\
Universitat Aut\`onoma de Barcelona,\\
08193 Bellaterra (Barcelona),
Spain}\\
\textit{E-mail adress:} \href{mailto:acebrian@mat.uab.cat}{acebrian@mat.uab.cat},\\

%----------------------------------------------------------------------------------------------------------------------------
%-----------------------------------------------END-----------------------------------------------------------------------
%-----------------------------------------------END-----------------------------------------------------------------------
%-----------------------------------------------END-----------------------------------------------------------------------
%-----------------------------------------------END-----------------------------------------------------------------------
%-----------------------------------------------END-----------------------------------------------------------------------
%-----------------------------------------------END-----------------------------------------------------------------------
%-----------------------------------------------END-----------------------------------------------------------------------
%----------------------------------------------------------------------------------------------------------------------------

 \end{document}